\documentclass[letter,11pt]{article}
\usepackage{amsthm}

\usepackage[mathlines]{lineno}
\usepackage{amsmath}
 \usepackage{graphicx}
 \usepackage[framemethod=tikz]{mdframed}
 \usepackage{algorithm}
 \usepackage{algpseudocode}
 \usepackage{epsfig,epstopdf}
 \usepackage{mathrsfs}
 \usepackage{bbm}
 \usepackage{seqsplit}
 \usepackage{caption}   
 \usepackage{float} 
 \usepackage{bigstrut}
 \usepackage{authblk}
  \usepackage{tcolorbox}
  
  \ifdefined\submit
 \usepackage[sort&compress,numbers]{natbib}
 
 \else
  \usepackage[numbers]{natbib}
 \fi

\ifdefined\submit
\usepackage[breaklinks=true,pdfstartview=FitH]{hyperref}
\else
\usepackage[breaklinks=true,pdfstartview=FitH,pagebackref=true]{hyperref}
\usepackage{footnotebackref}
\fi 
 \usepackage{tablefootnote}
 \usepackage{color}
\usepackage{subcaption}

\ifdefined\submit
\else 
\AtBeginDocument{
	\setlength{\abovedisplayskip}{4pt}
	\setlength{\belowdisplayskip}{4pt}
	\setlength{\abovedisplayshortskip}{2pt}
	\setlength{\belowdisplayshortskip}{2pt}
	\setlength{\jot}{1pt}
}

\allowdisplaybreaks
\fi

\usepackage{mathtools}
\usepackage{amssymb}
\usepackage[shortlabels]{enumitem}
\usepackage[T1]{fontenc}
\usepackage{multirow}
\usepackage{booktabs}
\usepackage{wrapfig}
\usepackage{bm}
\usepackage{soul}

\newcommand{\toweak}{\rightharpoonup}
\newcommand{\toset}{\rightrightarrows}
\def \rla{\right\rangle}
\def \lla{\left\langle}

 \newcommand{\ds}{\displaystyle}

\newcommand{\D}{\mathbf{D}}

\newcommand{\F}{\mathbf{F}}
\newcommand{\G}{\mathbf{G}}
\renewcommand{\H}{\mathcal{H}}

\newcommand{\K}{\mathbf{K}}

\newcommand{\R}{\mathbf{R}}

\renewcommand{\a}{\mathbf{a}}
\renewcommand{\b}{\mathbf{b}}

\renewcommand{\u}{\mathbf{u}}
\renewcommand{\v}{\mathbf{v}}
\newcommand{\s}{\mathbf{s}}
\newcommand{\w}{\mathbf{w}}
\newcommand{\x}{\mathbf{x}}
\newcommand{\y}{\mathbf{y}}
\newcommand{\z}{\mathbf{z}}

\newcommand{\yosida}[1]{{}^{#1} \hspace{-0.2em}}

\newcommand{\inner}[2]{\lla #1, #2 \rla}
\newcommand{\norm}[1]{{\left\|{#1}\right\|}}

\newcommand{\dr}[2]{T_{#1,#2}}

\newcommand{\jhsolved}[1]{}
\newcommand{\memosolved}[2]{}
\newcommand{\alert}[1]{#1}
\newcommand{\alertt}[1]{#1}
\newcommand{\SfA}[3]{S_{#1,#2}(#3; \gamma)}

\renewcommand{\Re}{{\rm I}\! {\rm R}}

\DeclareMathOperator*{\argmin}{arg\,min}

\DeclareMathOperator*{\Fix}{Fix}

\DeclareMathOperator*{\dist}{dist}

\DeclareMathOperator*{\sri}{sri}
\DeclareMathOperator*{\interior}{int}

\DeclareMathOperator*{\ran}{ran}
\DeclareMathOperator*{\dom}{dom}
\DeclareMathOperator*{\gra}{gra}
\DeclareMathOperator*{\Id}{Id}
\DeclareMathOperator*{\bfId}{\bm{\mathrm{Id}}}

\DeclareMathOperator*{\zer}{zer}
\DeclareMathOperator{\sign}{sign}

\usepackage{ifthen}

\newcommand{\sigmaf}[1][]{%
  \sigma_{_{f%
  \ifthenelse{\equal{#1}{}}{}{{}_{\mathrm{#1}}}%
  }}%
}

\newcommand{\Lf}[1][]{%
  L_{_{f%
  \ifthenelse{\equal{#1}{}}{}{{}_{\mathrm{#1}}}%
  }}%
}

\newcommand{\rhof}[1][]{%
  \rho_{_{f%
  \ifthenelse{\equal{#1}{}}{}{{}_{\mathrm{#1}}}%
  }}%
}

\ifdefined\submit 
\else 
\usepackage[margin=1in]{geometry}
\fi 

\ifdefined\submit
  \newcommand{\smartqedmark}{\qed}
\else
  \newcommand{\smartqedmark}{\qedhere}
\fi
\usepackage[capitalize,nameinlink]{cleveref}
\crefname{assumption}{Assumption}{Assumptions}
\crefformat{equation}{\textup{#2(#1)#3}}
\crefrangeformat{equation}{\textup{#3(#1)#4--#5(#2)#6}}
\crefmultiformat{equation}{\textup{#2(#1)#3}}{ and \textup{#2(#1)#3}}
{, \textup{#2(#1)#3}}{, and \textup{#2(#1)#3}}
\crefrangemultiformat{equation}{\textup{#3(#1)#4--#5(#2)#6}}%
{ and \textup{#3(#1)#4--#5(#2)#6}}{, \textup{#3(#1)#4--#5(#2)#6}}{, and \textup{#3(#1)#4--#5(#2)#6}}

\Crefformat{equation}{#2Equation~\textup{(#1)}#3}
\Crefrangeformat{equation}{Equations~\textup{#3(#1)#4--#5(#2)#6}}
\Crefmultiformat{equation}{Equations~\textup{#2(#1)#3}}{ and \textup{#2(#1)#3}}
{, \textup{#2(#1)#3}}{, and \textup{#2(#1)#3}}
\Crefrangemultiformat{equation}{Equations~\textup{#3(#1)#4--#5(#2)#6}}%
{ and \textup{#3(#1)#4--#5(#2)#6}}{, \textup{#3(#1)#4--#5(#2)#6}}{, and \textup{#3(#1)#4--#5(#2)#6}}

\crefdefaultlabelformat{#2\textup{#1}#3}

 \numberwithin{equation}{section}

 \ifdefined\submit
 \newtheorem{assumption}{Assumption}
 \else 
\newtheorem{theorem}{Theorem}
\newtheorem{definition}{Definition}

\newtheorem{proposition}{Proposition}
\newtheorem{lemma}{Lemma}
\newtheorem{assumption}{Assumption}

\theoremstyle{definition}
\newtheorem{example}{Example}
\newtheorem{remark}{Remark}
 \fi 

\def\TheTitle{Douglas--Rachford for multioperator comonotone inclusions with applications to multiblock optimization}

\ifpdf
\hypersetup{
  pdftitle={\TheTitle},
  pdfauthor={Jan Harold Alcantara, Minh N. Dao, and Akiko Takeda}}
\fi

\ifdefined\submit 
\title{\TheTitle\thanks{Communicated by Aviv Gibali.}}
\titlerunning{Douglas--Rachford for multioperator comonotone inclusions}
\author{Jan Harold Alcantara \and  Minh N. Dao  \and  Akiko Takeda}
\institute{*Corresponding author: Jan Harold Alcantara \at
	 Center for Advanced Intelligence Project, RIKEN, Tokyo, Japan.  \\	\email{\url{janharold.alcantara@riken.jp}}\\ \text{} \\ 
	Minh N. Dao \at
	School of Science, RMIT University, Melbourne, VIC 3000, Australia. \\ \email{\url{minh.dao@rmit.edu.au}}\\ \text{} \\ 
    Akiko Takeda \at Department of Mathematical Informatics, Graduate School of Information Science and Technology, University of Tokyo,
Tokyo, Japan, and Center for Advanced Intelligence Project, RIKEN, Tokyo, Japan. \\
\email{\url{takeda@mist.i.u-tokyo.ac.jp}}
	}
\date{}
\else 
 \title{\TheTitle}
 \date{\today}
\author{Jan Harold Alcantara\thanks{\url{janharold.alcantara@riken.jp}.  Center for Advanced Intelligence Project, RIKEN, Tokyo, Japan.} 
\qquad Minh N. Dao\thanks{\url{minh.dao@rmit.edu.au}. School of Science, RMIT University, Melbourne, VIC 3000, Australia. }
\qquad Akiko Takeda\thanks{\url{takeda@mist.i.u-tokyo.ac.jp}.
Department of Mathematical Informatics, Graduate School of Information Science and Technology, University of Tokyo,
Tokyo, Japan, and Center for Advanced Intelligence Project, RIKEN, Tokyo, Japan.}
}
\fi 

\begin{document}
\maketitle 

\begin{abstract}
     We study the doubly relaxed Douglas--Rachford (DR) algorithm for solving a multioperator inclusion problem involving the sum of maximally comonotone operators. To address such problems, we adopt a product space reformulation that accommodates nonconvex-valued operators, which is essential when dealing with weakly comonotone mappings. We establish the convergence of the doubly relaxed DR algorithm under comonotonicity assumptions, subject to suitable conditions on the algorithm parameters and the comonotonicity moduli of the operators. Our analysis relies on the Attouch--Th\'{e}ra duality framework, which enables the study of convergence through the corresponding dual inclusion problem. As an application, we derive a multiblock ADMM-type algorithm for structured convex and nonconvex optimization problems by applying the doubly relaxed DR algorithm to the operator inclusion formulation of the KKT system. The resulting method extends the classical duality between the DR algorithm and the alternating direction method of multipliers from the convex two-block case to multiblock and nonconvex settings. Moreover, we establish convergence guarantees in both the fully convex and strongly convex-weakly convex regimes. \\

\ifdefined\submit
   	
\else 
\fi 
\end{abstract}

\section{Introduction}

Consider the multioperator inclusion problem
\begin{equation}
    \text{Find~}x\in \H ~\text{such that } 0\in A_1 (x) + A_2(x) + \cdots + A_m (x)
    \label{eq:inclusion}
\end{equation}
where $A_1,A_2,\dots,A_m:\H \toset \H$ are set-valued operators on a real Hilbert space $\H$. In this paper, we study the convergence of the \textit{Douglas--Rachford (DR) algorithm} based on the product space reformulation proposed in \cite{AlcantaraTakeda2025} for solving \eqref{eq:inclusion} when the operators $A_1,A_2,\dots, A_m$ are \textit{maximally comonotone} (see \cref{defn:genmonotone}).

\subsection{Motivation}

The DR algorithm for comonotone operators is of particular interest because such operators naturally arise in the convergence analysis of the alternating direction method of multipliers (ADMM) for linearly constrained multiblock problems of the form  
\begin{equation}
    \min_{u_i \in \mathcal{H}_i} \, f_1(u_1) + \cdots + f_m(u_m) \quad \text{s.t.} \quad \sum_{i=1}^m L_i u_i = b,
    \label{eq:multiblock}
\end{equation}
where $b \in \mathcal{H}$, each $f_i : \mathcal{H}_i \to (-\infty, +\infty]$ is a proper closed function, and $L_i : \mathcal{H}_i \to \mathcal{H}$ is a bounded linear operator. This connection becomes especially relevant when the functions $f_i$ exhibit a mix of strong and weak convexity, as will be elaborated in detail in \cref{sec:admm}.  

This correspondence was observed by Bartz, Campoy, and Phan~\cite{Bartz2022AdaptiveADMM} in the two-block case:
\begin{equation}
    \min_{u_i \in \mathcal{H}_i} \, f_1(u_1) + f_2(u_2) \quad \text{s.t.} \quad L_1 u_1 - u_2 = 0,
    \label{eq:twoblock}
\end{equation}
where $f_1$ is strongly convex and $f_2$ is weakly convex. Under suitable assumptions, critical points of the Lagrangian associated with \eqref{eq:twoblock} are indeed solutions of the original problem and correspond to solutions of the inclusion problem \cite[Lemma 4.3 and Proposition 4.1]{Bartz2022AdaptiveADMM}
\begin{equation}
    \text{Find } x \in \mathcal{H} \text{ such that } 0 \in A_1(x) + A_2(x),
    \label{eq:inclusion2}
\end{equation}
where the operators $A_1$ and $A_2$ are defined by
\begin{equation}
    A_1 \coloneqq -L_1 \circ (\partial f_1)^{-1} \circ (-L_1^*), \quad A_2 \coloneqq (\widehat{\partial} f_2)^{-1},
    \label{eq:A1A2}
\end{equation}
with the subdifferential and inverse mappings defined as in \cref{sec:setvaluedoperators}. It is shown in \cite[Lemmas 4.6 and 4.7]{Bartz2022AdaptiveADMM} that $A_1$ and $A_2$ are strongly and weakly comonotone, respectively.  
Moreover, the DR algorithm applied to the inclusion~\eqref{eq:inclusion2} with operators $A_1$ and $A_2$ defined in~\eqref{eq:A1A2} is equivalent to the ADMM iteration for~\eqref{eq:twoblock} under suitable conditions~\cite[Lemmas 4.6 and 4.8]{Bartz2022AdaptiveADMM}. Leveraging this equivalence, the convergence of ADMM for~\eqref{eq:twoblock} with strongly convex $f_1$ and weakly convex $f_2$ was established via the DR framework, using the properties of strongly and weakly comonotone operators $A_1$ and $A_2$. This extends the classical duality between ADMM and DR in fully convex settings~\cite{Gabay1983MultiplierVI} to broader problems involving both strongly and weakly convex functions. 

However, extending this framework to the multiblock setting remains an open question, \alertt{which is our motivation for studying the DR algorithm for multioperator inclusion problems involving comonotone operators}.


\subsection{Two-operator reformulations}

A standard strategy for addressing \eqref{eq:inclusion} involves reformulating it as a two-operator inclusion over a lifted space. A traditional two-operator reformulation is \emph{Pierra's product space reformulation} \cite{Pierra1976,Pierra1984}: 
 \begin{equation}
        \text{Find}~\x \in \H^m ~\text{such that } 0\in \F_{\rm Pierra} (\x) + \G_{\rm Pierra}(\x),  
        \label{eq:pierra}
    \end{equation}
where $\x = (x_1,\dots,x_m)\in \H^m$, 
$\F_{\rm Pierra} (\x) \coloneqq  A_1(x_1) \times \cdots \times A_m(x_m)$ and $\G_{\rm Pierra} \coloneqq N_{\D_m} $, the normal cone\footnote{That is, $N_{\D_m} (\x) = \D_m^\perp = \{ \w = (w_1,\dots,w_m) \in \H^m : \sum_{i=1}^m w_i =0\}$ for $\x \in \D_m$ and $N_{\D_m}(\x) = \emptyset$ if $\x \notin \D_m$.} operator to $\D_m \coloneqq \{ (x_1,\dots,x_m)\in \H^m: x_1=\cdots = x_m\}$. However, in this reformulation, the second operator $\G_{\rm Pierra}$ is fixed as a maximally monotone operator. As a result, any nonmonotonicity present in the original operators $A_1, \dots, A_m$ is entirely absorbed into the first operator $\F_{\rm Pierra}$, which may be undesirable for convergence analysis. To the best of our knowledge, there are currently no existing results on the application of the DR algorithm to \eqref{eq:pierra} in the nonmonotone setting.

Based on \cite{Kruger1985}, another product space reformulation was proposed in \cite{Campoy2022}:
     \begin{equation}
        \text{Find}~\x \in \H^{m-1} ~\text{such that } 0\in \F_{\rm Campoy} (\x) + \G_{\rm Campoy}(\x),  
        \label{eq:campoy}
    \end{equation}
where 
\begin{align*}
    \F_{\rm Campoy} (\x) & \coloneqq A_1(x_1) \times \cdots \times A_{m-1}(x_{m-1}),\\
    \G_{\rm Campoy} (\x) & \coloneqq  \frac{1}{m-1} A_m(x_1) \times \cdots \times \frac{1}{m-1}A_{m}(x_{m-1}) + N_{\D_{m-1}}(\x) . 
\end{align*}
This reformulation reduces the ambient space dimension by one compared to \eqref{eq:pierra}, which could be more desirable in practice \cite{MalitskyTam2023}. Moreover, since $A_m$ is embedded in the second operator, the reformulation permits some nonmonotonicity in the second operator $\G_{\rm Campoy}$, in contrast to the formulation in \eqref{eq:pierra}. However, as noted in \cite{AlcantaraTakeda2025}, this reformulation is not suitable when the operator $A_m$ is not convex-valued. This last point is quite important and relevant, since weakly comonotone operators are not necessarily convex-valued in general; see \cref{ex:comonotone_nonconvexvalue}.

\subsection{Our approach and contributions}

\alert{The starting point of our approach is the product space reformulation proposed
in \cite{AlcantaraTakeda2025}, which introduces a simple but} crucial modification of Campoy's formulation \eqref{eq:campoy} to accommodate nonconvex-valued operators, such as comonotone ones. In particular, the flexible reformulation is given by 
     \begin{equation}
        \text{Find}~\x \in \H^{m-1} ~\text{such that } 0\in \F (\x) + \G(\x),  
        \label{eq:alcantara_takeda}
    \end{equation}
where 
    \begin{align}
        \F (\x) & \coloneqq \F_{\rm Campoy} (\x) = A_1(x_1) \times \cdots \times A_{m-1}(x_{m-1}), \label{eq:F} \\
        \G (\x) & \coloneqq  \K (\x) + N_{\D_{m-1}}(\x) ,\label{eq:G}
    \end{align}
with 
    \begin{equation*}
        \K (\x ) \coloneqq \begin{cases}
            \{\frac{1}{m-1}( v, \ldots, v): v\in A_m(x_1) \} & \text{if}~\x =(x_1,\dots,x_{m-1})\in \D_{m-1} \\
            \emptyset & \text{otherwise}.
        \end{cases}
    \end{equation*}
Indeed,
    \begin{equation}
        0\in \F(\x) + \G(\x) \quad \Longleftrightarrow \quad \exists x\in \H ~\text{s.t.}~\x = (x,\dots, x)~\text{and}~0\in A_1(x)+\cdots+ A_m(x).
        \label{eq:zeros_equivalence}
    \end{equation}
With this product space reformulation, we are able to apply the algorithm proposed in \cite{DaoPhan2019} for two-operator inclusion problems, which we refer to as the \emph{doubly relaxed Douglas--Rachford algorithm}, to the reformulated problem \eqref{eq:alcantara_takeda}. The resulting algorithm is detailed in \cref{alg:dr} in \cref{sec:dr_comonotone}. 

Our main contributions are threefold:
\begin{enumerate}[(I)]
    \item We revisit the two-operator setting through the Attouch--Th\'era duality framework \cite{AttouchThera1996}, which enables the analysis of the doubly relaxed DR algorithm on the dual inclusion. \alert{While Attouch--Th\'era duality itself is classical, to the best of our
    	knowledge, its use to transfer convergence results between the generalized
    	monotone and comonotone settings had not previously been exploited. In
    	particular, these two settings were analyzed separately in
    	\cite{Bartz2020Demiclosedness,Bartz2021Conical}. By exploiting this connection,
    	we are able to address the open question posed in
    	\cite[Remark~4.2]{Bartz2020Demiclosedness} concerning the strong convergence
    	of the shadow sequence.} Moreover, we sharpen \cite[Theorem~4.2]{Bartz2020Demiclosedness} by proving weak convergence of the shadow sequence under strictly weaker assumptions.
    
    \item We establish the convergence of the doubly relaxed DR algorithm for multioperator inclusion problems where each $A_i$ is maximally $\sigma_i$-comonotone, with $\sigma_i \geq 0$ for $i = 1, \dots, m{-}1$ and $\sigma_m \leq 0$. \alert{A direct application of the two-operator DR theory to the
    	product space reformulation already yields a convergence result. We go beyond
    	this direct argument by developing a refined analysis that establishes the
    	same convergence conclusions under less restrictive conditions on the
    	comonotonicity moduli; see \cref{thm:adaptiveDR_combined}. To the best of our
    	knowledge, this provides the first theoretically justified multioperator
    	splitting method for \eqref{eq:inclusion} under the comonotonicity
    	assumptions we consider.}
    
    \item We \alert{formally} derive a multiblock ADMM algorithm with convergence guarantees for solving problem~\eqref{eq:multiblock}.
    \begin{enumerate}[(a)]
        \item This result generalizes to the multiblock case the well-known duality between the DR and ADMM algorithms for two-block problems in the convex setting, as well as its extension to the strongly convex-weakly convex setting in \cite{Bartz2022AdaptiveADMM}. Notably, the derived multiblock ADMM algorithm preserves this duality structure even when the involved functions are nonconvex. In particular, the algorithm corresponds to the DR method applied to the multioperator inclusion associated with the (nonsmooth) KKT system of \eqref{eq:multiblock}.
        
        \item \alert{Using} our main convergence result (\cref{thm:adaptiveDR_combined}), the proposed multiblock ADMM algorithm enjoys theoretical convergence guarantees in both the fully convex case (i.e., each $f_i$ is convex) and the strongly convex-weakly convex setting (i.e., $f_i$ is strongly convex for $i = 1, \dots, m{-}1$, and $f_m$ is weakly convex).
    \end{enumerate}
\end{enumerate}

\subsection{Organization of the paper}

\Cref{sec:prelim} reviews key definitions and results on generalized monotone and comonotone operators, as well as existing convergence results for the doubly relaxed DR algorithm in the two-operator setting. In \cref{sec:attouchthera}, we revisit the doubly relaxed DR algorithm through the lens of the Attouch--Th\'era duality framework. This perspective enables us to leverage convergence results established for generalized monotone operators to obtain \emph{improved} convergence results for comonotone operators. As a key application, we resolve the open problem posed in \cite[Remark 4.2]{Bartz2020Demiclosedness} regarding the strong convergence of the shadow sequence. 

Using the refined convergence theory from \cref{sec:attouchthera}, we establish convergence results for the multioperator inclusion problem reformulated as the two-operator problem \eqref{eq:alcantara_takeda} in \cref{sec:directapplication}. A more delicate analysis is presented in \cref{sec:refinedanalysis}, where we again exploit the Attouch--Th\'era duality framework to relax the conditions on the comonotonicity moduli, particularly in the case where $A_m$ is weakly comonotone. 

In \cref{sec:admm}, we apply our main results to the multiblock optimization problem \eqref{eq:multiblock}, which involves strongly convex functions and one weakly convex term, and derive a convergent ADMM-type algorithm. Finally, concluding remarks and directions for future work are provided in \cref{sec:conclusion}.

\section{Preliminaries}
\label{sec:prelim}

Throughout this paper, $\H$ denotes a real Hilbert space endowed with the inner product $\lla \cdot, \cdot \rla$ and induced norm $\|\cdot\|$. For any real numbers $\alpha,\beta\in \Re$ and any $x,y \in \H$, we recall the following identity:
    \begin{equation}
        \norm{\alpha x + \beta y}^2 = \alpha (\alpha +\beta) \norm{x}^2 + \beta (\alpha+\beta) \norm{y}^2 - \alpha\beta \norm{x - y}^2. \label{eq:identity_squarednorm}
    \end{equation}
When $\alpha+\beta \neq 0$, \eqref{eq:identity_squarednorm} is equivalent to 
    \begin{equation}
        \alpha \norm{x}^2 + \beta \norm{y}^2 = \frac{\alpha\beta}{\alpha + \beta}\norm{x - y}^2 + \frac{1}{\alpha + \beta} \norm{\alpha x + \beta y}^2. \label{eq:identity_squarednorm2}
    \end{equation}
The following identity is an important tool in our convergence analysis.  
\begin{lemma}
\label{lemma:squaredsum_to_sumsquared}
For any nonzero real numbers $\theta_1,\dots,\theta_n$ with $\sum_{i=1}^n \frac{1}{\theta_i}=1$ and for any $u_1,\dots,u_n\in \H$, it holds that 
    \begin{equation}
        \norm{\sum_{i=1}^n u_i}^2 = \sum_{i=1}^n \theta_i\norm{u_i}^2 - \sum_{1\leq i<j\leq n} \frac{1}{\theta_i \theta_j} \norm{\theta_i u_i-\theta_j u_j}^2. \label{eq:squaredsum_to_sumsquared}
    \end{equation}
\end{lemma}
\begin{proof}
This directly follows from \cite[Lemma 2.14]{Bauschke2017}.
\smartqedmark \end{proof}

A sequence $\{x^k\}$ is said to be \emph{Fej\'{e}r monotone} with respect to a nonempty subset $S\subseteq \H$ if
for any $z\in S$, $\norm{x^{k+1}-z } \leq \norm{x^k - z}$  for all $k\in \mathbb{N}$. 
We use the standard notations $\to$ and $\toweak$ to denote strong and weak convergence, respectively.

\subsection{Set-valued operators}
\label{sec:setvaluedoperators}

A set-valued operator $A:\H\toset \H$ maps each point $x\in \H$ to a subset $A(x)$ of $\H$, which is not necessarily nonempty. The direct image of a subset $D\subseteq \H$ is given by $A(D) \coloneqq \bigcup _{x\in D}A(x)$. The \emph{domain} and \emph{range}  of $A$ are given respectively by $   \dom (A)   \coloneqq \{ x\in \H: A(x) \neq \emptyset\}$ and $ \ran (A)  \coloneqq \{ y\in \H : y\in A(x) ~\text{for some }x\in \H\}. $
The \emph{graph} of $A$ is the subset of $\H\times \H$ given by 
$\gra (A) \coloneqq \{ (x,y) \in \H \times \H : y\in A(x)\}. $
The \emph{inverse} of $A$, denoted by $A^{-1}$, is the set-valued operator whose graph is given by $\gra (A^{-1}) = \{ (y,x)\in \H\times \H: (x,y) \in \gra (A)\}. $
The \emph{zeros} and \emph{fixed points} of $A$ are given respectively by $\zer (A)  \coloneqq A^{-1}(0) = \{ x: 0\in A(x)\}, $ and $\Fix (A) \coloneqq \{ x\in \H : x\in A(x)\}.$
The \emph{resolvent} of $A:\H\toset\H$ with parameter $\gamma >0$, denoted by $J_{\gamma A}:\H\toset\H$, is defined by $ J_{\gamma A} \coloneqq (\Id+\gamma A)^{-1},$
where $\Id:\H\to\H$ is the identity operator $\Id(x)=x$, and the \emph{Yosida approximation of $A$} with parameter $\gamma$ is $\yosida{\gamma} A \coloneqq \frac{1}{\gamma}(\Id - J_{\gamma A}).$
The \emph{$\lambda$-relaxed resolvent} of $A$ with parameter $\gamma>0$ is defined by 
$J_{\gamma A}^{\lambda} \coloneqq \lambda
J_{\gamma A} + (1-\lambda)\Id .$
The following identity gives the relationship between the resolvent of $A$ and its inverse $A^{-1}$ (see \cite[Proposition 23.7(ii)]{Bauschke2017}):
    \begin{equation}
        \yosida{\gamma}A = \frac{1}{\gamma}\left( \Id - J_{\gamma A} \right) = J_{\frac{1}{\gamma} A^{-1}} \circ \left( \frac{1}{\gamma}\Id \right) .
        \label{eq:resolvent_inverse}
    \end{equation}

\subsection{Generalized monotone and comonotone operators}

\begin{definition}
\label{defn:genmonotone}
    Let $A:\H\toset \H$ and let $\sigma\in \Re$. 
    \begin{enumerate}[(i)]
        \item $A$ is \emph{$\sigma$-monotone} if 
        \[\lla x - y, u - v \rla \geq \sigma \norm{x-y}^2 \quad \forall (x,u),(y,v)\in \gra(A) .\]
        The case $\sigma=0$ corresponds to (standard) \emph{monotonicity}, $\sigma>0$ to \emph{strong monotonicity}, and $\sigma<0$ to \emph{weak monotonicity}.  
        \item $A$ is \emph{$\sigma$-comonotone} if 
        \[\lla x - y, u - v \rla \geq \sigma \norm{u-v}^2 \quad \forall (x,u),(y,v)\in \gra(A) .\]
         The case $\sigma=0$ corresponds to (standard) \emph{monotonicity}, $\sigma>0$ to \emph{strong comonotonicity}, and $\sigma<0$ to \emph{weak comonotonicity}.
        \item $A$ is \emph{maximally $\sigma$-monotone} if $A$ is $\sigma$-monotone and there is no $\sigma$-monotone operator whose graph properly contains $\gra (A)$; i.e., for every $(x,u)\in \H\times \H$,
        \[(x,u)\in \gra(A) \quad \Longleftrightarrow \quad \lla x - y, u - v \rla \geq \sigma \norm{x-y}^2 \quad \forall (y,v)\in \gra(A)  .\]
        When $\sigma=0$, we say that $A$ is \emph{maximally monotone}. 
        \item $A$ is \emph{maximally $\sigma$-comonotone} if $A$ is $\sigma$-comonotone and there is no $\sigma$-comonotone operator whose graph properly contains $\gra (A)$; i.e., for every $(x,u)\in \H\times \H$,
        \[(x,u)\in \gra(A) \quad \Longleftrightarrow \quad \lla x - y, u - v \rla \geq \sigma \norm{u-v}^2 \quad \forall (y,v)\in \gra(A)  .\]
    \end{enumerate}
    
\end{definition}

We summarize some facts about maximally monotone operators.

\begin{lemma}
\label{lemma:maximalmonotone_properties}
    Let $A:\H\toset \H$ be a monotone operator, and let $\gamma >0$.
    \begin{enumerate}[(i)]
        \item $J_{\gamma A}$ is monotone and nonexpansive on $\dom (J_{\gamma A})$.
        \item $\dom (J_{\gamma A}) = \ran( \Id+\gamma A) =\H$ if and only if $A$ is maximally monotone.
        \item If $A$ is maximally monotone, then $A(x)$ is convex for any $x\in \H$.
        \item If $A$ and $B$ are maximally monotone operators such that $\interior (\dom (A)) \cap \dom (B) \neq \emptyset$, then $A+B$ is maximally monotone.
    \end{enumerate}
\end{lemma}
\begin{proof}
    Parts (i) and (ii) are from \cite[Proposition 23.8]{Bauschke2017}, while (iii) holds by \cite[Proposition 20.36]{Bauschke2017}. 
    Part (iv) follows from \cite[Theorem 1 and Theorem 2]{Rockafellar1970}.
\smartqedmark \end{proof}

We note that $\sigma$-monotone and $\sigma$-comonotone operators have close relationships with maximally monotone operators, which are summarized in the following lemmas. 
\begin{lemma}
\label{lemma:maximal_sigma_equivalent}
   Let $A:\H\toset \H$ and let $\sigma\in \Re$. 
   The following are equivalent:
   \begin{enumerate}[(i)]
       \item $A$ is $\sigma$-monotone.
       \item $A-\sigma \Id$ is  monotone.
       \item $A^{-1}$ is $\sigma$-comonotone.
   \end{enumerate}
Moreover, the following are equivalent:
   \begin{enumerate}[(i)]
       \item $A$ is maximally $\sigma$-monotone.
       \item $A-\sigma \Id$ is maximally monotone.
       \item $A^{-1}$ is maximally $\sigma$-comonotone.
   \end{enumerate}
\end{lemma}
\begin{proof}
    See \cite[Lemmas 2.6 and 2.8]{Bauschke2021}.
\smartqedmark \end{proof}

The following are some properties of resolvents of $\sigma$-monotone and $\sigma$-comonotone operators.

\begin{lemma}
\label{lemma:maximal_single-valued-resolvent-monotone}
    Let $A:\H \toset \H$ be $\sigma$-monotone, and let $\gamma>0$ such that $1+\gamma \sigma>0$. Then the following hold:
    \begin{enumerate}[(i)]
        \item $J_{\gamma A}$ is at most single-valued.
        \item $\dom (J_{\gamma A})= \H$ if and only if $A$ is maximally $\sigma$-monotone.
    \end{enumerate}
\end{lemma}
\begin{proof}
    See \cite[Proposition 3.4]{DaoPhan2019}.
\smartqedmark \end{proof}

\begin{lemma}
\label{lemma:maximal_single-valued-resolvent-comonotone}
    Let $A:\H \toset \H$ be $\sigma$-comonotone, and let $\gamma>0$ such that $\gamma+ \sigma>0$. Then the following hold:
    \begin{enumerate}[(i)]
        \item $J_{\gamma A}$ is at most single-valued.
        \item $\dom (J_{\gamma A})= \H$ if and only if $A$ is maximally $\sigma$-comonotone.
    \end{enumerate}
\end{lemma}
\begin{proof}
    See \cite[Proposition 3.7]{Bartz2021Conical}.
\smartqedmark \end{proof}

We also mention some remarks about the convex-valuedness of maximally $\sigma$-monotone and $\sigma$-comonotone operators. Note that from \cref{defn:genmonotone}(iii), we may write 
\[A(x) = \bigcap_{(y,v)\in \gra (A)} \left\lbrace
 u\in \H: \inner{x-y}{u-v}\geq \sigma\norm{x-y}^2\right\rbrace \]
 for any $x\in \dom (A)$. Hence, $A(x)$ is the intersection of closed convex sets, and therefore $A(x)$ is convex. Consequently, property (iii) in \cref{lemma:maximalmonotone_properties} holds for maximally $\sigma$-monotone operators. However, this is not true for maximally $\sigma$-comonotone operators. Indeed, from \cref{defn:genmonotone}(iv), we may write 
\[A(x) = \bigcap_{(y,v)\in \gra (A)} \left\lbrace
 u\in \H: \inner{x-y}{u-v}\geq \sigma\norm{u-v}^2\right\rbrace, \]
but the set  $\left\lbrace
 u\in \H: \inner{x-y}{u-v}\geq \sigma\norm{u-v}^2\right\rbrace$ is not necessarily convex when $\sigma<0$. A particular example is provided below.
 \begin{example}
    \label{ex:comonotone_nonconvexvalue}
    Let $f$ be the nonconvex function $f(x) = \sin (x)$. Note that $f(x) + \frac{1}{2}\norm{x}^2$ is a convex function, and therefore $f' + \Id$ is a maximally monotone operator. By \cref{lemma:maximal_sigma_equivalent}, it follows that 
    $A(x) \coloneqq (f')^{-1}(x) = \cos^{-1}(x) =\{ y : \cos (y) = x\}$ is a maximally $\sigma$-comonotone operator with $\sigma=-1$.  Moreover, $A(x)$ is clearly not convex for any $x\in \dom (A) = [-1,1]$. 
 \end{example}

\subsection{Subdifferentials}

A function $f: \mathcal{H} \to (-\infty, +\infty]$ is said to be \emph{proper} if $\dom (f) \coloneqq \{ x\in \H : f(x) < \infty\}$ is nonempty, and it is \emph{closed} if it is lower semicontinuous. Given a proper and closed function $f$ and a point $x\in \dom (f)$, the \emph{regular subdifferential} of $f$ is defined by
\[
\widehat{\partial} f(x) \coloneqq \left\lbrace z\in \H : 
 \liminf_{y \to x} \frac{f(y) - f(x) - \langle z, y - x \rangle}{\|y - x\|} \geq 0 \right\rbrace .
\]
 The \emph{limiting subdifferential} of $f$ at $x$ is given by 
\[ \partial f(x) \coloneqq \left\lbrace z\in \H : 
  \exists x^k \to x ~\text{and}~z^k\in \widehat{\partial} f(x^k)~\text{with}~f(x^k)\to f(x) ~\text{and}~z^k\to z \right\rbrace .  \]
If $x \notin \dom (f)$, both $\widehat{\partial} f(x)$ and $\partial f(x)$ are defined to be empty. When $f$ is convex, the regular and limiting subdifferentials coincide and reduce to the classical subdifferential from convex analysis. In this case,
\[
\widehat{\partial} f(x) = \partial f(x) = \left\{ x^* \in \H : \forall y \in \mathcal{H}, \; f(y) \geq f(x) + \langle x^*, y - x \rangle \right\}.
\]

\subsection{Doubly relaxed Douglas--Rachford algorithm for two-operator inclusions}
\label{sec:adaptiveDR}

Let $(\kappa, \gamma,\delta,\lambda,\mu)\in \Re^5_{++}$, and $A,B : \H \toset \H$ be set-valued operators. We recall from \cite{DaoPhan2019} the operator associated with $(A,B)$ and parameters $(\kappa, \gamma,\delta,\lambda,\mu)$, which we refer to as the \emph{doubly relaxed DR operator}, defined by 
    \begin{equation*}
        T_{A,B} \coloneqq (1-\kappa) \Id + \kappa J_{\delta B}^{\mu} J_{\gamma A}^{\lambda} ,
    \end{equation*}
where
 \begin{equation}
        (\lambda-1)(\mu-1)=1\quad \text{and} \quad \delta = \gamma (\lambda-1) .
        \label{eq:parameters_basic_requirement}
    \end{equation}
The \emph{doubly relaxed DR algorithm} is given by the iterations
    \begin{equation}
        x^{k+1} \in T_{A,B} (x^k), \quad \forall k\in \mathbb{N},
        \tag{DR}
        \label{eq:aDR_algorithm}
    \end{equation}
which can also be expressed  as 
    \begin{subequations} \label{eq:adr_stepbystep}
            \begin{align}
             y^k & \in  J_{\gamma A}(x^k) \label{eq:adr_stepbystep_y} \\
            z^k & \in J_{\delta  B}((1-\lambda)x^k + \lambda y^k) \label{eq:adr_stepbystep_z}\\
            x^{k+1} & = x^k + \kappa \mu (z^k - y^k). \label{eq:adr_stepbystep_x}
            \end{align}
    \end{subequations}
We refer to $\{y^k\}$ as the shadow sequence. When $J_{\gamma A}$ is single-valued, we have 
    \begin{equation}
        T_{A,B} = \Id +\kappa \mu (J_{\delta B}J_{\gamma A}^{\lambda} - J_{\gamma A}).
        \label{eq:T_AB_singlevalued}
    \end{equation}
We summarize the convergence properties of~\eqref{eq:aDR_algorithm} for generalized monotone inclusions.

\begin{lemma}[Doubly relaxed DR for generalized monotone inclusions]
\label{lemma:aDR_monotonecase}
    Suppose that $A$ and $B$ are maximally $\alpha$-monotone and maximally $\beta$-monotone, respectively, with $\alpha+\beta\geq 0$ and $\zer (A+B) \neq \emptyset$. Let $(\gamma,\delta,\lambda,\mu)\in \Re^2_{++} \times (1,+\infty)^2$ such that  
    \begin{equation*}
        \begin{cases}
            \delta ( 1 + 2\gamma\alpha)=\gamma, &\text{if}~\alpha+\beta = 0, \\
            (\gamma+\delta)^2 < 4\gamma\delta(1+\gamma\alpha)(1+\delta\beta) & \text{if}~\alpha+\beta>0.
        \end{cases}
    \end{equation*}
    Set $\kappa \in (0,\kappa^*)$ 
    where
    \begin{equation*}
        \kappa^* \coloneqq \begin{cases}
            1 & \text{if}~\alpha+\beta = 0, \\
            \frac{4\gamma\delta(1+\gamma\alpha)(1+\delta\beta) - (\gamma+\delta)^2}{2\gamma\delta(\gamma+\delta)(\alpha+\beta)} & \text{if}~\alpha+\beta > 0.
        \end{cases}
    \end{equation*}
    Let $\{(x^k,y^k,z^k)\}$ be generated by \eqref{eq:adr_stepbystep} for an arbitrary initial point $x^0 \in \H$. Then the following hold:
        \begin{enumerate}[(i)]
            \item There exists $\bar{x} \in \Fix (T_{A,B})$ with $J_{\gamma A}(\bar{x}) \in \zer (A+B)$ such that $x^k\toweak \bar{x}$;
            \item $\norm{x^k-x^{k+1}} = o(1/\sqrt{k})$ and $\norm{y^k-y^{k+1}} = o(1/\sqrt{k})$ as $k\to \infty$; 
            \item The shadow sequences $\{y^k \} $ and $\{z^k\}$ converge weakly to $J_{\gamma A}(\bar{x}) $; and
            \item If $\alpha + \beta >0$, $  0<\kappa <1 $, and either
            \begin{enumerate}[leftmargin=2.5em]
                \item \alertt{$ \kappa^*\geq 1$}; or 
                \item \alertt{$\gamma=\delta>0$, $\lambda=\mu=2$,}
            \end{enumerate}
            then the shadow sequences $\{y^k \} $ and $\{z^k\}$ converge strongly to $J_{\gamma A}(\bar{x}) $ and $\zer (A+B) = \{ J_{\gamma A}(\bar{x}) \}$. 
        \end{enumerate}
\end{lemma}
\begin{proof}
    Part (i) and the rate $\norm{x^k-x^{k+1}} = o(1/\sqrt{k})$ in (ii) directly follow from \cite[Theorem 5.7]{Bartz2021Conical}. To prove that $\norm{y^k-y^{k+1}}= o(1/\sqrt{k})$, note from the  proof of \cite[Theorem 5.7]{Bartz2021Conical} that $1+\gamma \alpha>0$ and so $J_{\gamma A}$ is Lipschitz continuous by \cite[Lemma 3.3]{DaoPhan2019}. Since $y^k = J_{\gamma A}(x^k)$ and $\norm{x^k-x^{k+1}}=o(1/\sqrt{k})$, the claimed rate follows.  
    The weak convergence of $\{y^k\}$ to $J_{\gamma A}(\bar{x}) $ holds by \cite[Theorem 4.1]{Bartz2020Demiclosedness}. Since $z^k-y^k\to 0$ by noting part (ii) and \eqref{eq:adr_stepbystep_x}, the weak convergence of $\{z^k\}$  to $J_{\gamma A}(\bar{x}) $ also holds. This proves part (iii).
    \alertt{Lastly, by \cite[Remark 5.8]{Bartz2021Conical}, the condition $\kappa^*\geq1$
    in (a) is equivalent to $1+2\gamma\alpha>0$ and $
    \mu\in[2-2\gamma\beta,2+2\gamma\alpha].$
    On the other hand, under (b), we have
   $
    \kappa < \kappa^*=1+\gamma\frac{\alpha\beta}{\alpha+\beta}.
    $
    Hence, in both cases, the assumptions of \cite[Theorem~4.5]{DaoPhan2019}
    are satisfied, which proves (iv).}
\smartqedmark \end{proof}
\begin{remark}
    For brevity, we mention without formal statements that convergence results also exist for comonotone inclusion problems. Specifically, Bartz et al.~\cite{Bartz2021Conical} introduced the concept of conical averagedness to establish weak convergence of the sequence $\{x^k\}$ to a fixed point of $T_{A,B}$. In addition, Bartz et al.~\cite{Bartz2020Demiclosedness} showed weak convergence of the shadow sequence $\{y^k\}$ to a zero of $A+B$. In \cref{sec:attouchthera}, we will derive similar convergence results and further establish some stronger properties.
\end{remark}

\section{Douglas--Rachford for two-operator comonotone inclusions via Attouch--Th\'{e}ra duality}
\label{sec:attouchthera}

Consider the two-operator problem 
    \begin{equation}
        \text{Find~} x \in \H ~\text{such that } 0\in A(x) + B(x),
        \tag{P}
        \label{eq:P}
    \end{equation}
where $A,B:\H \toset \H$ are maximally $\alpha$-comonotone and maximally $\beta$-comonotone operators on $\H$. We provide a simple convergence analysis of the doubly relaxed DR algorithm \eqref{eq:aDR_algorithm} for solving the comonotone inclusion problem \eqref{eq:P}, relying only on the established convergence results in \cref{lemma:aDR_monotonecase} for the \alertt{generalized} monotone case. To this end, we consider the \textit{Attouch--Th\'{e}ra dual} of \eqref{eq:P} \cite{AttouchThera1996}, which is given by the inclusion problem 
    \begin{equation*}
        \text{Find~} u \in \H ~\text{such that } 0\in -A^{-1}(-u) + B^{-1}(u).
    \end{equation*}
For simplicity, we denote
    \begin{equation}
        {A}'= -A^{-1}\circ (-\Id) \quad \text{and} \quad {B}'= B^{-1},
        \label{eq:AB_dual}
    \end{equation}
so that the dual problem can be written as 
    \begin{equation}
        \text{Find~} u \in \H ~\text{such that } 0\in  A'(u)+ B'(u).
        \tag{D}
        \label{eq:D}
    \end{equation}
Let us consider the doubly relaxed DR algorithm (see \eqref{eq:aDR_algorithm}) for solving the dual problem \eqref{eq:D} with parameters $(\kappa, \gamma',\delta',\lambda',\mu')$, given by the iterations 
    \begin{equation}
        u^{k+1} \in T_{{A}',{B}'}(u^k) , \quad \forall k\in \mathbb{N},
        \tag{DR'}
        \label{eq:aDR_algorithm_dual}
    \end{equation}
where $T_{{A}',{B}'} = (1-\kappa) \Id + \kappa J_{\delta' {B}'}^{\mu'}J_{\gamma' {A}'}^{\lambda'}.$
Equivalently, this can be written as
    \begin{subequations} \label{eq:adr_stepbystep_dual}
            \begin{align}
             v^k & \in  J_{\gamma' A'}(u^k) \label{eq:adr_stepbystep_dual_v}\\
            w^k & \in  J_{\delta'  B'}((1-\lambda')u^k + \lambda' v^k) \label{eq:adr_stepbystep_dual_w}\\
            u^{k+1} & = u^k + \kappa \mu' (w^k - v^k). \label{eq:adr_stepbystep_dual_u}
            \end{align}
    \end{subequations}
 
\alert{The following lemma provides the precise relationship between the iterates of the doubly relaxed DR algorithms \eqref{eq:aDR_algorithm} and \eqref{eq:aDR_algorithm_dual} for the primal and dual problems. Although its derivation relies on classical resolvent identities, we include the details for completeness and to make explicit the corresponding transformation of the iterates and algorithm parameters.}

\begin{lemma}
\label{lemma:adr_primal_dual}
Suppose $J_{\gamma A}$ and $J_{\delta B}$ are single-valued with full domain. Let $(\kappa, \gamma,\delta,\lambda,\mu)\in \Re^5_{++}$ such that \eqref{eq:parameters_basic_requirement} holds.
Let $(\gamma',\delta',\lambda',\mu')\in \Re^4_{++}$ be given by 
    \begin{equation}
        \gamma' \coloneqq \frac{1}{\gamma}, \quad \delta' \coloneqq \frac{1}{\delta}, \quad \lambda' \coloneqq \frac{\lambda\gamma}{\delta}, \quad \mu' \coloneqq \frac{\mu\delta}{\gamma}.
        \label{eq:dual_parameters}
    \end{equation}
 Let $\{(x^k,y^k,z^k)\}$ be the sequence generated by \eqref{eq:adr_stepbystep} with initial point $x^0$, and $\{(u^k,v^k,w^k)\}$ be the sequence generated by \eqref{eq:adr_stepbystep_dual} with initial point $u^0 = -\frac{x^0}{\gamma}$. Then the following hold:
 \begin{enumerate}[(i)]
     \item We have the identities 
         \begin{equation}
        (\lambda'-1)(\mu'-1) = 1 \quad \text{and} \quad \delta' = \gamma' (\lambda' - 1); 
        \label{eq:dual_parameters_identity}
    \end{equation}
    \item For all $k\in \mathbb{N}$, 
        \begin{subequations}
        \label{eq:primaldual}
            \begin{align}
                x^k & = -\frac{u^k}{\gamma'} \label{eq:primaldual_x} \\ 
                y^k & = \frac{v^k-u^k}{\gamma'} \label{eq:primaldual_y}\\
                z^k & = \frac{(1-\lambda')u^k + \lambda' v^k - w^k}{\delta'}; \label{eq:primaldual_z}
            \end{align}
        \end{subequations}
    \item $\gamma \Fix (T_{A',B'}) = -\Fix (T_{A,B})$;
    \item  $J_{\gamma' A'}(\Fix (T_{A',B'})) = \zer (A'+B')$ and $J_{\gamma A}(\Fix (T_{A,B})) = \zer (A+B)$;
    \item $\zer (A'+B') = - \yosida{\gamma} A \left( \Fix (T_{A,B}) \right)$ and $\zer (A+B)= - \yosida{\gamma'}A' \left( \Fix (T_{A',B'}) \right)$.
 \end{enumerate}
\end{lemma}
\begin{proof}
We have from \eqref{eq:parameters_basic_requirement} that $\frac{1}{\gamma} = \frac{\lambda}{\delta} - \frac{1}{\delta}$, and so
\begin{align*}
    \gamma'(\lambda'-1) = \frac{1}{\gamma} \left( \frac{\lambda\gamma}{\delta} - 1\right) = \frac{\lambda}{\delta} - \frac{1}{\gamma} = \frac{1}{\delta} = \delta'.
\end{align*}
We also have from  \eqref{eq:parameters_basic_requirement} that $\frac{\gamma}{\delta} = \frac{1}{\lambda-1}$ and $(\mu-1)\delta = \gamma$. Thus, 
\begin{align*}
    (\lambda'-1)(\mu'-1) & = \left( \frac{\lambda \gamma}{\delta} - 1\right) \left( \frac{\mu\delta}{\gamma}-1\right)= \frac{1}{(\lambda-1)(\mu-1)}.
\end{align*}
Together with \eqref{eq:parameters_basic_requirement}, we obtain \eqref{eq:dual_parameters_identity}. To prove (ii), first note that $x^0 = -\frac{u^0}{\gamma'}$ by hypothesis, and therefore \eqref{eq:primaldual_x} holds when $k=0$. Meanwhile, we have from the formulas $A' = -A^{-1} \circ (-\Id)$ and \eqref{eq:resolvent_inverse} that
    \begin{equation}
        J_{\gamma' A'} (u) = u + \gamma' J_{\frac{1}{\gamma'}A}  \left( -\frac{u}{\gamma'}\right) \quad \forall u\in \H.
        \label{eq:resolvent_Aprime}
    \end{equation}
Hence, 
    \begin{align}
        v^0 & \overset{\eqref{eq:adr_stepbystep_dual_v}}{=} J_{\gamma' A'}(u^0) \overset{\eqref{eq:resolvent_inverse}}{=} u^0 + \gamma' J_{\frac{1}{\gamma'}A}\left( -\frac{u^0}{\gamma'}\right) = u^0 + \gamma' y^0,
        \label{eq:v0}
    \end{align}
where the last equality holds by the definition of $x^0$, \eqref{eq:adr_stepbystep_y}, and \eqref{eq:dual_parameters}. It follows that \eqref{eq:primaldual_y} holds for $k=0$. We also have 
    \begin{align}
        w^0  = J_{\delta' B'}((1-\lambda')u^0 + \lambda' v^0) &\overset{\eqref{eq:resolvent_inverse}}{=} (1-\lambda')u^0 + \lambda' v^0 - \delta' J_{\frac{1}{\delta'}B}\left(\frac{(1-\lambda')u^0 + \lambda' v^0 }{\delta'} \right)\notag
    \end{align}
and therefore 
    \begin{align}
        \frac{(1-\lambda')u^0 + \lambda' v^0  - w^0}{\delta'} &=J_{\delta B}\left(\frac{\gamma' (\lambda'-1)x^0 + \lambda' v^0 }{\delta'} \right) \notag \\
         &= J_{\delta B}\left(\frac{\delta' x^0 + \lambda' (-\gamma' x^0+\gamma' y^0) }{\delta'} \right) \notag \\
        & = J_{\delta B} \left( \left( 1-\frac{\lambda'\gamma'}{\delta'}\right)x^0 + \frac{\lambda'\gamma'}{\delta'}y^0\right) \notag \\
        & \overset{\eqref{eq:dual_parameters}}{=} J_{\delta B} \left( \left( 1-\lambda\right)x^0 + \lambda y^0\right) \overset{\eqref{eq:adr_stepbystep_z}}{=} z^0, \notag 
        \end{align}
where we used $u^0 = -\gamma' x^0$ in the first and third equalities, and equations \eqref{eq:dual_parameters_identity} and \eqref{eq:v0} for the second. This proves that \eqref{eq:primaldual_z} holds for $k=0$. By induction, suppose \eqref{eq:primaldual} holds for some $k\geq 0$. Subtracting \eqref{eq:primaldual_y} from \eqref{eq:primaldual_z}, we get
    \[z^k - y^k = \frac{\gamma' (1-\lambda')u^k + \gamma' \lambda' v^k - \gamma' w^k- \delta' (v^k-u^k)}{\delta' \gamma' } \overset{\eqref{eq:dual_parameters_identity}}{=} \frac{1}{\delta'}(v^k-w^k).\]
Then
\begin{align*}
    x^{k+1} \overset{\eqref{eq:adr_stepbystep_x}}{=} x^k + \frac{\kappa\mu}{\delta'}(v^k-w^k) & \overset{\eqref{eq:primaldual_x}}{=} -\frac{u^k}{\gamma'} + \frac{\kappa\mu}{\delta'}(v^k-w^k) \\ 
    & \overset{\eqref{eq:adr_stepbystep_dual_u}}{=} -\frac{u^k}{\gamma'} + \frac{\mu }{\mu'\delta'
    }(u^k-u^{k+1}) \overset{\eqref{eq:dual_parameters}}{=} -\frac{u^{k+1}}{\gamma'}.
\end{align*}
Thus, \eqref{eq:primaldual_x} holds when $k$ is replaced with $k+1$. The proofs for \eqref{eq:primaldual_y} and \eqref{eq:primaldual_z}  proceed in a similar fashion. 

To prove part (iii), let $u\in \Fix (T_{A',B'})$ and set $u^0\coloneqq u$. If $u^1\coloneqq T_{A',B'}(u^0)$, then $u^0=u^1$. It follows from \eqref{eq:primaldual_x} that $x^0=x^1 = T_{A,B}(x^0)$ where $x^0=-\frac{u^0}{\gamma'}=-\gamma u^0$. That is, $x^0=-\gamma u^0\in \Fix (T_{A,B})$ and therefore  $\gamma \Fix(T_{A',B'}) \subseteq -\Fix (T_{A,B})$. The other inclusion holds by reversing the arguments. Part (iv) holds by using \cite[Lemma 4.1(iii)]{DaoPhan2019}, noting that \eqref{eq:parameters_basic_requirement} and \eqref{eq:dual_parameters_identity} hold. Using parts (iii) and (iv) together with the formulas \eqref{eq:resolvent_Aprime} and \eqref{eq:resolvent_inverse}, we obtain part (v).
\smartqedmark \end{proof}

Switching $A$ and $B$ yields the following analogous result.
\begin{lemma}
\label{lemma:adr_primal_dual_switched}
Suppose that the hypotheses of \cref{lemma:adr_primal_dual} hold. Let $\{(x^k,y^k,z^k)\}$ be generated by 
            \begin{align*}
             y^k & =  J_{\delta B}(x^k) 
             \\
            z^k & = J_{\gamma  A}((1-\mu)x^k + \mu y^k) 
            \\
            x^{k+1} & = T_{B,A}(x^k) =  x^k + \kappa \lambda (z^k - y^k), 
            \end{align*}
with initial point $x^0$, and let $\{ (u^k,v^k,w^k)\}$ be generated by 
            \begin{align*}
             v^k & =  J_{\delta' B'}(u^k)
             \\
            w^k & = J_{\gamma'  A'}((1-\mu')u^k + \mu' v^k) 
            \\
            u^{k+1} & = T_{B',A'}(u^k) = u^k + \kappa \lambda' (w^k - v^k), 
            \end{align*}
with initial point $u^0 = \frac{x^0}{\delta}$. Then the following hold:
\begin{enumerate}[(i)]
    \item For all $k\in \mathbb{N}$, $x^k = \frac{u^k}{\delta'}$, $ y^k  = \frac{u^k-v^k}{\delta'}$, and $z^k  = \frac{w^k - (1-\mu')u^k - \mu' v^k}{\gamma'};$
    \item $\delta \Fix (T_{B',A'}) = \Fix (T_{B,A})$; 
    \item  $J_{\delta' B'}(\Fix (T_{B',A'})) = \zer (A'+B')$ and $J_{\delta B}(\Fix (T_{B,A})) = \zer (A+B)$; 
    \item $\zer (A'+B') =  \yosida{\delta} B \left( \Fix (T_{B,A}) \right)$ and $\zer (A+B)= \yosida{\delta'} B' \left( \Fix (T_{B',A'}) \right)$. 
\end{enumerate}
\end{lemma}

\cref{lemma:adr_primal_dual} and \cref{lemma:adr_primal_dual_switched} imply that any convergence result for the doubly relaxed DR algorithm applied to the dual problem \eqref{eq:D} can be directly translated into a corresponding convergence result for the primal problem \eqref{eq:P} (and vice versa) after an appropriate transformation of the parameters via \eqref{eq:dual_parameters}. As an application, we establish the convergence of the doubly relaxed DR algorithm for comonotone inclusion problems using the existing theory for monotone inclusions.

\begin{theorem}[Doubly relaxed DR for comonotone inclusions]
    \label{thm:aDR_comonotonecase}
    Suppose that $A$ and $B$ are maximally $\alpha$-comonotone and maximally $\beta$-comonotone, respectively, with $\alpha+\beta\geq 0$ and $\zer (A+B) \neq \emptyset$. Let $(\gamma,\delta,\lambda,\mu)\in \Re^2_{++} \times (1,+\infty)^2$ such that 
    \begin{equation}
        (\lambda-1)(\mu-1)=1\quad \text{and} \quad \delta = \gamma (\lambda-1) ,
        \label{eq:parameters_basic_requirement_comonotone}
    \end{equation}
    and 
    \begin{equation}
        \begin{cases}
            \delta = \gamma + 2\alpha, &\text{if}~\alpha+\beta = 0, \\
            (\gamma+\delta)^2 < 4(\gamma+\alpha)(\delta+\beta) & \text{if}~\alpha+\beta>0.
        \end{cases}
        \label{eq:parameters_additional_reqs_comonotone}
    \end{equation}
    Set $\kappa \in (0,\kappa^*)$ where
    \begin{equation}
        \label{eq:kappastar_comonotone}
        \kappa^* \coloneqq \begin{cases}
            1 & \text{if}~\alpha+\beta = 0, \\
            \frac{4(\gamma+\alpha)(\delta+\beta) - (\gamma+\delta)^2}{2(\gamma+\delta)(\alpha+\beta)} & \text{if}~\alpha+\beta > 0.
        \end{cases}
    \end{equation}
    Let $\{(x^k,y^k,z^k)\}$ be generated by \eqref{eq:adr_stepbystep} for an arbitrary initial point $x^0 \in \H$. Then the following hold:
        \begin{enumerate}[(i)]
            \item There exists $\bar{x} \in \Fix (T_{A,B})$ with $J_{\gamma A}(\bar{x}) \in \zer (A+B)$ such that $x^k\toweak \bar{x}$;
            \item $\norm{x^k-x^{k+1}} = o(1/\sqrt{k})$ and $\norm{y^k-y^{k+1}} = o(1/\sqrt{k})$ as $k\to \infty$; 
            \item The shadow sequences $\{y^k \} $ and $\{z^k\}$ converge weakly to $J_{\gamma A}(\bar{x}) $; and
            \item If $\alpha + \beta >0$, $  0<\kappa <1$ and
            either
            \begin{enumerate}[leftmargin=2.5em]
                \item \alertt{$\kappa^*\geq 1$}; or
                \item \alertt{$\gamma=\delta>0$, $\lambda=\mu=2$,}
            \end{enumerate}
            then the sequences $\{x^k-y^k \} $ and $\{x^k-z^k\}$ converge strongly to $\gamma \yosida{\gamma }A(\bar{x}) = \bar{x}-J_{\gamma A}(\bar{x}) $ and $\alertt{\gamma }\yosida{\gamma}A(\bar{x}) + \zer (A+B) =  \Fix (T_{A,B})$. 
        \end{enumerate}
\end{theorem}
\begin{proof}
    Let $A'$ and $B'$ be given by \eqref{eq:AB_dual}, and let $(\gamma',\delta',\lambda',\mu')\in \Re^4_{++}$ be given by \eqref{eq:dual_parameters}. We use \cref{lemma:aDR_monotonecase} to establish the convergence of \eqref{eq:aDR_algorithm_dual} for the inclusion problem \eqref{eq:D}. To this end, we first prove that the following conditions hold:
        \begin{enumerate}[(C1)]
            \item $A'$ and $B'$ are maximally $\alpha$-monotone and maximally $\beta$-monotone with $\alpha+\beta \geq 0$;
            \item $\zer (A'+B') \neq \emptyset$;
            \item $(\gamma',\delta',\lambda',\mu')\in \Re^2_{++}\times (1,+\infty)^2$ satisfies
        \begin{equation}
        (\lambda'-1)(\mu'-1)=1\quad \text{and} \quad \delta' = \gamma' (\lambda'-1) ,
        \label{eq:parameters_basic_requirement_dual}
    \end{equation}
    and 
    \begin{equation}
        \begin{cases}
            \delta' ( 1 + 2\gamma'\alpha)=\gamma', &\text{if}~\alpha+\beta = 0, \\
            (\gamma'+\delta')^2 < 4\gamma'\delta'(1+\gamma'\alpha)(1+\delta'\beta) & \text{if}~\alpha+\beta>0.
        \end{cases}
        \label{eq:parameters_additional_reqs_dual}
    \end{equation}
    \item $\kappa \in (0,\kappa^*)$ where
    \begin{equation}
        \label{eq:kappastar_dual}
        \kappa^* =\begin{cases}
            1 & \text{if}~\alpha+\beta = 0, \\
            \frac{4\gamma'\delta'(1+\gamma'\alpha)(1+\delta'\beta) - (\gamma'+\delta')^2}{2\gamma'\delta'(\gamma'+\delta')(\alpha+\beta)} & \text{if}~\alpha+\beta > 0.
        \end{cases}
    \end{equation}
            \end{enumerate}
Part (C1) directly follows from the definition of comonotonicity and the definition of the operators $A'$ and $B'$ in \eqref{eq:AB_dual}; see also \cref{lemma:maximal_sigma_equivalent}. To prove (C2), we observe from the second equation in \cref{lemma:adr_primal_dual}(iv) that $\zer (A+B)\neq \emptyset$ implies that $\Fix (T_{A,B})\neq \emptyset$. Consequently, from the first equality in \cref{lemma:adr_primal_dual}(v), we see that $\zer(A'+B')\neq \emptyset$. Meanwhile, equation \eqref{eq:parameters_basic_requirement_dual} holds by \cref{lemma:adr_primal_dual}(i). Since $\gamma = \frac{1}{\gamma'}$ and $\delta = \frac{1}{\delta'}$ by \eqref{eq:dual_parameters}, we derive from \eqref{eq:parameters_additional_reqs_comonotone} and \eqref{eq:kappastar_comonotone} that \eqref{eq:parameters_additional_reqs_dual} and \eqref{eq:kappastar_dual} hold, respectively. Hence, parts (C3) and (C4) hold.

Now we invoke \cref{lemma:aDR_monotonecase}, from where we know that for any $\{(u^k,v^k,w^k)\}$ generated by \eqref{eq:adr_stepbystep_dual} with initial point $u^0=-\frac{x^0}{\gamma}$, the following hold:
\begin{enumerate}[(i')]
  \item There exists $\bar{u} \in \Fix (T_{A',B'})$ with $J_{\gamma' A'}(\bar{u}) \in \zer (A'+B')$ such that $u^k\toweak \bar{u}$;
    \item $\norm{u^k-u^{k+1}} = o(1/\sqrt{k})$ and $\norm{v^k-v^{k+1}}= o(1/\sqrt{k})$ as $k\to \infty$; 
    \item The shadow sequences $\{v^k \} $ and $\{w^k\}$ converge weakly to $J_{\gamma' A'}(\bar{u}) $; and
    \item If $\alpha + \beta >0$, $\kappa\in (0,1)$ and either
            \begin{enumerate}[(a'),leftmargin=3.5em]
                \item $1 + 2 \gamma' \alpha > 0 $ and $\kappa^*\geq 1$; or
                \item $\gamma'=\delta'>0$,  $\lambda'=\mu'=2$, and $\kappa^* > \kappa$,
            \end{enumerate}
            then the shadow sequences $\{v^k \} $ and $\{w^k\}$ converge strongly to $J_{\gamma' A'}(\bar{u}) $ and $\zer (A'+B') = \{ J_{\gamma' A'}(\bar{u}) \}$. 
\end{enumerate}
Using \cref{lemma:adr_primal_dual}, we have from \eqref{eq:primaldual_x} and part (i') that $x^k\toweak \bar{x}$, where $\bar{x}\coloneqq -\gamma \bar{u}\in -\gamma \Fix (T_{A',B'}) = \Fix (T_{A,B})$ and the last equality holds by \cref{lemma:adr_primal_dual}(iii). Note that by \eqref{eq:resolvent_Aprime}, we have 
\begin{equation}
J_{\gamma' A'}(\bar{u}) = \bar{u} + \gamma' J_{\gamma A}\left(\bar{x}\right),
\label{eq:shadowA'}
\end{equation}
and therefore $J_{\gamma A}(\bar{x}) = -\yosida{\gamma'}A'(\bar{u})$. Since $\bar{u}\in \Fix (T_{A',B'})$, it follows from the second equality in \cref{lemma:adr_primal_dual}(v) that $J_{\gamma A}(\bar{x})\in \zer (A+B)$. This proves part (i) of the theorem. On the other hand, part (ii) immediately follows from (ii'),  \eqref{eq:primaldual_x}, \eqref{eq:primaldual_y}, and the triangle inequality. Using  \eqref{eq:shadowA'}, and the limits $u^k\toweak \bar{u}$ and $v^k \toweak J_{\gamma' A'}(\bar{u})$ from (i') and (iii'), respectively, we see from \eqref{eq:primaldual_y} that  $y^k\toweak J_{\gamma A}(\bar{x})$. Similarly, we get $z^k \toweak J_{\gamma A}(\bar{x})$ from \eqref{eq:primaldual_z} and noting \eqref{eq:parameters_basic_requirement_dual}. This completes the proof of part (iii).  Finally, note that the conditions (a') and (b') are equivalent to conditions (a) and (b) in part (iv) by using \eqref{eq:dual_parameters}. By (iv'), \eqref{eq:primaldual_y}, \eqref{eq:primaldual_z} and \eqref{eq:shadowA'}, we obtain the limits $x^k-y^k \to \bar{x}-J_{\gamma A}(\bar{x})$ and $x^k-z^k \to \bar{x}-J_{\gamma A}(\bar{x})$. Meanwhile, using $\zer (A'+B') = \{ J_{\gamma' A'}(\bar{u}) \}$ from (iv') together with \cref{lemma:adr_primal_dual}(iii) and the second equality in \cref{lemma:adr_primal_dual}(v) and \eqref{eq:shadowA'}, it is not difficult to derive  $\alertt{\gamma }\yosida{\gamma}A(\bar{x}) + \zer (A+B) =  \Fix (T_{A,B})$, thus completing the proof of (iv). With these, the proof of the theorem is finished.  
\smartqedmark \end{proof} 

\begin{remark}[Comparison with existing convergence results for comonotone inclusions]
Under conditions \eqref{eq:parameters_basic_requirement_comonotone}, \eqref{eq:parameters_additional_reqs_comonotone} and \eqref{eq:kappastar_comonotone}, the claims in \cref{thm:aDR_comonotonecase}(i) and (ii) coincide with those in \cite[Theorem 5.4]{Bartz2021Conical}. The weak convergence of the shadow sequence $\{y^k\}$, on the other hand, was already proved in \cite[Theorem 4.2]{Bartz2020Demiclosedness}, but under the \textit{additional condition} that 
$\gamma + \delta \leq \min \{ 2(\gamma + \alpha), 2(\delta+\beta)\} .$
 Hence, \cref{thm:aDR_comonotonecase}(iii) is a significant improvement over the existing result. Finally, the strong convergence result in \cref{thm:aDR_comonotonecase}(iv) is \emph{new to the literature.}
\end{remark}

\begin{remark}[On the strong convergence of the shadow sequence]
   In \cite[Remark 4.2]{Bartz2020Demiclosedness}, it was conjectured that the shadow sequence $\{y^k\}$ converges
   strongly whenever $\alpha+\beta>0$. \cref{thm:aDR_comonotonecase}(iv) shows that strong
   convergence of $\{y^k\}$ would in turn imply strong convergence of the
   fixed-point sequence $\{x^k\}$. Since, in general, Douglas--Rachford fixed-point
   sequences are only known to converge weakly as in \cref{thm:aDR_comonotonecase}(i), the condition $\alpha+\beta>0$ alone
   does not provide a general mechanism for establishing strong convergence of the
   shadow sequence.
\end{remark}

\section{Douglas--Rachford for multioperator comonotone inclusions}
\label{sec:dr_comonotone}

In this section, we provide the convergence analysis for the doubly relaxed DR algorithm for solving the two-operator reformulation \eqref{eq:alcantara_takeda} of the multioperator inclusion problem \eqref{eq:inclusion}. \cref{sec:directapplication} gives a direct proof via \cref{thm:aDR_comonotonecase}, and \cref{sec:refinedanalysis} provides a sharper analysis to establish the result under strictly weaker assumptions.

The algorithm is given by 
\begin{equation}
    \x^{k+1} \in T_{{\F},{\G}}(\x^k) , \quad \forall k\in \mathbb{N},
    \tag{\textbf{DR}}
    \label{eq:aDR_algorithm_bold}
\end{equation}
where $    T_{{\F},{\G}} \coloneqq (1-\kappa) \bfId + \kappa J_{\delta {\G}}^{\mu}J_{\gamma {\F}}^{\lambda},$
and the parameters $(\gamma,\delta,\lambda,\mu)\in \Re^2_{++} \times (1,+\infty)^2$ satisfy \eqref{eq:parameters_basic_requirement}.
As mentioned in \cref{sec:adaptiveDR}, this can be written as
\begin{equation} \label{eq:adr_stepbystep_bold}
    \begin{aligned}
        \y^k & \in  J_{\gamma \F}(\x^k) \\ 
        \z^k & \in  J_{\delta  \G}((1-\lambda)\x^k + \lambda \y^k) \\ 
        \x^{k+1} & = \x^k + \kappa \mu (\z^k - \y^k). 
    \end{aligned}
\end{equation}
We recall from \cite[Propositions 3.5 and 3.6]{AlcantaraTakeda2025} the following resolvent formulas:
\begin{align}
    J_{\gamma \F} (\x) & = J_{\gamma A_1}(x_1) \times \cdots \times J_{\gamma A_{m-1}}(x_{m-1}), \label{eq:Fresolvent}\\ 
    J_{\delta \G} (\x) & = \left\lbrace (u,\dots,u)\in \H^{m-1} : u\in J_{\frac{\delta}{m-1} A_m} \left( \frac{1}{m-1}\sum_{i=1}^{m-1} x_i\right) \right\rbrace .\label{eq:Gresolvent}
\end{align}

Using these formulas, the algorithm can also be expressed 
as in \cref{alg:dr}.
\begin{algorithm}
    Choose $(x^0_1,\dots,x^0_{m-1})\in \H^{m-1}$, parameters $(\gamma,\delta,\lambda,\mu)\in \Re^2_{++}\times (1,+\infty)^2$ that satisfy \eqref{eq:parameters_basic_requirement}, and $\kappa\in \Re_{++}$. \\
    For $k =0, 1, 2, \dots,$
    \begin{equation*}
        \begin{array}{rll}
            \ds y_i^{k}&  \in J_{\gamma A_i}(x_i^k), & (i=1,\dots,m-1) \\
		\ds 	\ds z^{k}&  \ds  \in J_{\frac{\delta}{m-1} A_m}\left(\frac{1}{m-1} \sum_{i=1}^{m-1}  ((1-\lambda)x_i^k+\lambda y_i^k)\right) \\
		\ds 	x_i^{k+1} & = x_i^k + \kappa \mu ( z^{k} - y_i^{k} ) & (i=1,\dots,m-1) .
        \end{array}
		\end{equation*}
	\caption{Doubly relaxed Douglas--Rachford for $m$-operator inclusion problem \eqref{eq:inclusion}.}
	\label{alg:dr}
\end{algorithm}

\subsection{Convergence via naive application of theory for comonotone operators}
\label{sec:directapplication}

We now establish the convergence of \eqref{eq:aDR_algorithm_bold} by establishing that $\F$ and $\G$ are maximally comonotone operators and then directly applying \cref{thm:aDR_comonotonecase}. We first provide formulas for the inverse mappings of $\F$ and $\G$, and derive their comonotonicity moduli under the assumption that $A_i$ is $\sigma_i$-comonotone for each $i$. 

\begin{lemma}[Inverse formulas]
\label{lemma:inverseformula}
    Let $\F$ and $\G$ be given by \eqref{eq:F} and \eqref{eq:G}, respectively. Then for all $\u =(u_1,\dots,u_{m-1})\in \H^{m-1}$, the following hold:
    \begin{enumerate}[(i)]
        \item $\F^{-1} (\u) = A_1^{-1}(u_1) \times \cdots \times A_{m-1}^{-1}(u_{m-1})$.
        \item $\ds \G^{-1}(\u) = \left\lbrace (a,\dots,a)\in \H^{m-1} : a \in A_m^{-1} \left( \sum_{i=1}^{m-1} u_i\right)\right\rbrace$.
    \end{enumerate}
\end{lemma}
\begin{proof}
This follows directly from the definition of the inverse and the identity $N_{\D_{m-1}} (\x) = \D_{m-1}^\perp$ for $\x\in \D_{m-1}$.
\smartqedmark \end{proof}

\begin{proposition}
    \label{prop:F_monotone}
    Suppose that $A_i:\H\toset \H$ is $\sigma_i$-comonotone for $i=1,\dots, m-1$. Then $\F$ given by \eqref{eq:F} is $\alpha$-comonotone, where $\alpha\coloneqq \ds \left( \min_{i=1,\dots,m-1}\sigma_i\right)$. Furthermore, if each $A_i$ is maximally $\sigma_i$-comonotone, then $\F$ is maximally $\alpha$-comonotone. 
\end{proposition}
\begin{proof}
If $A_i$ is $\sigma_i$-comonotone, then $A_i^{-1}$ is $\sigma_i$-monotone  by \cref{lemma:maximal_sigma_equivalent}. By \cref{lemma:inverseformula}(i)  and \cite[Proposition 4.1]{AlcantaraTakeda2025}, $\F^{-1}$ is $ \alpha$-monotone. Hence, $\F$ is $\alpha$-comonotone by \cref{lemma:maximal_sigma_equivalent}. The arguments for the second part are the same.
\smartqedmark \end{proof}

\begin{proposition}
\label{prop:G_monotone}
    Suppose that $A_m$ is $\sigma_m$-comonotone with $\sigma_m\leq 0$. Then $\G$ is $(m-1)\sigma_m$-comonotone. Furthermore, if $A_m$ is maximally $\sigma_m$-comonotone with $\sigma_m\leq 0$, then $\G$ is maximally $(m-1)\sigma_m$-comonotone. 
\end{proposition}
\begin{proof}
    By \cref{lemma:maximal_sigma_equivalent}, it suffices to show that $\G^{-1}$ is $(m-1)\sigma_m$-monotone. Let $(\u,\a),(\v,\b)\in \gra (\G^{-1})$. By \cref{lemma:inverseformula}, $\a = (a,\dots,a)$ and $\b = (b,\dots,b)$ where $a \in A_m^{-1} \left( \sum_{i=1}^{m-1} u_i\right)$ and $b \in A_m^{-1} \left( \sum_{i=1}^{m-1} v_i\right)$. Thus,
    \begin{equation}
        \inner{\u - \v}{\a -\b} =  \inner{\sum_{i=1}^{m-1} u_i-\sum_{i=1}^{m-1}v_i}{a-b}  \geq \sigma_m \norm{\sum_{i=1}^{m-1}(u_i-v_i)}^2,
        \label{eq:weirdmonotonicityofGinverse}
    \end{equation}
    where the last inequality holds by $\sigma_m$-monotonicity of $A_m^{-1}$, due to \cref{lemma:maximal_sigma_equivalent}. The claim now follows by noting that 
    \[\norm{\sum_{i=1}^{m-1}(u_i-v_i)}^2 \leq (m-1) \sum_{i=1}^{m-1}\norm{(u_i-v_i)}^2 = (m-1)\norm{\u - \v}^2\]
    and $\sigma_m\leq0$. To prove that $\G$ is maximally $(m-1)\sigma_m$-comonotone when $A_m$ is maximally $\sigma_m$-comonotone, let $\gamma>0$ be such that $\gamma + (m-1)\sigma_m>0$. By \cref{lemma:maximal_single-valued-resolvent-comonotone}(ii), it is enough to prove that $J_{\gamma \G}$ has full domain. 
   Indeed, since $A_m$ is maximally $\sigma_m$-comonotone and $\frac{\gamma}{m-1}+\sigma_m>0$, \cref{lemma:maximal_single-valued-resolvent-comonotone}(ii) implies that $J_{\frac{\gamma}{m-1}A_m}$ has full domain. Thus, $J_{\gamma \G}$ has full domain by noting \eqref{eq:Gresolvent}.
\smartqedmark \end{proof}

\begin{theorem}
\label{thm:adaptiveDR}
   Suppose that $A_i$ is maximally $\sigma_i$-comonotone for $i=1,\dots,m$, $\zer (A_1+\cdots + A_m) \neq \emptyset$, $\sigma_m\leq 0$ and $\sigma_i + (m-1)\sigma_m\geq 0$ for all $i=1,\dots,m-1$. Suppose $(\gamma,\delta,\lambda,\mu)\in \Re^4_{++}$ together with $(\alpha,\beta)\coloneqq (\min_{1\leq i \leq m-1}\sigma_i,(m-1)\sigma_m)$ satisfies \eqref{eq:parameters_basic_requirement_comonotone} and \eqref{eq:parameters_additional_reqs_comonotone}. Set
 $\kappa \in (0, \kappa^*)$ with $\kappa^*$ given by \eqref{eq:kappastar_comonotone}. If $\{ (\x^k,\y^k,\z^k)\}$ is generated by \eqref{eq:adr_stepbystep_bold} from an arbitrary initial point $\x^0\in \H^{m-1}$, then the following hold:
 \begin{enumerate}[(i)]
     \item There exists $\bar{\x}\in \Fix (T_{\F,\G})$ with $J_{\gamma \F}(\bar{\x})\in \zer (\F+\G)$ such that $\x^k\toweak \bar{\x}$; 
     \item $\norm{\x^k-\x^{k+1}}=o(1/\sqrt{k})$ and $\norm{\y^k-\y^{k+1}}=o(1/\sqrt{k})$ as $k\to\infty$;
     \item The shadow sequences $\{\y^k\}$ and $\{\z^k\}$ converge weakly to $J_{\gamma \F}(\bar{\x})\in \zer (\F+\G)$; and 
     \item If $\sigma_i+(m-1)\sigma_m>0$ for all $i=1,\dots,m-1$, $  0<\kappa <1$,
     and either
        \begin{enumerate}[leftmargin=2.5em]
            \item $  \kappa^*\geq 1$; or
            \item \alertt{$\gamma=\delta>0$, $\lambda=\mu=2$, }
        \end{enumerate}
    then the sequences $\{\x^k-\y^k \} $ and $\{\x^k-\z^k\}$ converge strongly to $\gamma \yosida{\gamma }\F(\bar{\x}) = \bar{\x}-J_{\gamma \F}(\bar{\x}) $ and $\alertt{\gamma }\yosida{\gamma}\F(\bar{\x}) + \zer (\F+\G) =  \Fix (T_{\F,\G})$. 
 \end{enumerate}
\end{theorem}
\begin{proof}
    Note that $\zer(\F+\G)\neq \emptyset$ due to \eqref{eq:zeros_equivalence}. Then the result is immediate by applying \cref{thm:aDR_comonotonecase}, \cref{prop:F_monotone} and \cref{prop:G_monotone}.
\smartqedmark \end{proof}

\subsection{A refined convergence analysis}
\label{sec:refinedanalysis}
In this section, we derive stronger convergence results for the doubly relaxed DR algorithm \eqref{eq:aDR_algorithm_bold}, specifically when $\sigma_m<0$. To this end, we consider the Attouch--Th\'{e}ra dual of \eqref{eq:alcantara_takeda}, and invoke the results established in \cref{sec:attouchthera}. In particular, we consider the dual problem of \eqref{eq:alcantara_takeda} given by 
\begin{equation*}
    \text{Find~} \u \in \H^{m-1} ~\text{such that } 0\in \F'(\u) + \G'(\u),
    \tag{\textbf{D}}
    \label{eq:D_bold}
\end{equation*}
where $    {\F}'= -\F^{-1}\circ (-\bfId) $ and  ${\G}'= \G^{-1}.$
The doubly relaxed DR algorithm for solving \eqref{eq:D_bold} with parameters $(\kappa, \gamma',\delta',\lambda',\mu')$ is given by 
\begin{equation}
    \u^{k+1} \in T_{{\F}',{\G}'}(\u^k) , \quad \forall k\in \mathbb{N},
    \tag{\textbf{DR'}}
    \label{eq:aDR_algorithm_dual_bold}
\end{equation}
where
\begin{equation}
    T_{{\F}',{\G}'} = (1-\kappa) \bfId + \kappa J_{\delta' {\G}'}^{\mu'}J_{\gamma' {\F}'}^{\lambda'}.
    \label{eq:Tf'g'}
\end{equation}
This can be equivalently written  as
\begin{subequations} \label{eq:adr_stepbystep_dual_bold}
    \begin{align}
        \v^k & \in  J_{\gamma' \F'}(\u^k)\label{eq:adr_stepbystep_dual_v_bold}\\
        \w^k & \in  J_{\delta'  \G'}((1-\lambda')\u^k + \lambda' \v^k) \label{eq:adr_stepbystep_dual_w_bold}\\
        \u^{k+1} & = \u^k + \kappa \mu' (\w^k - \v^k). \label{eq:adr_stepbystep_dual_u_bold}
    \end{align}
\end{subequations}
Similarly to \eqref{eq:T_AB_singlevalued}, we may write $T_{\F',\G'}$ as
    \begin{equation}
        T_{\F',\G'} = \bfId + \kappa\mu' (J_{\delta' \G'}J_{\gamma' \F'}^{\lambda'}-J_{\gamma' \F'}). \label{eq:T_FG'_singlevalued}
    \end{equation}

We first establish some properties of $\F'$ and $\G'$.
\begin{lemma}[Properties of relaxed resolvents]
    Let $A_i:\H\toset \H$ be $\sigma_i$-comonotone for each $i=1,\dots,m$. For any $\lambda'>1$, the following hold:
    \begin{enumerate}[(i)]
    \item For any $({\x},\u'), ({\y},\v')\in \gra (J_{\gamma' \F^{-1}}^{\lambda'})$,
    \begin{equation*}
        \norm{{\u}' - {\v}'}^2 \leq (\lambda'-1)^2\norm{\x-\y}^2 -\lambda'  \sum_{i=1}^{m-1}(2(\lambda'-1)(1+\gamma'\sigma_i)-\lambda') \norm{u_i - v_i}^2,
    \end{equation*}
    where $\u=(u_1,\dots,u_{m-1})\in J_{\gamma' \F^{-1}}(\x)$ and $\v = (v_1,\dots,v_{m-1}) \in J_{\gamma' \F^{-1}}(\y)$ are such that ${\u}'=\lambda' \u + (1-\lambda')\x$ and ${\v}'=\lambda' \v + (1-\lambda')\y$. 

    \item For any $({\x},\u'), ({\y},\v')\in \gra (J_{\gamma' \G ^{-1}}^{\lambda'})$,
    \begin{align*}
            \norm{\u' - \v '}^2 & \leq (\lambda'-1)^2 \norm{\x-\y}^2 +\lambda' (2-\lambda') \norm{\u-\v}^2 \\
            & \quad -2\lambda'(\lambda' - 1)\gamma' \sigma_m \norm{\sum_{i=1}^{m-1} (u_i - v_i)}^2,
    \end{align*}
     where $\u=(u_1,\dots,u_{m-1})\in J_{\gamma' \G^{-1}}(\x)$ and $\v = (v_1,\dots,v_{m-1}) \in J_{\gamma' \G^{-1}}(\y)$ are such that ${\u}'=\lambda' \u + (1-\lambda')\x$ and ${\v}'=\lambda' \v + (1-\lambda')\y$.
\end{enumerate}
\label{lemma:properties_reflected_resolvents}
\end{lemma}
\begin{proof}
    First, note that given $(\x,\u)$ and $(\y,\v)$ and defining  ${\u}'\coloneqq \lambda' \u + (1-\lambda')\x$ and ${\v}'\coloneqq \lambda' \v + (1-\lambda')\y$, we have 
    \begin{align}
        \norm{\u'-\v'}^2 = (\lambda')^2 \norm{\u - \v}^2  - 2\lambda' (\lambda'-1) \inner{\u-\v}{\x-\y} + (\lambda'-1)^2\norm{\x - \y}^2.
        \label{eq:u'-v'}
    \end{align}
    To prove (i), suppose $(\x,\u),(\y,\v)\in \gra (J_{\gamma' \F^{-1}})$. Let $\a\in \F^{-1}(\u)$ and $\b\in \F^{-1}(\v)$ such that $\x = \u + \gamma' \a$ and $\y = \v + \gamma' \b$. Then 
    \begin{align}
        \inner{\u-\v}{\x-\y}  & = \inner{\u - \v }{\u - \v + \gamma' (\a - \b )} \notag \\
        & = \norm{\u-\v}^2 + \gamma' \inner{\u-\v}{\a-\b} \label{eq:innprod_2nd} \\
        & = \norm{\u-\v}^2 + \gamma' \sum_{i=1}^{m-1} \inner{u_i-v_i}{a_i-b_i} \notag \\
        & \geq \norm{\u-\v}^2 + \gamma' \sum_{i=1}^{m-1}\sigma_i\norm{u_i-v_i}^2 ,\label{eq:innprod_lastline}
    \end{align}
where the last inequality holds by noting that $(a_i,u_i)\in \gra (A_i)$ (\cref{lemma:inverseformula}(i)) and $\sigma_i$-comonotonicity of $A_i$. Combining \eqref{eq:innprod_lastline} with \eqref{eq:u'-v'} proves part (i). The proof of (ii) is similar. For $(\x,\u),(\y,\v)\in \gra (J_{\gamma' \G^{-1}})$, we let $\a\in \G^{-1}(\u)$ and $\b\in \G^{-1}(\v)$ such that $\x = \u + \gamma' \a$ and $\y = \v + \gamma' \b$. From \eqref{eq:innprod_2nd} and \eqref{eq:weirdmonotonicityofGinverse}, we have 
    \begin{equation*}
        \inner{\u-\v}{\x-\y} \geq \norm{\u-\v}^2 + \gamma' \sigma_m \norm{\sum_{i=1}^{m-1} (u_i-v_i)}^2.
    \end{equation*}
The claim holds by combining this with \eqref{eq:u'-v'}.
\smartqedmark \end{proof}

\begin{lemma}
\label{lemma:composition_of_relaxed_resolvents}
   Let $\u,\bar{\u}\in \H^{m-1}$ and let $(\gamma',\delta',\lambda',\mu')\in \Re^4_{++}$ satisfy \eqref{eq:dual_parameters_identity}. Suppose that $J_{\gamma'\F^{-1}}$ and $J_{\delta' \G^{-1}}$ are single-valued with full domain, and denote $\v =(v_1,\dots,v_{m-1})\coloneqq J_{\gamma' \F'}(\u)$ and $ \bar{\v} =(\bar{v}_1,\dots,\bar{v}_{m-1}) \coloneqq J_{\gamma' \F'}(\bar{\u})$. Further, denote $\s  \coloneqq   J_{\gamma' \F'}^{\lambda'}(\u)$ and $ \bar{\s} \coloneqq  J_{\gamma' \F'}^{\lambda'}(\bar{\u})$, and let $\w = (w_1,\dots,w_{m-1})\coloneqq J_{\delta' {\G^{-1}}} \left( \s \right)$, and $\bar{\w}= (\bar{w}_1,\dots,\bar{w}_{m-1})\coloneqq J_{\delta' {\G^{-1}}} \left( \bar{\s} \right)$. Then 
    \begin{align}
       &  \norm{J_{\delta' {\G}'}^{\mu'}J_{\gamma' {\F}'}^{\lambda'}(\u) - J_{\delta' {\G}'}^{\mu'}J_{\gamma' {\F}'}^{\lambda'}(\bar{\u})}^2  \notag \\
        & \leq\norm{\u-\bar{\u}}^2 - \sum_{i=1}^{m-1} \mu' (2+2\gamma'\sigma_i-\mu')\norm{v_i-\bar{v}_i}^2 \notag \\
        & \quad  +\mu'(2-\mu')\norm{\w -\bar{\w}}^2 -2\mu'\gamma'\sigma_m \norm{\sum_{i=1}^{m-1}(w_i-\bar{w}_i)}^2 . \label{eq:composition_ineq}
    \end{align}
Consequently, denoting $\R \coloneqq \bfId - T_{\F',\G'}$, 
  \begin{align}
       \norm{T_{\F',\G'}(\u)-T_{\F',\G'}(\bar{\u})}^2&  \leq \norm{\u-\bar{\u}}^2 - \frac{1-\kappa}{\kappa} \norm{\R(\u) - \R(\bar{\u})}^2 \notag \\
        & - \kappa \sum_{i=1}^{m-1} \mu' (2+2\gamma'\sigma_i-\mu')\norm{v_i-\bar{v}_i}^2 \notag \\
        &  +\kappa \mu'(2-\mu')\norm{\w -\bar{\w}}^2 -2\kappa \mu'\gamma'\sigma_m \norm{\sum_{i=1}^{m-1}(w_i-\bar{w}_i)}^2 . \label{eq:Tnonexpansive}
    \end{align}
\end{lemma}
\begin{proof}
Since $\G'=\G^{-1}$, we have 
\begin{align}
    \norm{J_{\delta' {\G}'}^{\mu'}J_{\gamma' {\F}'}^{\lambda'}(\u) - J_{\delta' {\G}'}^{\mu'}J_{\gamma' {\F}'}^{\lambda'}(\bar{\u})}^2 &  = \norm{J_{\delta' {\G^{-1}}}^{\mu'} \left( \s \right) - J_{\delta' {\G^{-1}}}^{\mu'} \left(\bar{\s}\right) }^2 .\label{eq:JFG_diff}
\end{align}
By \cref{lemma:properties_reflected_resolvents}(ii), we have
    \begin{align}
        & \norm{J_{\delta' {\G^{-1}}}^{\mu'} \left( \s \right) - J_{\delta' {\G^{-1}}}^{\mu'} \left(\bar{\s}\right) }^2\notag \\
        & \leq (\mu' -1 )^2 \norm{\s - \bar{\s}}^2  +\mu'(2-\mu')\norm{\w -\bar{\w}}^2 -2\mu'(\mu'-1)\delta'\sigma_m \norm{\sum_{i=1}^{m-1}(w_i-\bar{w}_i)}^2 \notag \\
        & = (\mu' -1 )^2 \norm{\s - \bar{\s}}^2  +\mu'(2-\mu')\norm{\w -\bar{\w}}^2 -2\mu'\gamma'\sigma_m \norm{\sum_{i=1}^{m-1}(w_i-\bar{w}_i)}^2 ,\label{eq:JdeltamuG_diff}
    \end{align}
where the last line follows from \eqref{eq:dual_parameters_identity}. Meanwhile, note that since $\F' = -\F^{-1} \circ (-\bfId)$, then $J_{\gamma' \F' } (\u) = -J_{\gamma' \F^{-1}}(-\u)$ for any $\u \in \H^{m-1}$, and therefore $ J_{\gamma' \F'}^{\lambda'} (\x) = -J_{\gamma' \F^{-1}}^{\lambda'}(-\x)\quad \forall \x \in \H^{m-1}. $
Meanwhile, we have from \cref{lemma:inverseformula}(i) and \eqref{eq:Fresolvent} that $J_{\gamma' \F^{-1}} (\x) = J_{\gamma' A_1^{-1}}(x_1) \times \cdots \times J_{\gamma' A_{m-1}^{-1}}(x_{m-1})$. By \cref{lemma:properties_reflected_resolvents}(i), we have 
    \begin{align*}
        & \norm{\s-\bar{\s}}^2 = \norm{ J_{\gamma' {\F}^{-1}}^{\lambda'}(-\u) -  J_{\gamma' {\F}^{-1}}^{\lambda'}(-\bar{\u})}^2 \notag  \\ 
        & \leq (\lambda'-1)^2\norm{\u-\bar{\u}}^2 \notag \\
        & \quad - \lambda' \sum_{i=1}^{m-1} (2(\lambda'-1)(1+\gamma'\sigma_i)-\lambda')
        \norm{J_{\gamma' A_i^{-1}}(-u_i) - J_{\gamma' A_i^{-1}}(-\bar{u}_i)}^2\notag \\
       & =  (\lambda'-1)^2\norm{\u-\bar{\u}}^2 - \lambda' \sum_{i=1}^{m-1} (2(\lambda'-1)(1+\gamma'\sigma_i)-\lambda')\norm{v_i-\bar{v}_i}^2.
    \end{align*}
Multiplying both sides by $(\mu'-1)^2$ and using \eqref{eq:dual_parameters_identity}, we obtain
    \begin{align}
       &  (\mu'-1)^2\norm{\s-\bar{\s}}^2 \notag  \\ 
        & \leq  \norm{\u-\bar{\u}}^2 - \lambda'(\mu'-1)^2 \sum_{i=1}^{m-1} (2(\lambda'-1)(1+\gamma'\sigma_i)-\lambda')\norm{v_i-\bar{v}_i}^2 \notag \\
        &  = \norm{\u-\bar{\u}}^2 - \sum_{i=1}^{m-1} \mu' (2+2\gamma'\sigma_i-\mu')\norm{v_i-\bar{v}_i}^2 \label{eq:JgammaF_diff} ,
    \end{align}
where the last line follows by writing 
\[\lambda'(\mu'-1)^2 (2(\lambda'-1)(1+\gamma'\sigma_i)-\lambda') = (\lambda'(\mu'-1)) \cdot (2(\mu'-1)(\lambda'-1)(1+\gamma'\sigma_i) - \lambda'(\mu'-1)).\]
Combining \eqref{eq:JgammaF_diff} with \eqref{eq:JdeltamuG_diff} and \eqref{eq:JFG_diff} gives \eqref{eq:composition_ineq}.  To prove \eqref{eq:Tnonexpansive}, note that \begin{align}
        \norm{T_{\F',\G'}(\u) - T_{\F',\G'}(\bar{\u})}^2 & = (1-\kappa) \norm{\u-\bar{\u}}^2 + \kappa \norm{J_{\delta' {\G}'}^{\mu'} J_{\gamma' {\F}'}^{\lambda'}(\u) - J_{\delta' {\G}'}^{\mu'} J_{\gamma' {\F}'}^{\lambda'}(\bar{\u})}^2 \notag \\
        &- (1-\kappa)\kappa \norm{\left(\u - J_{\delta' {\G}'}^{\mu'} J_{\gamma' {\F}'}^{\lambda'}(\u) \right) - \left( \bar{\u}- J_{\delta' {\G}'}^{\mu'} J_{\gamma' {\F}'}^{\lambda'}(\bar{\u})\right)}^2 \notag \\ 
         & = (1-\kappa) \norm{\u-\bar{\u}}^2 + \kappa \norm{J_{\delta' {\G}'}^{\mu'} J_{\gamma' {\F}'}^{\lambda'}(\u) - J_{\delta' {\G}'}^{\mu'} J_{\gamma' {\F}'}^{\lambda'}(\bar{\u})}^2 \notag \\
        & \quad - \frac{1-\kappa}{\kappa} \norm{(\bfId - T_{\F',\G'})(\u) - (\bfId - T_{\F',\G'})(\bar{\u})}^2 , \notag 
    \end{align}
    where the first equality holds by \eqref{eq:identity_squarednorm}, and the second holds by using the identity $\bfId - T_{\F',\G'} = \kappa (\bfId - J_{\delta' {\G}'}^{\mu'}J_{\gamma' {\F}'}^{\lambda'})$ which can be derived from \eqref{eq:Tf'g'}. The above identity together with \eqref{eq:composition_ineq} gives \eqref{eq:Tnonexpansive}. 
\smartqedmark \end{proof}

\begin{lemma}
\label{lemma:Tnonexpansive2}
   Suppose that the hypotheses of \cref{lemma:composition_of_relaxed_resolvents} hold. Let $\theta_1,\dots,\theta_{m-1}$ be nonzero real numbers such that $\sum_{i=1}^{m-1} \frac{1}{\theta_i}=1$, and for $i=1,\dots,m-1$, denote
   \begin{equation}
    \begin{array}{lll}
         &  \alpha_i \coloneqq 2\gamma'\sigma_i  + 2-\mu'  ,    & \qquad \beta_i \coloneqq 2\gamma'\sigma_m\theta_i-2+\mu' , \\
         &  \nu_i  \coloneqq \sigma_i + \sigma_m\theta_i , &   \qquad  \omega_i  \coloneqq  1-\kappa + \frac{\alpha_i\beta_i}{2\gamma'\mu'\nu_i}.
    \end{array}
    \label{eq:coefficients_and_parameters}
   \end{equation}
    \alertt{If $\nu_i \neq 0$ for all $i=1,\dots,m-1$,} then 
   \begin{align}
       \norm{T_{\F',\G'}(\u)-T_{\F',\G'}(\bar{\u})}^2&  \leq \norm{\u-\bar{\u}}^2 - \frac{1}{\kappa}\sum_{i=1}^{m-1} \omega_i \norm{R_i(\u)-R_i(\bar{\u})}^2 \notag \\
        & - \frac{\mu'\kappa}{2\gamma'}\sum_{i=1}^{m-1} \frac{1}{\nu_i}\norm{\alpha_i(v_i-\bar{v}_i) + \beta_i(w_i-\bar{w}_i)}^2 \notag \\
        & +2\kappa \mu' \gamma'\sigma_m \sum_{1\leq i<j\leq m-1} \frac{1}{\theta_i\theta_j}\norm{\theta_i (w_i-\bar{w}_i) - \theta_j (w_j-\bar{w}_j)}^2 . \notag 
    \end{align}
    
\end{lemma}

\begin{proof}  
Using \cref{lemma:squaredsum_to_sumsquared}, we obtain that
\begin{align}
   &  2\gamma'\sigma_m \norm{\sum_{i=1}^{m-1}(w_i-\bar{w}_i)}^2 - (2-\mu')\norm{\w -\bar{\w}}^2\notag \\
   & = \sum_{i=1}^{m-1} \beta_i\norm{w_i-\bar{w}_i}^2 -2\gamma'\sigma_m \sum_{1\leq i<j\leq m-1} \frac{1}{\theta_i\theta_j}\norm{\theta_i (w_i-\bar{w}_i) - \theta_j (w_j-\bar{w}_j)}^2 .\label{eq:w_part}
\end{align}
Meanwhile, from \eqref{eq:identity_squarednorm2}, we have 
\begin{align}
    & \alpha_i\norm{v_i-\bar{v}_i}^2 +  \beta_i\norm{w_i-\bar{w}_i}^2 \notag \\
    & = \frac{\alpha_i\beta_i}{2\gamma' (\sigma_i+\sigma_m\theta_i)}\norm{(v_i-w_i) - (\bar{v}_i-\bar{w}_i)}^2 + \frac{1}{2\gamma'(\sigma_i+\sigma_m\theta_i)}\norm{\alpha_i(v_i-\bar{v}_i) + \beta_i(w_i-\bar{w}_i)}^2.\notag 
\end{align}
Hence, noting from \eqref{eq:T_FG'_singlevalued} that $\R(\u) =( \bfId - T_{\F',\G'})(\u) = -\kappa \mu' (\w-\v)$ and $\R(\bar{\u}) = (\bfId - T_{\F',\G'})(\bar{\u}) = -\kappa \mu' (\bar{\w}-\bar{\v})$, we have 
\begin{align}
    & \alpha_i\norm{v_i-\bar{v}_i}^2 +  \beta_i\norm{w_i-\bar{w}_i}^2 \notag \\
    & = \frac{\alpha_i\beta_i}{2\gamma' (\sigma_i+\sigma_m\theta_i)\kappa^2\mu'^2}\norm{R_i(\u) - R_i(\bar{\u})}^2 + \frac{1}{2\gamma'(\sigma_i+\sigma_m\theta_i)}\norm{\alpha_i(v_i-\bar{v}_i) + \beta_i(w_i-\bar{w}_i)}^2.\label{eq:v+w}
\end{align}
Finally, the result follows by combining \eqref{eq:Tnonexpansive}, \eqref{eq:w_part} and \eqref{eq:v+w}.
\smartqedmark \end{proof}

\begin{lemma}
\label{lemma:Theta}
    Let $\sigma_1,\dots,\sigma_m\in \Re$ such that $\sigma_i>0$ for $i=1,\dots,m-1$, $\sigma_m<0$ and $\sum_{i=1}^{m} \frac{1}{\sigma_i}<0$. Then the set
        \begin{equation}
        \Theta\coloneqq \left\lbrace \theta = (\theta_1,\dots,\theta_{m-1}) \in \Re^{m-1} : \forall i, ~\theta_i>0, ~ \sigma_i+\sigma_m\theta_i>0,~\text{and}~\sum_{i=1}^{m-1}\frac{1}{\theta_i}= 1\right\rbrace    
        \label{eq:Theta}
    \end{equation}
    is nonempty. \alertt{Moreover, if $A_i$ is maximally $\sigma_i$-comonotone for every $i$, then for any $\gamma',\delta'>0$ such that $1+\delta'(m-1)\sigma_m>0$,     the resolvents  $J_{\gamma' \F^{-1}}$ and $J_{\delta' \G^{-1}}$ are single-valued with full domain.   }
\end{lemma}
\begin{proof}
Let $X\coloneqq \left\lbrace  (\delta_1,\dots,\delta_{m-1}) \in \Re^{m-1} : -\frac{1}{\sigma_i} - \frac{1}{\sigma_m}\delta_i >0 ~\text{and}~\sum_{i=1}^{m-1}\delta_i = 1\right\rbrace  .$
    Since $\sum_{i=1}^{m} \frac{1}{\sigma_i}<0$, we have from  \cite[Proposition 4.12(iii)]{AlcantaraTakeda2025} that $X$ is nonempty. Moreover, noting that $\sigma_m<0$ and $\sigma_i>0$ for all $i=1,\dots,m-1$, it follows that $X\subseteq (0,+\infty)^{m-1}$. Take any $\delta\in X$ and set $\theta_i = \frac{1}{\delta_i}>0$ for each $i=1,\dots, m-1$. Then  $\theta =(\theta_1,\dots,\theta_{m-1})\in (0,+\infty)^{m-1}$ belongs to $\Theta$, and so $\Theta\neq \emptyset$.
    
    \alertt{The maximal monotonicity of $\F^{-1}$ by \cref{prop:F_monotone,lemma:maximal_sigma_equivalent} immediately implies that $J_{\gamma'\F^{-1}}$ is single-valued with full domain (see \cref{lemma:maximalmonotone_properties}). On the other hand, since $\G^{-1}$ is maximally $(m{-}1)\sigma_m$-monotone by \cref{prop:G_monotone,lemma:maximal_sigma_equivalent}, the claim on $J_{\delta'\G^{-1}}$
    	\alertt{follows from}
    	\cref{lemma:maximal_single-valued-resolvent-monotone}.}
\smartqedmark \end{proof}

We are now ready to establish the convergence of the doubly relaxed DR algorithm \eqref{eq:adr_stepbystep_dual_bold}. 

\begin{lemma}
\label{lemma:improvedconvergence}
    Suppose that $A_i$ is maximally $\sigma_i$-comonotone for $i=1,\dots,m$ such that $\zer (A_1+\cdots + A_m)\neq \emptyset$, $\sigma_i>0$ for $i=1,\dots,m-1$, $\sigma_m<0$, and  $\ds \sum_{i=1}^{m} \frac{1}{\sigma_i}<0$. Let $\theta\in \Theta$, with $\Theta$ given in \eqref{eq:Theta} and  $\kappa\in (0,1)$. \alertt{Let $(\gamma', \delta',\lambda',\mu')\in \Re_{++}^4$ satisfy \eqref{eq:dual_parameters_identity}, define $(\alpha_i,\beta_i,\nu_i,\omega_i)$ as in \eqref{eq:coefficients_and_parameters}, and suppose that $1+\delta'(m-1)\sigma_m>0$ and $\omega_i>0$ for all $i=1,\dots,m-1$.}
    
    Let $\{ (\u^k,\v^k,\w^k)\}$ be a sequence generated by \eqref{eq:adr_stepbystep_dual_bold}. Then the following hold:
    \begin{enumerate}[(i)]
        \item There exists $\bar{\u}\in\Fix (T_{\F',\G'})$ with $J_{\gamma'\F'}(\bar{\u})\in \zer (\F'+\G')$ such that $\u^k\toweak \bar{\u}$;
        \item $\norm{\u^k-\u^{k+1}}=o(1/\sqrt{k})$ and $\norm{\v^k-\v^{k+1}}=o(1/\sqrt{k})$ as $k\to\infty$;
        \item The shadow sequences $\{\v^k\}$ and $\{\w^k\}$ converge strongly to $J_{\gamma'\F'}(\bar{\u})$ and $ \zer (\F'+\G') = \{J_{\gamma'\F'}(\bar{\u}) \} $. 
    \end{enumerate}
\end{lemma}
\begin{proof}
	Since $\theta\in \Theta$, note that $\nu_i>0$ and $\theta_i\theta_j>0$ for any $i,j$. \alertt{From \cref{lemma:Theta}, the resolvents $J_{\gamma' \F^{-1}} $ and $J_{\delta'\G^{-1}}$ are single-valued with full domain. Hence, we can now invoke \cref{lemma:Tnonexpansive2}} and obtain 
	\begin{align}
		\norm{T_{\F',\G'}(\u)-T_{\F',\G'}(\bar{\u})}^2&  \leq \norm{\u-\bar{\u}}^2 - \frac{1}{\kappa}\sum_{i=1}^{m-1} \omega_i \norm{R_i(\u)-R_i(\bar{\u})}^2 \notag \\
		& - \frac{\mu'\kappa}{2\gamma'}\sum_{i=1}^{m-1} \frac{1}{\nu_i}\norm{\alpha_i(v_i-\bar{v}_i) + \beta_i(w_i-\bar{w}_i)}^2 ,\notag 
		\label{eq:Tnonexpansive_simplified2}
	\end{align}
	with $\omega_i,\nu_i>0$. \alertt{Noting that $\alpha_i+\beta_i = 2\gamma'\nu_i>0$ for all $i$, we can follow analogous arguments to those in \cite[Theorem 4.2]{DaoPhan2019} to prove the claims.} 
	\smartqedmark \end{proof}

The main result, which is an improvement of \cref{thm:adaptiveDR}, is now a direct consequence of the above lemma and the duality established in \cref{sec:attouchthera} when $\sigma_m<0$. 
\begin{theorem}
\label{thm:adaptiveDR_improved}
      Suppose that $A_i$ is maximally $\sigma_i$-comonotone for $i=1,\dots,m$ such that $\zer (A_1+\cdots + A_m)\neq \emptyset$, $\sigma_i>0$ for $i=1,\dots,m-1$, $\sigma_m<0$, and  $\ds \sum_{i=1}^{m} \frac{1}{\sigma_i}<0$. Let $(\gamma,\delta,\lambda,\mu)\in \Re^4_{++} $ satisfy \eqref{eq:parameters_basic_requirement_comonotone} and  $\theta\in \Theta$, where $\Theta$ is given in \eqref{eq:Theta}. 
      For $i=1,\dots,m-1$, define
     \begin{equation}
         \kappa_i^* \coloneqq \frac{4(\gamma+\sigma_i)(\delta+\sigma_m\theta_i)-(\gamma+\delta)^2}{2(\gamma+\delta)(\sigma_i+\sigma_m\theta_i)}.
         \label{eq:kappa_i}
     \end{equation}
      Suppose that $\kappa \in (0,1)$ and \alertt{$\ds \min_{1\leq i\leq m-1} \kappa_i^* >\kappa $.}
    If $\{ (\x^k,\y^k,\z^k)\}$ is generated by \eqref{eq:adr_stepbystep_bold} given an arbitrary initial point $\x^0\in \H^{m-1}$, then the following hold:
    \begin{enumerate}[(i)]
        \item There exists $\bar{\x}\in\Fix (T_{\F,\G})$ with $J_{\gamma\F}(\bar{\x})\in \zer (\F+\G)$ such that $\x^k\toweak \bar{\x}$;
        \item $\norm{\x^k-\x^{k+1}}=o(1/\sqrt{k})$ and $\norm{\y^k-\y^{k+1}}=o(1/\sqrt{k})$ as $k\to\infty$;
        
        \item The shadow sequences $\{\y^k\}$ and $\{\z^k\}$ converge weakly to $J_{\gamma\F}(\bar{\x})\in \zer (\F+\G)$; 
        
        \item The sequences $\{ \x^k-\y^k\}$ and $\{\x^k-\z^k\}$ converge strongly to $\gamma \yosida{\gamma}\F(\bar{\x}) = \bar{\x}- J_{\gamma \F}(\bar{\x})$ and $\alertt{\gamma}\yosida{\gamma}\F(\bar{\x}) + \zer (\F+\G) = \Fix (T_{\F,\G})$.
    \end{enumerate}
\end{theorem}
\begin{proof}
	Set $(\gamma', \delta', \lambda', \mu') \in \Re_{++}^4$ as in \eqref{eq:dual_parameters}, and let $(\alpha_i,\beta_i,\nu_i,\omega_i)$ be given by \eqref{eq:coefficients_and_parameters}. \alertt{Direct calculation gives $\kappa_i^* = 1+ \frac{\alpha_i\beta_i}{2\gamma'\mu'\nu_i}$ and hence $\omega_i=\kappa_i^*-\kappa>0$. Moreover, since $\kappa_i^* >\kappa>0$, \eqref{eq:kappa_i} implies that $\delta + \sigma_m \theta_i>0$ for all $i=1,\dots,m-1$. Dividing by $\theta_i$ and summing yields $\delta+(m-1)\sigma_m>0$. Hence, all the conditions in \cref{lemma:improvedconvergence} are satisfied.}
	The remainder of the proof follows by applying \cref{lemma:improvedconvergence}(i)--(iii) together with \cref{lemma:adr_primal_dual}; see also the proof of \cref{thm:aDR_comonotonecase}.\smartqedmark \end{proof}

\begin{remark}
	Recall that in \cref{thm:adaptiveDR}, \alertt{the strict condition required for strong convergence} when $\sigma_m<0$ is that $\sigma_i+\sigma_m(m-1)>0$ for all $i=1,\dots,m-1$. This readily implies that $\sigma_i>0$ for $i=1,\dots,m-1$ and that \alertt{$\sum_{i=1}^{m}\frac{1}{\sigma_i}<0$}. Hence, the condition in \cref{thm:adaptiveDR_improved} is indeed weaker than \alertt{the strict condition required for strong convergence in} \cref{thm:adaptiveDR}; see also \cref{fig:compare_regions}. Moreover, \alertt{the strong-convergence conclusion} in \cref{thm:adaptiveDR} for the case $\sigma_m<0$ can also be derived from \cref{thm:adaptiveDR_improved} by setting $\theta_i\equiv m-1$ for all $i$, \alertt{provided that $\kappa\in(0,1)$}.
\end{remark}
\ifdefined\submit
  \def\figscale{0.5}
\else
  \def\figscale{.7}
\fi

 \begin{figure}[tb!]
	\centering
	\includegraphics[scale=\figscale]{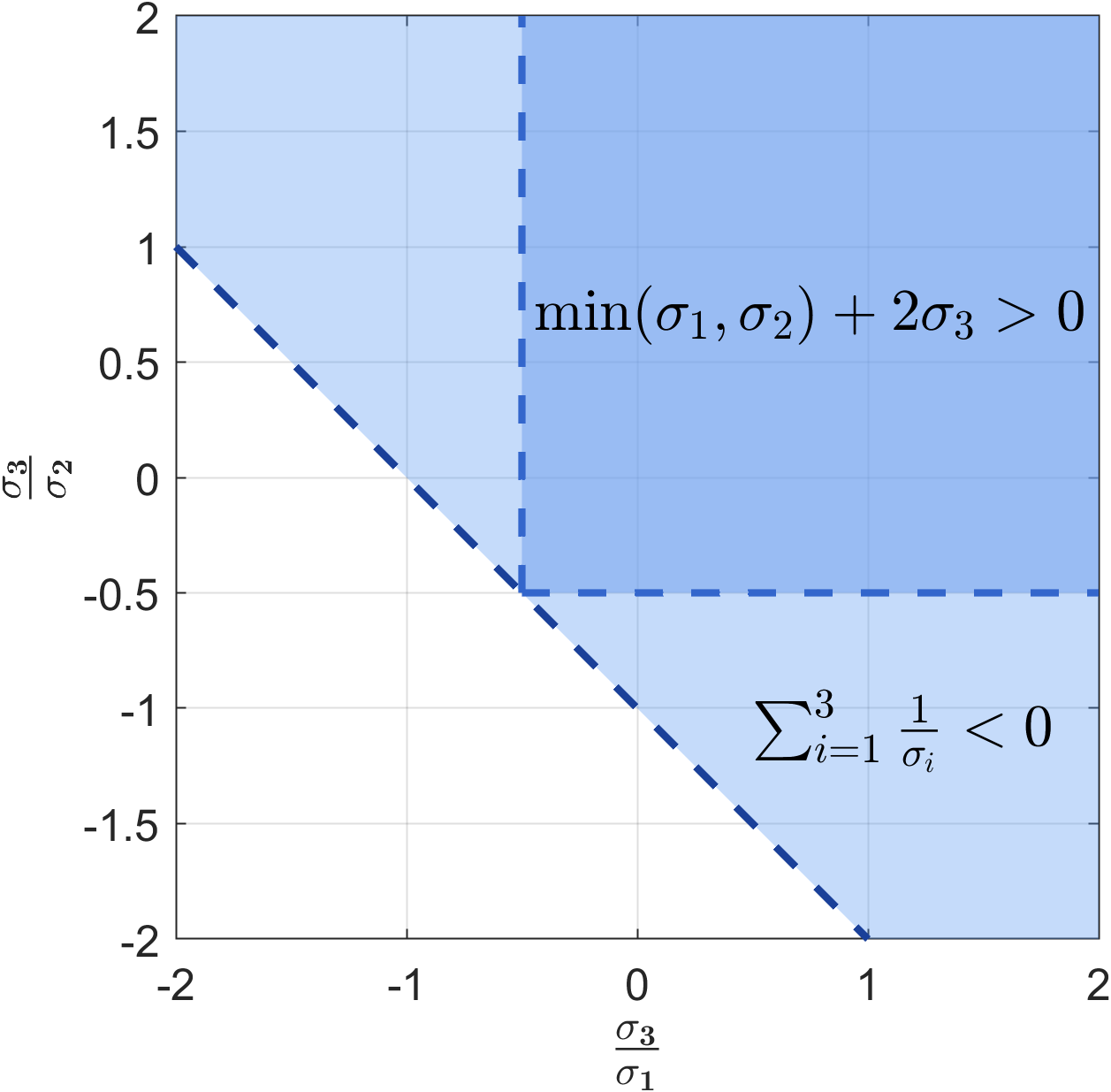}
\caption{{\small Convergence regions for the doubly relaxed Douglas--Rachford method with $m=3$ and $\sigma_3<0$.
		The axes plot $\sigma_3/\sigma_1$ (horizontal) and $\sigma_3/\sigma_2$ (vertical).
		\alert{The dark region corresponds to the strict condition on the comonotonicity moduli in \cref{thm:adaptiveDR}; this region is also covered by \cref{thm:adaptiveDR_improved} for suitable choices of the algorithm parameters.}
		\alert{The light region consists of the additional configurations of the comonotonicity moduli that are not covered by the direct two-operator reduction in \cref{thm:adaptiveDR}, but for which convergence is guaranteed by the refined analysis in \cref{thm:adaptiveDR_improved} for suitable choices of the algorithm parameters.}
	}
}
\label{fig:compare_regions}
\end{figure}

For clarity, we summarize in one theorem the results in \cref{thm:adaptiveDR,thm:adaptiveDR_improved}. 

\begin{theorem}
	\label{thm:adaptiveDR_combined}
	Suppose that $A_i$ is maximally $\sigma_i$-comonotone for $i=1,\dots,m$, and $\zer (A_1+\cdots + A_m) \neq \emptyset$. Denote $\underline{\sigma} \coloneqq  \min_{1\leq i\leq m-1}\sigma_i$. Let $(\gamma,\delta,\lambda,\mu)\in \Re^4_{++}$ be parameters that satisfy \eqref{eq:parameters_basic_requirement_comonotone} and suppose that one of the following conditions holds:
	\begin{enumerate}[(C1)]
		\item $\underline{\sigma} =- (m-1)\sigma_m\geq 0$ and $\delta = \gamma + 2\underline{\sigma}$, \alertt{and $\kappa\in(0,1)$}; 
		
		\item $\underline{\sigma} >- (m-1)\sigma_m\geq  0$ and $ (\gamma + \delta)^2
		<
		4\left(\gamma + \underline{\sigma}\right)
		\left(\delta+(m-1)\sigma_m\right),$
		\alertt{and $\kappa\in(0,\kappa^*)$, where}
		\begin{equation}
			\alertt{
				\kappa^*
				\coloneqq
				\frac{
					4(\gamma+\underline{\sigma})
					(\delta+(m-1)\sigma_m)
					-(\gamma+\delta)^2
				}{
					2(\gamma+\delta)
					(\underline{\sigma}+(m-1)\sigma_m)
				}.}
			\label{eq:kappastar_3cases}
		\end{equation}
		
		\item $\sigma_i>0$ for $i=1,\dots,m-1$, $\sigma_m<0$, and $\sum_{i=1}^m \frac{1}{\sigma_i}<0.$
		\alertt{In addition, $\kappa\in(0,1)$ and $
			\min_{1\leq i\leq m-1}\kappa_i^*>\kappa,$
			where $\kappa_i^*$ is given by \eqref{eq:kappa_i} for some
			$\theta=(\theta_1,\dots,\theta_{m-1})\in\Theta$, with $\Theta$
			defined in \eqref{eq:Theta}.}
	\end{enumerate}
	
	If $\{(\x^k,\y^k,\z^k)\}$ is generated by
	\eqref{eq:adr_stepbystep_bold} from an arbitrary initial point
	$\x^0\in\H^{m-1}$, then the following hold:
	\begin{enumerate}[(i)]
		\item There exists $\bar{\x}\in \Fix(T_{\F,\G})$ with
		$J_{\gamma\F}(\bar{\x})\in\zer(\F+\G)$ such that
		$\x^k\toweak\bar{\x}$;
		
		\item
		\[
		\norm{\x^k-\x^{k+1}}=o(1/\sqrt{k})
		\quad\text{and}\quad
		\norm{\y^k-\y^{k+1}}=o(1/\sqrt{k})
		\quad\text{as }k\to\infty;
		\]
		
		\item The shadow sequences $\{\y^k\}$ and $\{\z^k\}$ converge
		weakly to
		$J_{\gamma\F}(\bar{\x})\in\zer(\F+\G)$.
		
		\item \alertt{If (C3) holds, or if (C2) holds and, in addition,
			$0<\kappa<1$ and either $\kappa^*\geq1$, or
			$\gamma=\delta>0$ and $\lambda=\mu=2$,}
		then $\{\x^k-\y^k\}$ and $\{\x^k-\z^k\}$ converge strongly to
		\[
		\gamma\yosida{\gamma}\F(\bar{\x})
		=
		\bar{\x}-J_{\gamma\F}(\bar{\x}),
		\]
		and
		\[
		\alertt{\gamma}\yosida{\gamma}\F(\bar{\x})
		+\zer(\F+\G)
		=
		\Fix(T_{\F,\G}).
		\]
	\end{enumerate}
\end{theorem}

\begin{remark}[Switching $\F$ and $\G$]\label{remark:switched_ADR}
The doubly relaxed DR algorithm 
\begin{equation}
    \x^{k+1} \in T_{{\G},{\F}}(\x^k) \coloneqq ( (1-\kappa) \bfId + \kappa J_{\gamma {\F}}^{\lambda} J_{\delta {\G}}^{\mu})(\x^k), \quad \forall k\in \mathbb{N},
    \label{eq:aDR_algorithm_bold_switched}
\end{equation}
where $\F$ and $\G$ are switched, can also be expressed as
\begin{subequations} \label{eq:adr_stepbystep_bold_switched}
    \begin{align}
        \y^k & \in  J_{\delta \G}(\x^k)\\
        \z^k & \in  J_{\gamma   \F}((1-\mu)\x^k + \mu \y^k) \\ 
        \x^{k+1} & = \x^k + \kappa \lambda (\z^k - \y^k).
    \end{align}
\end{subequations}
In general, the order of the operators may affect the convergence behavior of the algorithm. However, we note that for \eqref{eq:aDR_algorithm_bold_switched}, we still obtain the convergence results established in \cref{thm:adaptiveDR_combined} under the same hypotheses on the moduli $\sigma_1,\dots,\sigma_{m-1}$. In particular, the statement of \cref{thm:adaptiveDR_combined} holds for the switched algorithm \eqref{eq:adr_stepbystep_bold_switched} by simply replacing $T_{\F,\G}$, $J_{\gamma \F}$, and $\gamma \yosida{\gamma}\F$ with $T_{\G,\F}$, $J_{\delta \G}$ and $\delta \yosida{\delta}\G$, respectively. While the proof is not entirely straightforward, it essentially follows similar arguments to those in the preceding section and is therefore omitted for brevity. 
\end{remark}

\section{Multiblock ADMM}
\label{sec:admm}

We now revisit the multiblock optimization problem 
\begin{equation}
    \min_{u_i \in \mathcal{H}_i} \, f_1(u_1) + \cdots + f_m(u_m) \quad \text{s.t.} \quad \sum_{i=1}^m L_i u_i = b,
    \label{eq:multiblock_again}
\end{equation}
where $f_i : \mathcal{H}_i \to (-\infty, +\infty]$ is a proper function,  $L_i : \mathcal{H}_i \to \mathcal{H}$ is a bounded linear operator, and $b\in \H$. \alert{Our focus in this section is on deriving a multiblock ADMM-type scheme from
	the multioperator Douglas--Rachford framework and establishing its convergence
	properties. The resulting iteration is closely related to that proposed in
	\cite{HeYuan2015}; see \cref{remark:relatedworks}.}

To motivate the operator inclusion reformulation of \eqref{eq:multiblock_again}, we recall the following definition; see also \cite[page 16]{MalitskyTam2023}.
\begin{definition}
\label{defn:kuhntucker}
    We say that  $(u_1,\dots,u_m,y)\in \H_1\times \cdots \times \H_m \times \H$ is a \emph{Karush--Kuhn--Tucker (KKT) point for \eqref{eq:multiblock_again}} if  $ 0\in \widehat{\partial} f_i (u_i) + L_i^* y$ for all $i=1,\dots,m$ and $\sum_{i=1}^mL_iu_i = b$.
\end{definition}
Hence, if $(u_1,\dots,u_m,y)\in \H_1\times \cdots \times \H_m \times \H$ is a KKT point for \eqref{eq:multiblock_again}, then $u_i\in (\widehat{\partial} f_i)^{-1}(-L_i^*y)$ for all $i$. Since $\sum_{i=1}^mL_iu_i = b$, it is immediate to see that $y$ is a solution of the multioperator inclusion problem \eqref{eq:inclusion} with 
$A_i:\H\toset \H$  given by 
\begin{equation}
    A_i \coloneqq 
    \begin{cases}
    -L_i \circ (\widehat{\partial} f_i) ^{-1} \circ \left(-L^*_i\right) & \text{if}~i =1,\dots,m-1, \\   
    -L_m \circ (\widehat{\partial} f_m) ^{-1} \circ \left(-L_m^*\right) + b & \text{if}~i = m,
    \end{cases}
    \label{eq:Ai_for_multiblock}
\end{equation}
which is similar to that in \cite{Bartz2022AdaptiveADMM} when $m=2$. Conversely, a solution of \eqref{eq:inclusion} with $A_i$'s given by \eqref{eq:Ai_for_multiblock} can be shown to correspond to a KKT point for \eqref{eq:multiblock_again}. 
\begin{proposition}
\label{prop:kuhntucker_equivalent_multiblock}
    The inclusion problem \eqref{eq:inclusion} with $A_i$'s given by \eqref{eq:Ai_for_multiblock} has a solution if and only if a KKT point for \eqref{eq:multiblock_again} exists.
\end{proposition}

\subsection{Derivation of a multiblock ADMM}
\label{subsec:derivation_multiblock}

To apply the DR algorithm for \eqref{eq:inclusion} with $A_i$'s given in \eqref{eq:Ai_for_multiblock}, we first provide a formula for obtaining elements of the resolvent of each $A_i$. 

\begin{proposition}
\label{prop:resolvents_Ai_multiblock}
    Let $L$ be a bounded linear operator from $\H'$ to $\H$, $f$ be a proper closed function, and $A \coloneqq (-L) \circ (\widehat{\partial} f)^{-1} \circ \left(-L^*\right)$. For any $x\in \H$ and $\gamma>0$,  
    \[ x+\gamma L\zer \left( \widehat{\partial} f (\cdot) + \gamma L^* \circ \left(L (\cdot) + \frac{x}{\gamma}\right)\right) \subseteq J_{\gamma A}(x).\]
    In particular, suppose $u$ satisfies 
    \begin{equation}
         u\in \SfA{f}{L}{x}  \coloneqq \argmin_{w\in \H'} f(w) + \frac{\gamma}{2}\norm{Lw+\frac{x}{\gamma}}^2.
         \label{eq:prox-L}
    \end{equation}
    Then $ x+  \gamma Lu \in J_{ \gamma A}(x)$. Consequently, if $B\coloneqq A + b$, where $b\in \H$, we have $J_{ \gamma B} (x)=J_{ \gamma A}(x-\gamma b) \ni x-\gamma b + \gamma Lu$, where $ u\in \SfA{f}{L}{x-\gamma b}$.
\end{proposition}
\begin{proof}
    If $u\in \zer \left( \widehat{\partial} f (\cdot) + \gamma L^* \circ \left( L (\cdot) + \frac{x}{\gamma}\right)\right) $, then $0\in \widehat{\partial} f(u) +  L^* z$ where $z = x+ \gamma Lu $ by \cite[Proposition 1.114]{Mordukhovich2006}. Hence, $u\in (\widehat{\partial} f)^{-1} \left( -L^*z\right) $ and so $z \in x +  \gamma L \left( (\widehat{\partial} f)^{-1} \left( -L^*z\right) \right) =x - \gamma A(z)$. Thus, $z\in J_{\gamma A}(x)$, proving the desired inclusion. The last claim is immediate by noting that from the optimality condition of \eqref{eq:prox-L},  a solution $u$ of \eqref{eq:prox-L} is a zero of $ \widehat{\partial} f (\cdot) + \gamma L^* \circ \left( L (\cdot) + \frac{x}{\gamma}\right)$. 
\smartqedmark \end{proof}

Let $\F$ and $\G$ be given by \eqref{eq:F} and \eqref{eq:G}, respectively, where each $A_i$ is given by \eqref{eq:Ai_for_multiblock}. We now apply the doubly relaxed DR algorithm \eqref{eq:aDR_algorithm_bold_switched} with a renumbered iteration count: Given $\x^0$, we let $\y^0 \in J_{\delta \G}(\x^0)$, and  for $k=0,1,2,\dots$, we rewrite \eqref{eq:adr_stepbystep_bold_switched} as 
    \begin{subequations}
    \label{eq:dr_stepbystep2_renumbered}
        \begin{align}
            \z^{k+1} & \in J_{\gamma  \F}((1-\mu)\x^k + \mu \y^k ) \\
            \y^{k+1} & \in J_{\delta \G}(\x^k + \kappa \lambda (\z^{k+1} - \y^k)) \\ 
            \x^{k+1} & = \x^k + \kappa \lambda (\z^{k+1} - \y^k). \label{eq:dr_stepbystep2_renumbered_x}
        \end{align}
    \end{subequations}
Denoting 
    \begin{equation}
    \s^k \coloneqq \y^k - \x^k,
    \label{eq:changevariable}
    \end{equation}
we obtain
   \begin{subequations}
   \label{eq:admm_derivation}
        \begin{align}
            \z^{k+1} & \in J_{\gamma \F}(\y^k+ (\mu-1)\s^k) \\
            \y^{k+1} & \in J_{\delta \G}((1-\kappa \lambda)\y^k + \kappa\lambda \z^{k+1} -\s^k) \\ 
            \s^{k+1} & = \s^k - (1-\kappa\lambda)\y^k   +\y^{k+1} - \kappa \lambda \z^{k+1}.
        \end{align}
    \end{subequations}
Suppose that $S_{f_i,L_i}(x;\gamma)\neq 
\emptyset$ for any $x$  and $\gamma>0$. Using formulas  \eqref{eq:Fresolvent}-\eqref{eq:Gresolvent} and \cref{prop:resolvents_Ai_multiblock}, we derive \cref{alg:admm_general} as an instance of the doubly relaxed DR algorithm \eqref{eq:admm_derivation}.

\begin{algorithm}[tb]
    Input initial point $(y^0,s^0_1,\dots,s^0_{m-1})\in \H^{m}$, parameters $(\gamma,\delta,\lambda,\mu)\in \Re_{++}^2 \times (1,+\infty)^2$ satisfying \eqref{eq:parameters_basic_requirement}, and $\kappa \in (0,1)$. 
      \begin{subequations}
    \label{eq:admm_derivation_stepbystep}
    \begin{align}
            u_i^{k+1} & \ds \in  \argmin_{w_i\in \H_i} f_i(w_i) + \frac{\gamma}{2}\norm{L_iw_i + \frac{y^k + (\mu-1)s_i^k}{\gamma}}^2  \quad (i=1,\dots,m-1) \label{eq:admm_u_update}\\
            z_i^{k+1} & = y^k + (\mu-1)s_i^k + \gamma L_i u_i^{k+1}   \quad (i=1,\dots,m-1) \label{eq:admm_z_update} \\ 
            v_i^{k+1} & = (1-\kappa\lambda)y^k + \kappa\lambda
            z_i^{k+1} - s_i^k \quad (i=1,\dots,m-1) \label{eq:admm_v_update} \\
            u_m^{k+1} & \ds \in  \argmin_{w_m\in \H_m} f_m(w_m) + \frac{\delta}{2(m-1)} \norm{L_m w_m + \frac{1}{\delta}\sum_{j=1}^{m-1} v_j^{k+1} -b}^2 \label{eq:admm_um_update}\\
            y^{k+1} & \ds = \frac{1}{m-1}\sum_{j=1}^{m-1}  v_j^{k+1}- \frac{\delta}{m-1} b + \frac{\delta}{m-1} L_m u_m^{k+1} \label{eq:admm_y_update}\\ 
		  s_i^{k+1} & = s_i^k - (1-\kappa \lambda)y^k + y^{k+1} - \kappa\lambda z_i^{k+1} \quad (i=1,\dots,m-1) . \label{eq:admm_s_update}
    \end{align}
    \end{subequations}
	\caption{Multiblock splitting algorithm}
	\label{alg:admm_general}
\end{algorithm}

A special instance of \cref{alg:admm_general} can be obtained by setting $\kappa\in (0,1)$ as $ \kappa \coloneqq \frac{1}{\lambda(\mu-1)} = \frac{\lambda-1}{\lambda}. $ To see this, 
note from \eqref{eq:admm_z_update} and \eqref{eq:admm_v_update} that 
 \[\sum_{j=1}^{m-1}v_j^{k+1} = (m-1)y^k  + \kappa \lambda \gamma \sum_{j=1}^{m-1} L_ju_j^{k+1} + (\kappa \lambda (\mu-1) - 1)\sum_{j=1}^{m-1}s_j^k.\] 
Substituting $\kappa$ and noting \eqref{eq:parameters_basic_requirement_comonotone}, we see that $\kappa \lambda \gamma = \delta$ and that the last term above vanishes. Further calculations lead to the following iterations:
 \begin{align}
            u_i^{k+1} & \in  \ds \argmin_{w_i\in \H_i} f_i(w_i) + \frac{\gamma}{2}\norm{L_iw_i + \frac{y^k + (\mu-1)s_i^k}{\gamma}}^2 \quad (i=1,\dots, m-1) \label{eq:admm_ui_specialcase} \\
            u_m^{k+1} & \ds \in  \argmin_{w_m\in \H_m} f_m(w_m) + \frac{\delta}{2(m-1)} \norm{L_m w_m + \frac{m-1}{\delta}y^k + \sum_{j=1}^{m-1}L_ju_j^{k+1} -b }^2 \label{eq:admm_um_specialcase}\\
           y^{k+1} & \ds = y^k + \frac{\delta}{m-1}\left( \sum_{j=1}^{m} L_ju_j^{k+1} - b \right) \label{eq:admm_y_specialcase} \\ 
		  s_i^{k+1} & \ds = \frac{\delta}{m-1} \left( \sum_{j=1}^{m} L_ju_j^{k+1} - b \right) - \delta L_iu_i^{k+1}  \quad (i=1,\dots, m-1)  \label{eq:admm_si_specialcase}
\end{align}
Plugging \eqref{eq:admm_si_specialcase} into \eqref{eq:admm_ui_specialcase} and noting \eqref{eq:parameters_basic_requirement_comonotone}, we can simplify the $u_i$-update as 
    \begin{equation}
        u_i^{k+1} \in  \ds \argmin_{w_i\in \H_i} f_i(w_i) + \frac{\gamma}{2}\norm{L_iw_i - L_iu_i^k + \frac{v^k}{m-1}}^2 , \label{eq:admm_ui_specialcase2}
    \end{equation}
for $i=1,\dots,m-1$, where $v^k \coloneqq \sum_{j=1}^{m} L_ju_j^{k} - b  + \frac{y^k(m-1)}{\gamma}$. Meanwhile, note that 
\begin{align*}
   & \norm{L_iw_i -L_iu_i^k + \frac{v^k}{m-1} }^2 \\
   & = \frac{1}{(m-1)^2} \norm{(m-1)(L_iw_i-L_iu_i^k) + v^k}^2 \\
   & = \frac{1}{m-1} \left( (m-1)\norm{L_iw_i-L_iu_i^k}^2  + 2\inner{L_iw_i-L_iu_i^k}{v^k}+\frac{\norm{v^k}^2}{m-1}\right)\\
   & = \frac{1}{m-1} \norm{L_iw_i-L_iu_i^k + v^k}^2 + \frac{m-2}{m-1}\norm{L_iw_i-L_iu_i^k}^2 + \frac{2-m}{(m-1)^2}\norm{v^k}^2.
\end{align*}

With this, we can simplify the algorithm as shown in \cref{alg:admm_specialcase}. Note that \cref{alg:admm_specialcase} reduces to the standard ADMM algorithm when $m = 2$; see, for example, \cite[Section 15.2]{Beck17}. For $m > 2$, \cref{alg:admm_specialcase} differs from the more commonly used ``natural'' extensions of the two-block ADMM, such as those proposed in \cite{Chen2016ADMM,wang2019global}; see also \eqref{eq:GS-ADMM}. In particular, in \cref{alg:admm_specialcase}, the updates for $u_i^{k+1}$ for $i = 1, \dots, m{-}1$ can be computed in parallel.

\begin{algorithm}
    Input initial point $(u^0_1,\dots,u^0_{m})\in \H^1 \times \cdots \times \H^m$, $y^0\in \H$, and parameters $(\gamma,\delta,\lambda,\mu)\in \Re_{++}^2 \times (1,+\infty)^2$ satisfying \eqref{eq:parameters_basic_requirement}. Set $\gamma ' \coloneqq \frac{\gamma}{m-1}$ and $\delta' \coloneqq \frac{\delta}{m-1}$. \\
    For $k=0,1,2,\dots ,$
    \begin{align*}
             u_i^{k+1} & \in \ds \argmin_{w_i\in \H_i} f_i(w_i) + \frac{\gamma'}{2}\norm{L_iw_i  +\sum_{\substack{j = 1 \\ j \neq i}}^{m} L_ju_j^{k} - b  + \frac{y^k}{\gamma'} }^2 +\frac{\gamma' (m-2)}{2}\norm{L_i(w_i-u_i^k)}^2 \\& \quad (i=1,\dots,m-1) \\ 
            u_m^{k+1} & \ds \in  \argmin_{w_m\in \H_m} f_m(w_m) + \frac{\delta'}{2} \norm{L_m w_m + \sum_{j=1}^{m-1}L_ju_j^{k+1} -b + \frac{y^k}{\delta'}  }^2 \\
           y^{k+1} & \ds = y^k + \delta' \left( \sum_{j=1}^{m} L_ju_j^{k+1} - b \right)
    \end{align*}
	\caption{Multiblock splitting algorithm with $\kappa = \frac{\lambda-1}{\lambda}$}
	\label{alg:admm_specialcase}
\end{algorithm}

The following describes the fixed points of $T_{\G,\F}$, which is important to interpret the convergence result for the DR algorithm \eqref{eq:dr_stepbystep2_renumbered}.
\begin{proposition}
\label{prop:fixedpoints_admm}
  Let $\F$ and $\G$ be given by \eqref{eq:F} and \eqref{eq:G}, respectively, where each $A_i$ is given by \eqref{eq:Ai_for_multiblock}, and let $(\gamma,\delta,\lambda,\mu)\in\Re^4_{++}$ be parameters that satisfy \eqref{eq:parameters_basic_requirement_comonotone}. Suppose that $J_{\gamma \F}$ and $J_{\delta \G}$ are single-valued, and  $S_{f_m,L_m}(\cdot ;\frac{\delta}{m-1})$ and $S_{f_i,L_i}(\cdot ;\gamma)$ are nonempty-valued for $i=1,\dots,m-1$. If $\bar{\x}\in \Fix(T_{\G,\F})$, where $T_{\G,\F}$ is given by \eqref{eq:aDR_algorithm_bold_switched}, and $\bar{\y} \coloneqq J_{\delta \G}(\bar{\x})$, then the following hold:
   \begin{enumerate}[(i)]
       \item $\bar{\y} =(\bar{y},\dots,\bar{y})$ where 
     \begin{align}
        \bar{y} & \coloneqq \frac{1}{m-1}\sum_{i=1}^{m-1}\bar{x}_i - \frac{\delta}{m-1} b + \frac{\delta}{m-1}L_m\bar{u}_m \label{eq:y_admm}\\ 
         \bar{u}_m & \coloneqq  \argmin _{w_m\in \H_m}f_m(w_m) + \frac{\delta}{2(m-1)} \norm{L_mw_m + \frac{1}{\delta}\sum_{i=1}^{m-1}\bar{x}_i-b}^2.
         \label{eq:u_m}
     \end{align}
     \item For each $i=1,\dots,m-1$, denote
        \begin{equation}
            \bar{u}_i \coloneqq  \argmin_{w_i\in \H_i} f_i(w_i) + \frac{\gamma}{2}\norm{L_iw_i + \frac{(1-\mu)\bar{x}_i + \mu \bar{y}}{\gamma}}^2.
            \label{eq:u_i}
        \end{equation}
        Then $\bar{y}$ given by \eqref{eq:y_admm} is also given by $\bar{y}= \bar{x}_i - \delta L_i\bar{u}_i$ for any $i=1,\dots,m-1$. 
    \item $(\bar{u}_1,\dots,\bar{u}_m,\bar{y})$ is a KKT point of \eqref{eq:multiblock_again}.
   \end{enumerate}
\end{proposition}
\begin{proof}
    Part (i) follows from \cref{prop:resolvents_Ai_multiblock} and \eqref{eq:Gresolvent}. Since $\bar{\x}$ is a fixed point, $\bar{\y} = J_{\gamma \F}((1-\mu)\bar{\x}+\mu \bar{\y})$. With this, part (ii) is a consequence of \cref{prop:resolvents_Ai_multiblock}, \eqref{eq:Fresolvent} and \eqref{eq:parameters_basic_requirement_comonotone}. To prove (iii), we note that from the optimality conditions of \eqref{eq:u_m} and \eqref{eq:u_i}, it can be shown that $0\in \widehat{\partial}f_i(\bar{u}_i) + L_i^*\bar{y}$ for all $i=1,\dots,m-1$. Finally, we have 
    \[ \sum_{i=1}^{m-1} \bar{x}_i -\delta \sum_{i=1}^{m-1}  L_i\bar{u}_i =  (m-1) \bar{y} = \sum_{i=1}^{m-1} \bar{x}_i - \delta b + \delta L_m \bar{u}_m,\]
    where the first equality follows from part (ii), and the second holds by part (i). Hence, $\sum_{i=1}^{m} L_i\bar{u}_i = b$. This completes the proof. 
\smartqedmark \end{proof}

\subsection{Convergence results}

Note that \cref{alg:admm_general} is a generic algorithm for \eqref{eq:multiblock_again} that is well-defined provided that the minimization problems defining $u_i^{k+1}$ have solutions. We now establish its convergence using \cref{thm:adaptiveDR,thm:adaptiveDR_improved}. First, we provide sufficient conditions on $f_i$ and $L_i$ to guarantee comonotonicity of $A_i$ in \eqref{eq:Ai_for_multiblock}. 

Recall that $f$ is said to be $\rho$-convex for $\rho\in\Re$ if $f-\frac{\rho}{2}\norm{\cdot}^2$ is convex. Moreover, if $f$ is $\rho$-convex, then $ \widehat{\partial} f = \partial f$ and $\partial f$ is a maximally $\rho$-monotone operator (for instance, see \cite[Lemma 5.1]{AlcantaraTakeda2025}). 

\begin{proposition}
    \label{prop:rhoconvex_implies_comonotone}
   Let $f:\H'\to (-\infty,+\infty]$ be a proper closed $\rho$-convex function with $\rho\in \Re$, $L:\H'\to\H$ be a bounded linear operator, and define $A\coloneqq (-L)\circ (\partial f)^{-1} \circ (-L^*)$. 
    \begin{enumerate}[(i)]
        \item If $\rho\geq 0$, then $A$ is $\frac{\rho}{\norm{L}^2}$-comonotone. If $0\in \sri (\dom f^* - \ran L^*)$, then $A$ is maximally $\frac{\rho}{\norm{L}^2}$-comonotone\footnote{$\sri (C)$ denotes the \textit{strong relative interior of a convex set $C$}, which coincides with the {relative interior of $C$} when $\H$ is finite-dimensional. See \cite[Definition 6.9]{Bauschke2017}. Sufficient conditions for $0\in \sri (\dom (f^*) - \ran (L^*))$ are given in \cite[Lemma 4.9]{Bartz2022AdaptiveADMM}.}.
        \item If $\rho<0$ and $L^*L$ is invertible, then $A$ is maximally $\rho \norm{(L^*L)^{-1}}$-comonotone. 
    \end{enumerate}
  
\end{proposition}
\begin{proof}
The proof of part (i) is analogous to that of \cite[Lemmas 4.7 and 4.10]{Bartz2022AdaptiveADMM}. To prove (ii), let $(x,u),(y,v)\in \gra (A)$. Then there exist $c\in (\partial f)^{-1}(-L^*x)$ and $d\in (\partial f)^{-1}(-L^*y)$ such that $u=-Lc$ and $v=-Ld$. By the $\rho$-monotonicity of $\partial f$, we have 
    \begin{align*}
        \inner{x-y}{u-v} = \inner{-L^*x+L^*y}{c-d} \geq \rho \norm{c-d}^2 .
    \end{align*}
On the other hand,\footnote{Note that the inequality holds in an arbitrary (infinite-dimensional) Hilbert space.}
\[\norm{u-v}^2 = \norm{L(c-d)}^2 = \inner{c-d}{L^*L(c-d)} \geq \frac{1}{\norm{(L^*L)^{-1}}}\norm{c-d}^2. \]
Since $\rho<0$, it immediately follows that $A$ is $\rho \norm{(L^*L)^{-1}}$-comonotone. To show maximality, it suffices to show by \cref{lemma:maximal_single-valued-resolvent-comonotone}(ii) that $\dom (J_{\gamma A}) = \H$ when $\gamma + \rho \norm{(L^*L)^{-1}} >0$. To this end, note first that 
\[ f(w) + \frac{\gamma}{2}\norm{Lw+\frac{x}{\gamma}}^2= \left(f(w) - \frac{\rho}{2}\norm{w}^2 \right) + \underbrace{\left( \frac{\rho}{2}\norm{w}^2 + \frac{\gamma}{2}\norm{Lw+\frac{x}{\gamma}}^2\right)}_{\eqqcolon g(w)}.\]
Note that $g$ is $\rho+\frac{\gamma}{\norm{(L^*L)^{-1}}}$-strongly convex by \cite[Exercise 17.5]{Bauschke2017}, provided that $\gamma + \rho \norm{(L^*L)^{-1}} >0$. Hence, since $f$ is $\rho$-convex, it follows that the objective function in $S_{f,L}(x;\gamma)$ defined in \eqref{eq:prox-L} is strongly convex. By \cite[Corollary 11.17]{Bauschke2017},  $S_{f,L}(x;\gamma)$ is a singleton, and therefore $J_{\gamma A}(x) \neq \emptyset$ by  \cref{prop:resolvents_Ai_multiblock}. This completes the proof.  
\smartqedmark \end{proof}

We now present sufficient conditions under which the first part of \cref{prop:resolvents_Ai_multiblock} holds with equality, enabling the application of \cref{prop:fixedpoints_admm}.
\begin{lemma}\label{lemma:subproblem_existenceofsolutions} Let $f:\H'\to (-\infty,+\infty]$ be a proper closed $\rho$-convex function with $\rho\in \Re$, and $L:\H'\to\H$ be a bounded linear operator. Given $x\in \H$ and $\gamma>0$, let  $S_{f,L}(x;\gamma)$ be given by \eqref{eq:prox-L}. Then the following hold: 
    \begin{enumerate}[(i)]
        \item If $\rho>0$, then $S_{f,L}(x;\gamma)$ is a singleton.
        \item If $\rho=0$, and either $f$ is coercive or $L^*L$ is invertible, then $S_{f,L}(x;\gamma)$ is nonempty.
        \item If $\rho<0$, $L^*L$ is invertible, and $\gamma + \rho \norm{(L^*L)^{-1}} >0$, then $S_{f,L}(x;\gamma)$ is a singleton.
    \end{enumerate}
\end{lemma}
\begin{proof}
Part (i) follows from \cite[Corollary 11.17]{Bauschke2017}, while part (ii) is proved in \cite[Lemma 6.1]{MalitskyTam2023}. Part (iii) was already proved in the proof of \cref{prop:rhoconvex_implies_comonotone}(ii).
\smartqedmark \end{proof}

\begin{proposition}
\label{prop:resolvents_for_admm_equality}
    Let $f:\H'\to (-\infty,+\infty]$ be a proper closed $\rho$-convex function with $\rho\in \Re$, $L:\H'\to\H$ be a bounded linear operator, and define $A\coloneqq (-L)\circ (\partial f)^{-1} \circ (-L^*)$. Let $x\in \H$, $\gamma >0$, and $S_{f,L}(x;\gamma)$ be given by \eqref{eq:prox-L}.  Then $ x+ \gamma LS_{f,L}(x;\gamma) =J_{\gamma A}(x)$ 
    under any  of the following conditions:
    \begin{enumerate}[(i)]
        \item $\rho>0$ and $0\in \sri(\dom f^* - \ran L^*)$.
        \item $\rho=0$, $0\in \sri(\dom f^* - \ran L^*)$, and either $f$ is coercive or $L^*L$ is invertible. 
        \item $\rho<0$, $L^*L$ is invertible, and $\gamma + \rho \norm{(L^*L)^{-1}} >0$.
    \end{enumerate}
    
\end{proposition}
\begin{proof}
    Under the conditions in (i), (ii) or (iii), we have from \cref{prop:rhoconvex_implies_comonotone} and \cref{lemma:maximal_single-valued-resolvent-comonotone} that $J_{\gamma A}(x)$ is a singleton. Meanwhile,  \cref{prop:resolvents_Ai_multiblock} gives $x+ \gamma LS_{f,L}(x;\gamma) \subseteq J_{\gamma A}(x)$. \cref{lemma:subproblem_existenceofsolutions}, on the other hand, implies that $x+ \gamma LS_{f,L}(x;\gamma)$ is nonempty. Hence, equality holds.
\smartqedmark \end{proof}

Equipped with the above results, we summarize all the assumptions we need to establish the convergence of \cref{alg:admm_general}.

\begin{assumption}\label{assume:admm} The following hold:
\begin{enumerate}[(A)]
    \item A KKT point for \eqref{eq:multiblock_again} exists.
    \item  $f_i$ is a proper closed $\rho_i$-convex function, $i=1,\dots,m$, such that $\rho_i\geq 0$ for all $i=1,\dots,m-1$, and $\rho_m\leq 0$.
    \item $L_i:\H_i\to \H$ is a bounded linear operator such that  $L_m^*L_m$ is invertible if $\rho_m<0$.
    \item For any $i\in \{1,\dots,m\}$, $0\in \sri (\dom f_i^* - \ran L_i^*)$. If $\rho_i=0$, either $f_i$ is coercive or $L_i^*L_i$ is invertible.
\end{enumerate}
\end{assumption}

\begin{theorem}[Convergence of multiblock ADMM]
	\label{thm:admm1_convergence}
	Suppose that Assumption~\ref{assume:admm} holds, and let $\F$ and $\G$ be given by \eqref{eq:F} and \eqref{eq:G}, respectively, where each $A_i$ is given by \eqref{eq:Ai_for_multiblock}. Define the parameters 
	\begin{equation}
		\sigma_i \coloneqq \begin{cases}
			\frac{\rho_i}{\norm{L_i}^2} & \text{if}~i\in \{1,\dots,m-1\} \\
			0 & \text{if}~i=m~\text{and}~\rho_m = 0 \\
			\rho_m\norm{(L_m^*L_m)^{-1}} & \text{if}~i=m~\text{and}~\rho_m<0
		\end{cases},
		\label{eq:sigma_i}
	\end{equation}
	and denote $\underline{\sigma} \coloneqq \min_{1\leq i\leq m-1}\sigma_i$. Let $(\gamma,\delta,\lambda,\mu)\in \Re_{++}^2 \times (1,+\infty)^2$ be parameters satisfying \eqref{eq:parameters_basic_requirement_comonotone}, \alertt{let $\kappa\in(0,1)$,} and suppose that any one of the conditions (C1), (C2) and (C3) in \cref{thm:adaptiveDR_combined} holds. \alertt{Then} \cref{alg:admm_general} is well-defined. 
	
	Moreover, for any sequence $\{(u_1^k,\dots,u_m^k,y^k,s_1^k,\dots,s_{m-1}^k)\}$ generated by \cref{alg:admm_general} with any initial point $(s_1^0,\dots,s_{m-1}^0,y^0)\in \H^m$, the following hold: 
	\begin{enumerate}[(i)]
		\item Denote $\x^k \coloneqq \y^k - \s^k$ where $\y^k = (y^k,\dots,y^k)$ and $\s^k = (s_1^k,\dots,s_{m-1}^k)$. There exists $\bar{\x}\in \Fix (\dr{\G}{\F})$ such that $\bar{\y}\coloneqq J_{\delta \G}(\bar{\x})\in \zer (\F + \G)$ and $\x^k\toweak \bar{\x}$. Specifically, if $\kappa = \frac{\lambda-1}{\lambda}$, then $x_i^{k+1} = y^k + \delta L_iu_i^{k+1}\toweak \bar{x}_i$ for $i=1,\dots,m-1$. 
		
		\item $\norm{\s^{k+1} - \s^k}  = o(1/\sqrt{k})$ and $\norm{y^{k+1}-y^k} = o (1/\sqrt{k})$ as $k\to \infty$. Specifically, if $\kappa = \frac{\lambda-1}{\lambda}$, then $\norm{\sum_{j=1}^m L_ju_j^k-b}=o(1/\sqrt{k})$ as $k\to\infty$; and  
		
		\item $\{y^k\}$ converges weakly to $\bar{y}$, where $\bar{y}$ is given by \eqref{eq:y_admm}.
	\end{enumerate}
	In addition, if $\bar{u}_i$ is given by \eqref{eq:u_m}-\eqref{eq:u_i} for $i=1,\dots,m$, then
	\begin{enumerate}[label=(\roman*), start=4]
		\item $(\bar{u}_1,\dots,\bar{u}_m,\bar{y})$ is a KKT point; 
		
		\item For each $i=1,\dots,m$, $L_iu_i^k \toweak L_i\bar{u}_i$; and 
		
		\item \alertt{If (C3) holds, or if (C2) holds and either $\kappa^*\geq1$, or $\gamma=\delta>0$ and $\lambda=\mu=2$, then} $L_iu_i^{k}\to L_i \bar{u}_i$ for each $i=1,\dots,m$. Hence, if $L_i^*L_i$ is invertible, then $u_i^k\to \bar{u}_i$.
	\end{enumerate}
\end{theorem}
\begin{proof}
Using Assumption~\ref{assume:admm}(B) and (D), \cref{prop:resolvents_for_admm_equality}(i) and (ii) \alertt{imply} that $\Id + \gamma L_i S_{f_i,L_i}(\cdot;\gamma) = J_{\gamma A_i}$. \alertt{Under each of the conditions (C1)--(C3), we have $\delta+(m-1)\sigma_m>0$. In particular, under (C3), this follows from $\min_{1\leq i\leq m-1}\kappa_i^*>\kappa>0$, as shown in the proof of \cref{thm:adaptiveDR_improved}. Hence,} we obtain from Assumption~\ref{assume:admm}(B) and (C), together with \cref{prop:resolvents_for_admm_equality}(iii), that $\Id + \frac{\delta}{m-1} L_m S_{f_m,L_m}(\cdot\alertt{-\frac{\delta}{m-1}b};\frac{\delta}{m-1})\alertt{-\frac{\delta}{m-1}b} = J_{\frac{\delta}{m-1} A_m}$. Hence, algorithm \eqref{eq:adr_stepbystep_bold_switched} is well-defined and can be written in the form of \cref{alg:admm_general}, as derived in \cref{subsec:derivation_multiblock}.

In addition, \cref{prop:kuhntucker_equivalent_multiblock} and Assumption~\ref{assume:admm}(A) ensure that $\zer(A_1 + \cdots + A_m) \neq \emptyset$. Under Assumption~\ref{assume:admm}(B)--(D), \cref{prop:rhoconvex_implies_comonotone} implies that each $A_i$ is maximally $\sigma_i$-comonotone, where $\sigma_i$ is given in \eqref{eq:sigma_i}. Hence, all the conditions of \cref{thm:adaptiveDR_combined} are satisfied, allowing us to apply the theorem to conclude the convergence of \eqref{eq:adr_stepbystep_bold_switched}, and equivalently of \cref{alg:admm_general}; see also \cref{remark:switched_ADR}. Specifically, part (i) of the theorem, together with \eqref{eq:changevariable}, proves part (i); part (ii) follows from \cref{thm:adaptiveDR_combined}(ii), \eqref{eq:admm_y_specialcase}, and the triangle inequality; and part (iii) is a direct consequence of \cref{thm:adaptiveDR_combined}(iii).

The claim in part (iv) follows from part (i) and \cref{prop:fixedpoints_admm}(iii). To prove part (v), note from \eqref{eq:admm_z_update} and \eqref{eq:parameters_basic_requirement_comonotone} that $
L_i u_i^{k+1} = -\frac{1}{\delta} s_i^k + \frac{1}{\gamma}(z_i^{k+1} - y^k).$
Since $\norm{\z^{k+1} - \y^k} \to 0$ by \eqref{eq:dr_stepbystep2_renumbered_x} and \cref{thm:adaptiveDR_combined}(ii), and $\s^k = \y^k - \x^k \toweak \bar{\y} - \bar{\x}$, it follows that $L_i u_i^{k+1} \toweak -\frac{1}{\delta}(\bar{y} - \bar{x}_i) = L_i \bar{u}_i$ for each $i = 1, \dots, m-1$, where the last equality follows from \cref{prop:fixedpoints_admm}. Together with part (ii) and the fact that $\sum_{i=1}^m L_i \bar{u}_i = b$ from part (iv), we conclude that $L_m u_m^{k+1} \toweak L_m \bar{u}_m$. This completes the proof of part (v). The proof of part (vi) follows similarly, noting that in this case, \cref{thm:adaptiveDR_combined}(iv) guarantees strong convergence of the sequence $\{\s^k\}$.
\smartqedmark \end{proof}

In the fully convex case, \cref{thm:admm1_convergence}(i)--(v) establishes convergence of the proposed multiblock ADMM in \cref{alg:admm_general,alg:admm_specialcase}, in contrast to \eqref{eq:GS-ADMM} \cite{Chen2016ADMM}.

\begin{remark}[\alert{Comparison with related multiblock ADMM schemes}]
	\label{remark:relatedworks}
	\alert{Algorithm~3 is closely related to the block-wise ADMM proposed in
		\cite{HeYuan2015}, but is derived here from the multioperator
		Douglas--Rachford framework. It also differs from the classical
		Gauss--Seidel ADMM (GS-ADMM), which updates all primal blocks sequentially;
		see, for example, \cite{wang2019global}. In
		\cref{alg:admm_specialcase}, the first $m-1$ block updates can instead be
		computed in parallel, followed by the update of the final block and the
		multiplier.}
	
	\alert{Related multiblock ADMM schemes are generally analyzed under
		different structural assumptions. Jacobi and proximal ADMM variants
		typically introduce additional proximal or correction terms and have been
		primarily studied for convex multiblock problems; see, for example,
		\cite{Deng2017,HeXuYuan2016}. Nonconvex multiblock ADMM results, such as
		\cite{wang2019global}, allow broader forms of nonconvexity but rely on
		different regularity and linear-operator assumptions and establish
		convergence to stationary or critical points.}
	
	\alert{The closest GS-ADMM result for the strongly convex--weakly convex
		setting considered here is \cite{ChenCuiHan2025}. That result is restricted
		to three blocks, with the weakly convex block updated first and the two
		strongly convex blocks updated afterward, and imposes different rank and
		penalty-parameter conditions. In contrast, \cref{thm:admm1_convergence}
		applies to an arbitrary number of blocks, with the first $m-1$ functions
		strongly convex and the final function weakly convex, while allowing the
		functions to be nonsmooth. To the best of our knowledge, this provides the
		first convergence guarantee for a multiblock ADMM-type method with an
		arbitrary number of blocks in the nonsmooth strongly convex--weakly convex
		setting considered here. Since these methods rely on different update
		structures and assumptions, their admissible parameter ranges are not
		directly ordered.}
\end{remark}

\subsection{\alert{Numerical examples}}

\alert{The purpose of this section is to illustrate the behavior of
	\cref{alg:admm_specialcase} in the strongly convex--weakly convex setting.
	The experiments are not intended as a comprehensive computational or
	scalability study, nor do we claim general numerical superiority over existing algorithms.}

\alert{Throughout this section, the parameters of Algorithm~3 are chosen to satisfy
the sufficient conditions of the convergence theory rather than tuned to
optimize numerical performance. More generally, the relevant
comonotonicity moduli are problem-dependent and may be difficult to estimate
in practice. Developing adaptive parameter-selection strategies with
convergence guarantees is an interesting direction for future research.}

\paragraph{Implementation and stopping criterion.}
All experiments were implemented in Python~3.12.10 and were conducted on a computer equipped
with an Intel(R) Core(TM) i9-13900H CPU @ 2.60 GHz and 32 GB of RAM. \ifdefined\submit \else The code is available at
\url{https://github.com/jhalcantara/admm_multiblock}.\fi 

We terminate an algorithm when it generates a point
$(u_1^k,\dots,u_m^k,y^k)$ that is an approximate KKT point in the sense of
\cref{defn:kuhntucker}; namely,
\begin{equation}
	\max \left\lbrace
	\max_{i=1,\dots,m}
	\dist\left(0,\partial f_i(u_i^k)+L_i^\top y^k\right),
	\norm{\sum_{i=1}^{m}L_iu_i^k-b}
	\right\rbrace
	\leq \epsilon.
	\label{eq:stopcriterion}
\end{equation}
For \cref{alg:admm_specialcase}, the optimality conditions of the block
subproblems and the multiplier update imply that
\[
\begin{cases}
	0\in\partial f_i(u_i^{k+1})+L_i^\top y^{k+1}+s_i^{k+1},
	\quad i=1,\dots,m-1, \\ 
	0\in\partial f_m(u_m^{k+1})+L_m^\top y^{k+1},
\end{cases}
\]
where
\[
s_i^{k+1}
\coloneqq
\gamma L_i^\top L_i(u_i^{k+1}-u_i^k)
+\frac{L_i^\top}{m-1}
\left(
\sum_{j=1}^{m}L_j(\gamma u_j^k-\delta u_j^{k+1})
+(\delta-\gamma)b
\right).
\]
Consequently, \eqref{eq:stopcriterion} holds whenever
\begin{equation}
	\max\left\lbrace
	\norm{s_1^k},\dots,\norm{s_{m-1}^k},
	\norm{\sum_{i=1}^{m}L_iu_i^k-b}
	\right\rbrace
	\leq\epsilon.
	\label{eq:stopcriterion_simplified}
\end{equation}

For comparison, we also consider the classical Gauss--Seidel multiblock ADMM
(GS-ADMM), given by
\begin{align}
	u_i^{k+1}
	&=
	\arg\min_{w_i\in\H_i}
	f_i(w_i)
	+\frac{\gamma'}{2}
	\norm{
		\sum_{j=1}^{i-1}L_ju_j^{k+1}
		+L_iw_i
		+\sum_{j=i+1}^{m}L_ju_j^k
		-b+\frac{y^k}{\gamma'}
	}^2,
	\quad i=1,\dots,m,
	\notag\\
	y^{k+1}
	&=
	y^k+\gamma'
	\left(\sum_{i=1}^{m}L_iu_i^{k+1}-b\right).
	\tag{GS-ADMM}
	\label{eq:GS-ADMM}
\end{align}
For this algorithm, note that 

\[
0\in
\partial f_i(u_i^k)+L_i^\top y^k+s_{i,\mathrm{GS}}^k,
\quad i=1,\dots,m-1,
\]
where
$
s_{i,\mathrm{GS}}^k
\coloneqq
\gamma'
L_i^\top
\left(
\sum_{j=i+1}^{m}
L_j(u_j^{k-1}-u_j^k)
\right)$.
The last block optimality condition gives
$0\in\partial f_m(u_m^k)+L_m^\top y^k$. Thus,
\eqref{eq:stopcriterion_simplified}, with $s_i^k$ replaced by
$s_{i,\mathrm{GS}}^k$, also implies \eqref{eq:stopcriterion}.
Although the direct multiblock extension \eqref{eq:GS-ADMM} is widely used,
it is not guaranteed to converge in general, even for convex multiblock
problems~\cite{Chen2016ADMM}. We therefore include it only as an empirical
reference method.

\subsubsection{\alert{Denoising}}

We adopt the denoising problem from~\cite{Bartz2022AdaptiveADMM}, which
includes both strongly and weakly convex components. The observed signal
$\hat{\phi}\in\Re^n$ is modeled as
\begin{equation}
	\hat{\phi}=\phi+\xi,
	\label{eq:signal_with_noise}
\end{equation}
where $\phi\in\Re^n$ is the true signal and
$\xi\sim\mathcal{N}(0,\sigma^2I)$ is Gaussian noise. The goal is to recover
$u\approx\phi$ from the noisy data $\hat{\phi}$. We consider the total
variation regularized problem
\begin{equation}
	\min_{u=(u_1,\dots,u_N)\in
		\Re^{n_1}\times\cdots\times\Re^{n_N}}
	\frac{1}{2N}\sum_{i=1}^N
	\norm{u_i-\hat{\phi}_i}^2
	+\omega P(Du),
	\label{eq:numerical_example}
\end{equation}
where
$\hat{\phi}=(\hat{\phi}_1,\dots,\hat{\phi}_N)$,
$n_1+\cdots+n_N=n$, $\omega>0$, and
$D\in\Re^{(n-1)\times n}$ is the first-order difference matrix, i.e.,
\[
D_{ij}
=
\begin{cases}
	1,  & \text{if }j=i,\\
	-1, & \text{if }j=i+1,\\
	0,  & \text{otherwise}.
\end{cases}
\]
As in~\cite{Bartz2022AdaptiveADMM}, we use the minimax concave penalty
\[
P(w)
\coloneqq
P_\tau(w)
=
\sum_{i=1}^{n-1}p_\tau(w_i),
\qquad
p_\tau(t)
=
\begin{cases}
	|t|-\dfrac{t^2}{2\tau}, & \text{if }|t|\leq\tau,\\
	\dfrac{\tau}{2}, & \text{otherwise}.
\end{cases}
\]
Letting $u_{N+1}=Du$, problem~\eqref{eq:numerical_example} can be written
in the form~\eqref{eq:multiblock_again} with $m=N+1$,
\[
f_i(u_i)
=
\begin{cases}
	\dfrac{1}{2N}\norm{u_i-\hat{\phi}_i}^2,
	& \text{if }1\leq i\leq N,\\
	\omega P_\tau(u_{N+1}),
	& \text{if }i=N+1,
\end{cases}
\qquad
L_i
=
\begin{cases}
	D_i, & \text{if }1\leq i\leq N,\\
	-I_{n-1}, & \text{if }i=N+1,
\end{cases}
\]
where $D=[D_1\ \cdots\ D_N]$ and
$D_i\in\Re^{(n-1)\times n_i}$. Note that $f_i$ is
$\rho_i$-convex with $\rho_i=1/N$ for $i=1,\dots,N$, while
$\rho_{N+1}=-\omega/\tau$.

\paragraph{Problem data.}
\alert{We generate the piecewise constant signal in Python using NumPy. In
	particular, we use the normalized change-point locations
$
(0.10,0.13,0.15,0.23,0.25,0.40,0.44,0.65,0.76,0.78)
$
and the corresponding jump heights
$
(4,-5,3,-4,5,-4,4,-2,4,-5).
$
This produces a blocky signal with ten changes of varying magnitudes.}
We add Gaussian noise according to~\eqref{eq:signal_with_noise}, with
$\sigma=0.5$. 

\paragraph{Parameter settings.}
\alert{We set $N=5$ (i.e., $m=6$ blocks)}. \alertt{For each noisy
	realization, we choose $
	\omega \coloneqq 0.05\,\hat{\omega},$ where $
	\hat{\omega}
	\coloneqq
	\frac{1}{N}
	\max_{1\leq k\leq n-1}
	\left|
	\sum_{j=1}^{k}
	\left(\hat{\phi}_j-\overline{\hat{\phi}}\right)
	\right|,
	$
	and $\overline{\hat{\phi}}$ denotes the sample mean of the observed
	signal.} Denoting $\alpha
	\coloneqq
	\frac{1}{N}
	\min\left\{
	\norm{D_1}^{-2},\dots,\norm{D_N}^{-2}
	\right\},$ we set $
\tau
\coloneqq
\frac{1.01N\omega}{\alpha}$.
This choice guarantees that
$\theta=(N,\dots,N)\in\Theta$. Indeed,
$\sigma_i+N\sigma_{N+1}>0$ for $i=1,\dots,N$, where
$
\sigma_i=(N\norm{D_i}^2)^{-1},$ and 
$\sigma_{N+1}=-\frac{\omega}{\tau}.$

For the stepsizes, we set $\delta=\eta\gamma$ for some $\eta>0$ and impose
the stronger condition
$\min_{1\leq i\leq N}\kappa_i^*\geq1$.
\alertt{Since
	$\kappa=(\lambda-1)/\lambda=\eta/(1+\eta)<1$, this guarantees that
	$\min_{1\leq i\leq N}\kappa_i^*>\kappa$, as required in
	\cref{thm:adaptiveDR_combined}(C3).}
It can be shown that
\[
\kappa_i^*\geq1
\text{ and }\gamma>0
\quad\Longleftrightarrow\quad
\frac{-2\beta}{\eta-1}
\leq\gamma
\leq
\frac{2\sigma_i}{\eta-1},
\quad
\eta>1,
\qquad
\beta\coloneqq-\frac{N\omega}{\tau}.
\]
Thus,
\[
\frac{-2\beta}{\eta-1}
\leq\gamma
\leq
\frac{2\alpha}{\eta-1}.
\]
For the unequal-stepsize setting, we choose
\begin{equation}
	\gamma
	\coloneqq
	\frac{\alpha-\beta}{\eta-1},
	\qquad
	\delta\coloneqq\eta\gamma,
	\qquad
	\eta\coloneqq1.01.
	\label{eq:unequalstepsize}
\end{equation}
For the equal-stepsize setting, we choose
\begin{equation}
	\gamma
	=
	\delta
	=
	-1.01\frac{2\alpha\beta}{\alpha+\beta}.
	\label{eq:equalstepsize}
\end{equation}

\paragraph{Results.}
We first compare \cref{alg:admm_specialcase} under the equal and unequal
stepsizes in~\eqref{eq:equalstepsize} and~\eqref{eq:unequalstepsize}.
As shown in~\cref{tab:admm_comparison}, the unequal-stepsize variant requires
slightly fewer iterations to reach the prescribed residual, while the two
variants yield comparable mean absolute error,
\[
\operatorname{MAE}
\coloneqq
\frac{1}{n}\sum_{i=1}^n|u_i-\phi_i|.
\]
We also illustrate the residual and MAE plots of
	\cref{alg:admm_specialcase} and GS-ADMM for $n=3000$ \alert{and $n=5000$} in
	\cref{fig:denoising_comparison}. \alert{In these experiments,  \cref{alg:admm_specialcase} converges faster in terms of the residual norm than GS-ADMM. Moreover, it also achieves a lower mean absolute error (MAE) across iterations.}

\begin{table}[tb!]
	\centering
	\caption{{\small Performance comparison of
			\cref{alg:admm_specialcase} with unequal and equal stepsizes over
			10 runs with $n=3000$. The algorithm is terminated when
			\eqref{eq:stopcriterion_simplified} holds with $\epsilon=10^{-4}$.}}
	\begin{tabular}{lccc}
		\hline
		\cref{alg:admm_specialcase}
		& \textbf{Iter}
		& \textbf{Residual}
		& \textbf{MAE} $\boldsymbol{\pm}$ \textbf{std}\\
		\hline
		Unequal stepsizes
		& 1380.5
		& 9.96e-05
		& 0.0641 $\pm$ 0.003642\\
		Equal stepsizes
		& 1401.1
		& 9.97e-05
		& 0.0641 $\pm$ 0.003642\\
		\hline
	\end{tabular}
	\label{tab:admm_comparison}
\end{table}

\begin{figure}[tb!]
	\centering
	\begin{tabular}{@{}cc@{}}
		\includegraphics[scale=0.3]
		{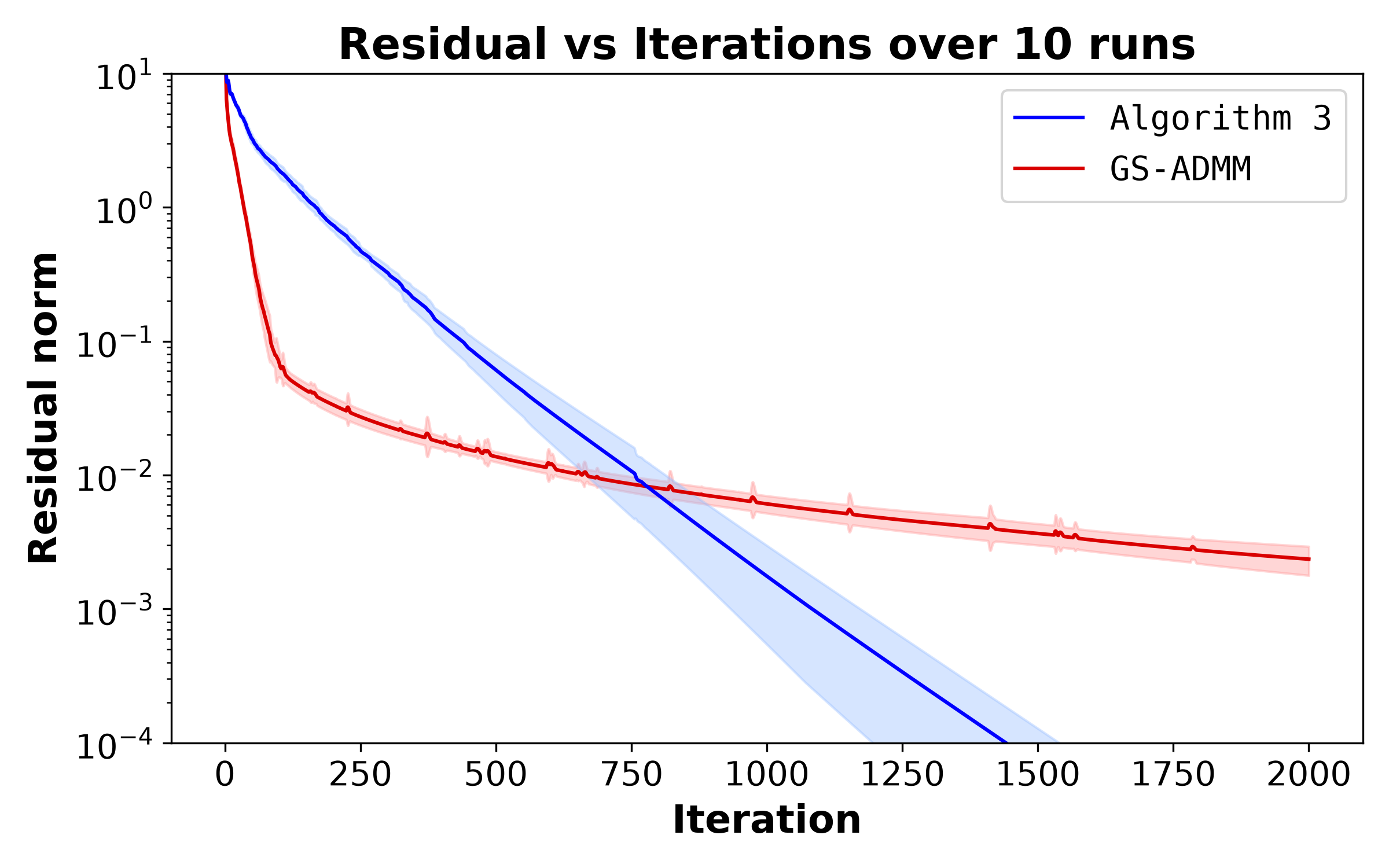}
		&
		\includegraphics[scale=0.3]
		{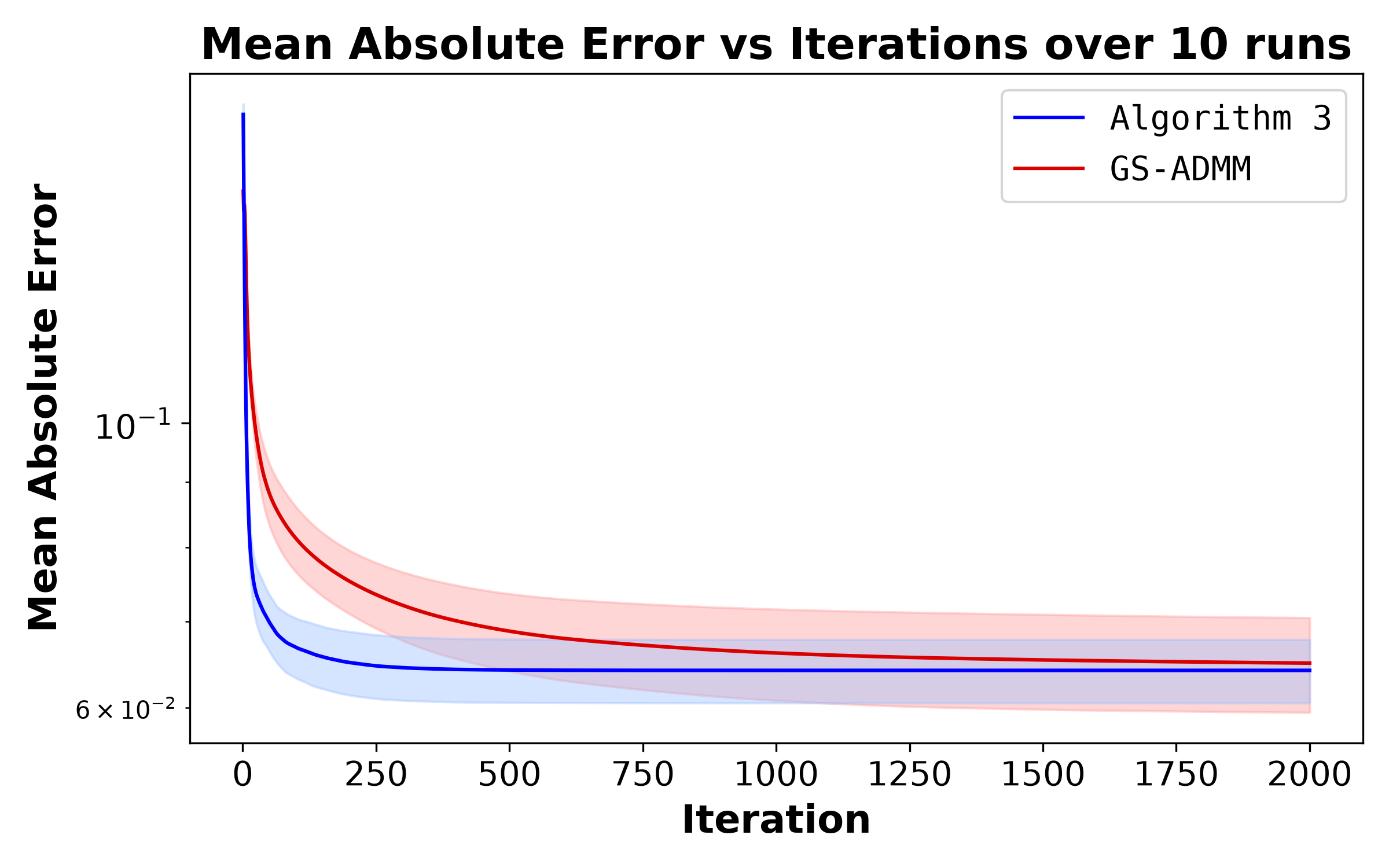}
		\\
		(a) Residual, $n=3000$
		&
		(b) MAE, $n=3000$
		\\[2mm]
		\includegraphics[scale=0.3]
		{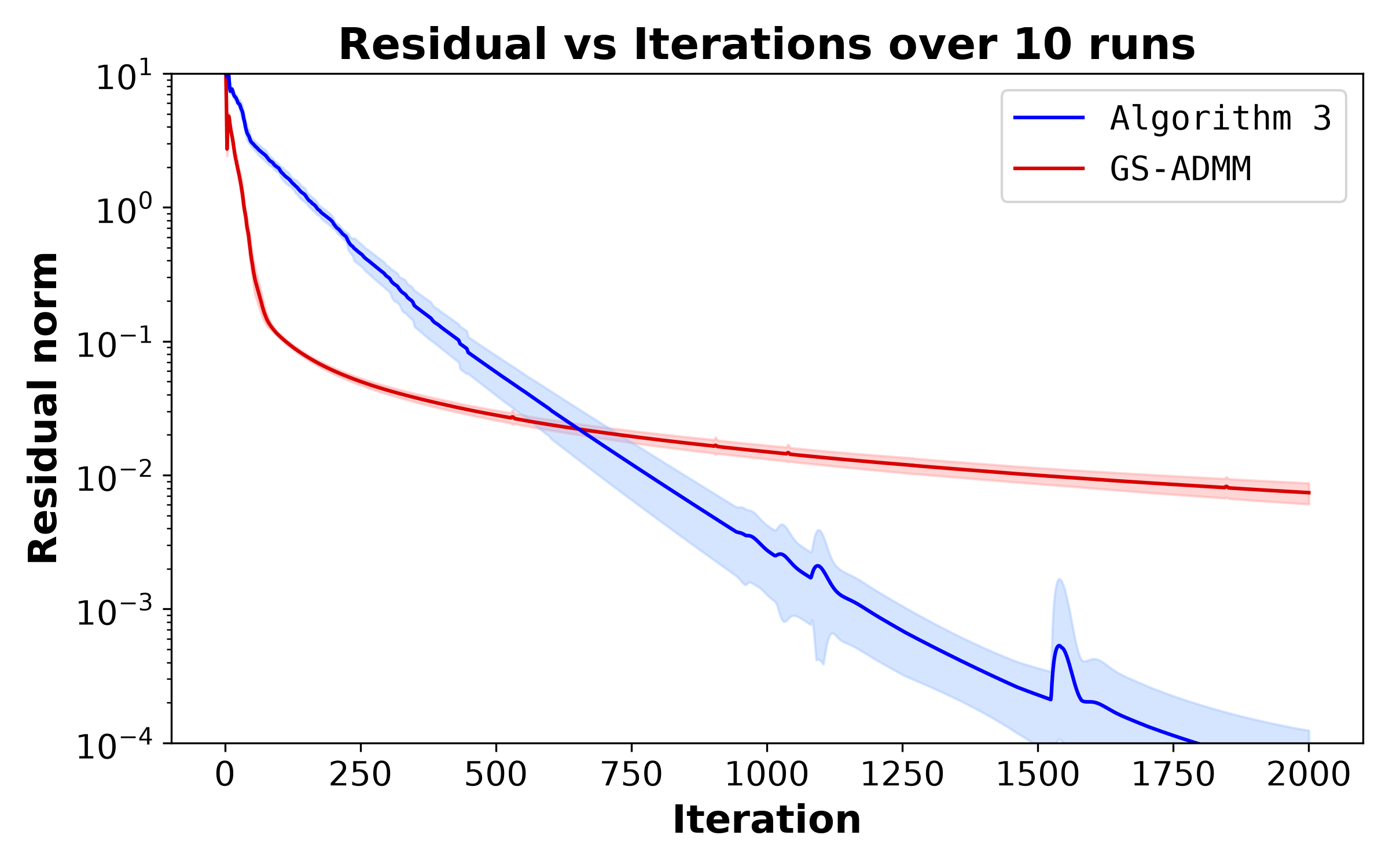}
		&
		\includegraphics[scale=0.3]
		{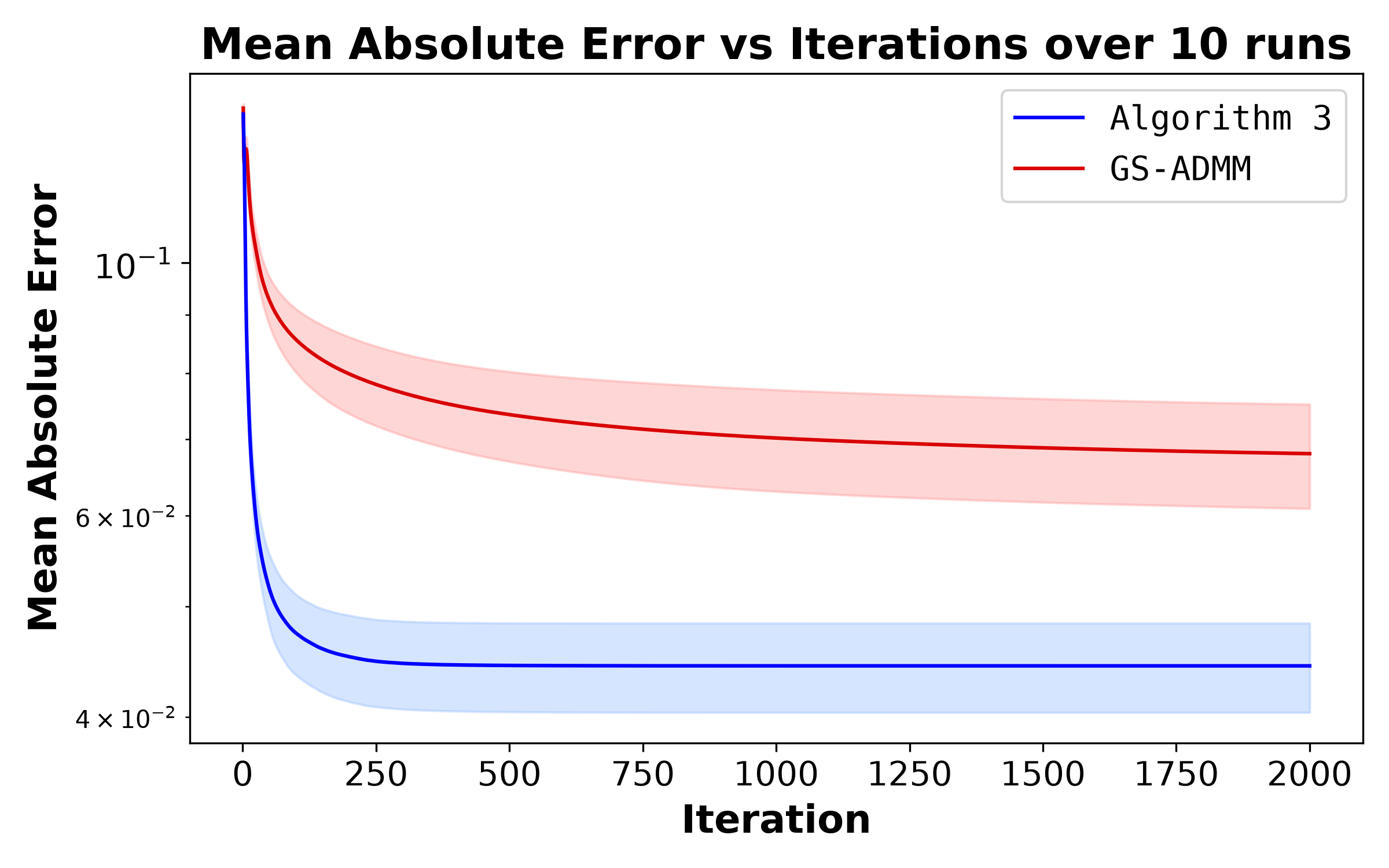}
		\\
		(c) Residual, $n=5000$
		&
		(d) MAE, $n=5000$
	\end{tabular}
	\caption{{\small Illustrative comparison of
			\cref{alg:admm_specialcase} and GS-ADMM over 10 independent runs for
			$n=3000$ and $n=5000$. The curves show the mean histories, while the
			shaded regions represent one standard deviation across runs.}}
	\label{fig:denoising_comparison}
\end{figure}

\subsubsection{\alert{Robust regression}}

\alert{We next consider the MCP-regularized least-absolute-deviation regression
	problem}
\begin{equation}
	\alert{
		\min_{z\in\Re^p}
		\frac{1}{n}\norm{Xz-c}_1
		+\frac{\rho}{2}\norm{z}^2
		+\omega P_\tau(z),
	}
	\label{eq:robust_regression}
\end{equation}
\alert{where $X\in\Re^{n\times p}$ is the design matrix, $c\in\Re^n$ is
	the response vector, $\rho>0$, and
	$P_\tau(z)\coloneqq\sum_{j=1}^p p_\tau(z_j)$, with $p_\tau$
	defined above. The least-absolute-deviation loss makes the model less sensitive to
	large residuals than a squared-loss formulation, while the MCP term promotes
	sparsity. Related nonsmooth sparse regression models have been studied, for example,
	in~\cite{BianChen2020}, where a smoothing proximal-gradient method is developed
	for convex nonsmooth regression with a cardinality penalty.}

\alert{To obtain a multiblock formulation, we partition the training samples
	into $N$ disjoint subsets $\mathcal I_1,\dots,\mathcal I_N$, and denote by
	$X_i$ and $c_i$ the corresponding rows of $X$ and $c$. We introduce coefficient
	copies $x_i,z_i\in\Re^p$ and residual variables
	$r_i\in\Re^{|\mathcal I_i|}$. The resulting problem is}
\begin{equation}
	\alert{
		\begin{aligned}
			\min_{\substack{x_i,z_i\in\Re^p\\
					r_i\in\Re^{|\mathcal I_i|}}}
			\quad&
			\sum_{i=1}^N\frac{\rho}{2N}\norm{x_i}^2
			+\frac{1}{n}\sum_{i=1}^N\norm{r_i}_1
			+\frac{\omega}{N}\sum_{i=1}^N P_\tau(z_i)
			+\delta_{\mathcal D_N}(z_1,\dots,z_N)
			\\
			\text{\rm s.t.}\quad&
			X_ix_i-r_i=c_i,
			\qquad
			x_i-z_i=0,
			\quad i=1,\dots,N,
		\end{aligned}
	}
	\label{eq:robust_regression_multiblock}
\end{equation}
\alert{where
	$\mathcal D_N=\{(z_1,\dots,z_N):z_1=\cdots=z_N\}$ and $\delta_{\mathcal D_N}$ is the indicator function on $\mathcal D_N$. On the feasible set,
	all coefficient copies coincide, and
	\eqref{eq:robust_regression_multiblock} reduces to
	\eqref{eq:robust_regression}. The formulation has $m=N+1$ blocks: the first $N$ block variables are
	$x_1,\dots,x_N$, with
	$f_i(x_i)=\frac{\rho}{2N}\norm{x_i}^2$, and hence the corresponding functions
	are $\rho/N$-strongly convex. The final block variable collects
	$(r_1,\dots,r_N,z_1,\dots,z_N)$, and its objective function is
	$\omega/(N\tau)$-weakly convex. Let
	$
	\H
	\coloneqq
	\left(\prod_{j=1}^N\Re^{|\mathcal I_j|}\right)\times(\Re^p)^N.
	$
	For $i=1,\dots,N$, define
	\[
	L_i x_i
	\coloneqq
	\bigl(
	(0,\dots,0,X_i x_i,0,\dots,0),
	(0,\dots,0,x_i,0,\dots,0)
	\bigr)\in\H,
	\]
	where the nonzero entries occupy the $i$th components. Moreover,
	$L_{N+1}=-\Id$ and
	$b=(c_1,\dots,c_N,0,\dots,0)\in\mathcal H$.}

\paragraph{\alert{Problem data.}}
\alert{We use the Wine Quality, Superconductivity, and YearPredictionMSD
	datasets from the UCI Machine Learning Repository
	\cite{Cortez2009WineQualityUCI,Hamidieh2018SuperconductivityUCI,
		BertinMahieux2011YearPredictionUCI}. For Wine Quality, we combine the red and
	white wine data and use a seeded $80/20$ training--test split. We use the same
	seeded split for Superconductivity. For YearPredictionMSD, we use the prescribed
	split consisting of the first $463715$ observations for training and the
	remaining $51630$ observations for testing.}

\begin{table}[tb!]
	\centering
	\caption{\alert{\small Dimensions of the robust-regression datasets after
			the training--test split.}}
	\begin{tabular}{lrrr}
		\hline
		\textbf{Dataset}
		& \textbf{Training}
		& \textbf{Test}
		& \textbf{Features}\\
		\hline
		Wine Quality
		& 5197
		& 1300
		& 11\\
		Superconductivity
		& 17010
		& 4253
		& 81\\
		YearPredictionMSD
		& 463715
		& 51630
		& 90\\
		\hline
	\end{tabular}
	\label{tab:robust_dataset_dimensions}
\end{table}

\alert{For Wine Quality and Superconductivity, the random split uses seed
	$2026$. For all datasets, the feature means and response mean are computed
	from the training data. Each retained feature column is centered and scaled
	to have unit Euclidean norm on the training set, and the same transformation
	is applied to the test set.}

\begin{figure}[tb!]
	\centering
	\begin{tabular}{@{}cc@{}}
		\includegraphics[scale=0.3]
		{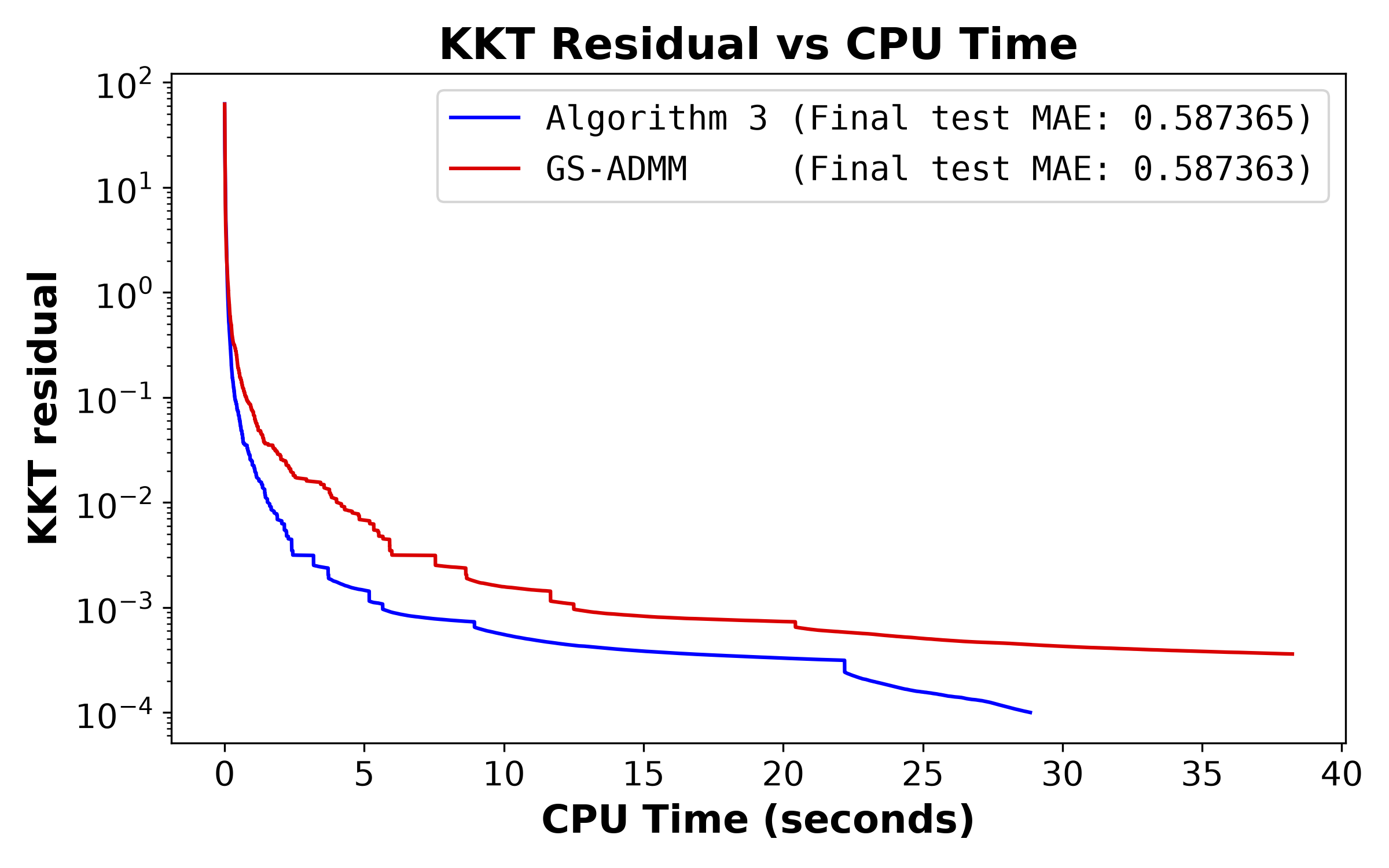}
		&
		\includegraphics[scale=0.3]
		{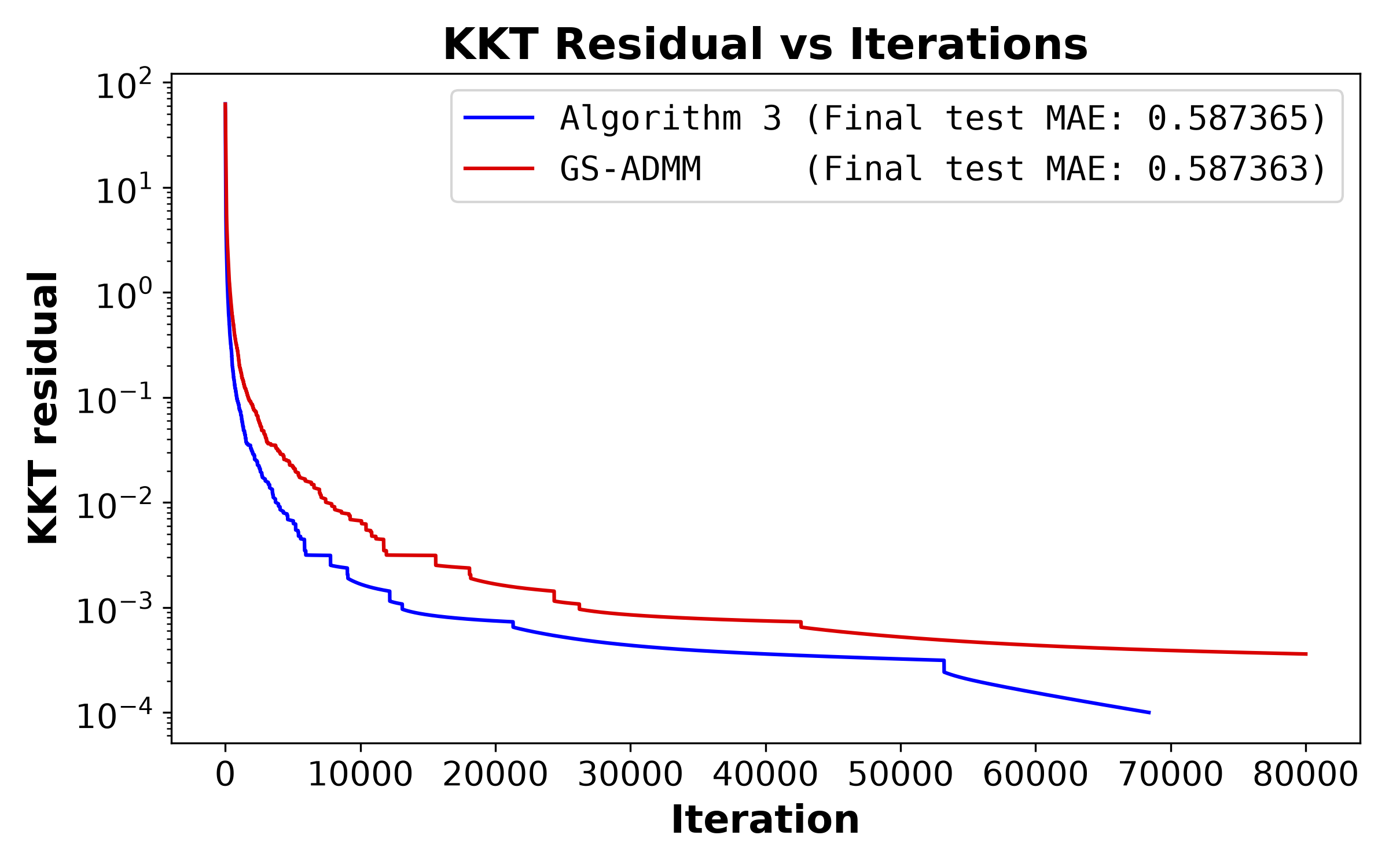}
		\\
		(a) Wine quality: Residual vs Time & (b) Wine quality: Residual vs Iteration \\
		& \\ 
		\includegraphics[scale=0.3]
		{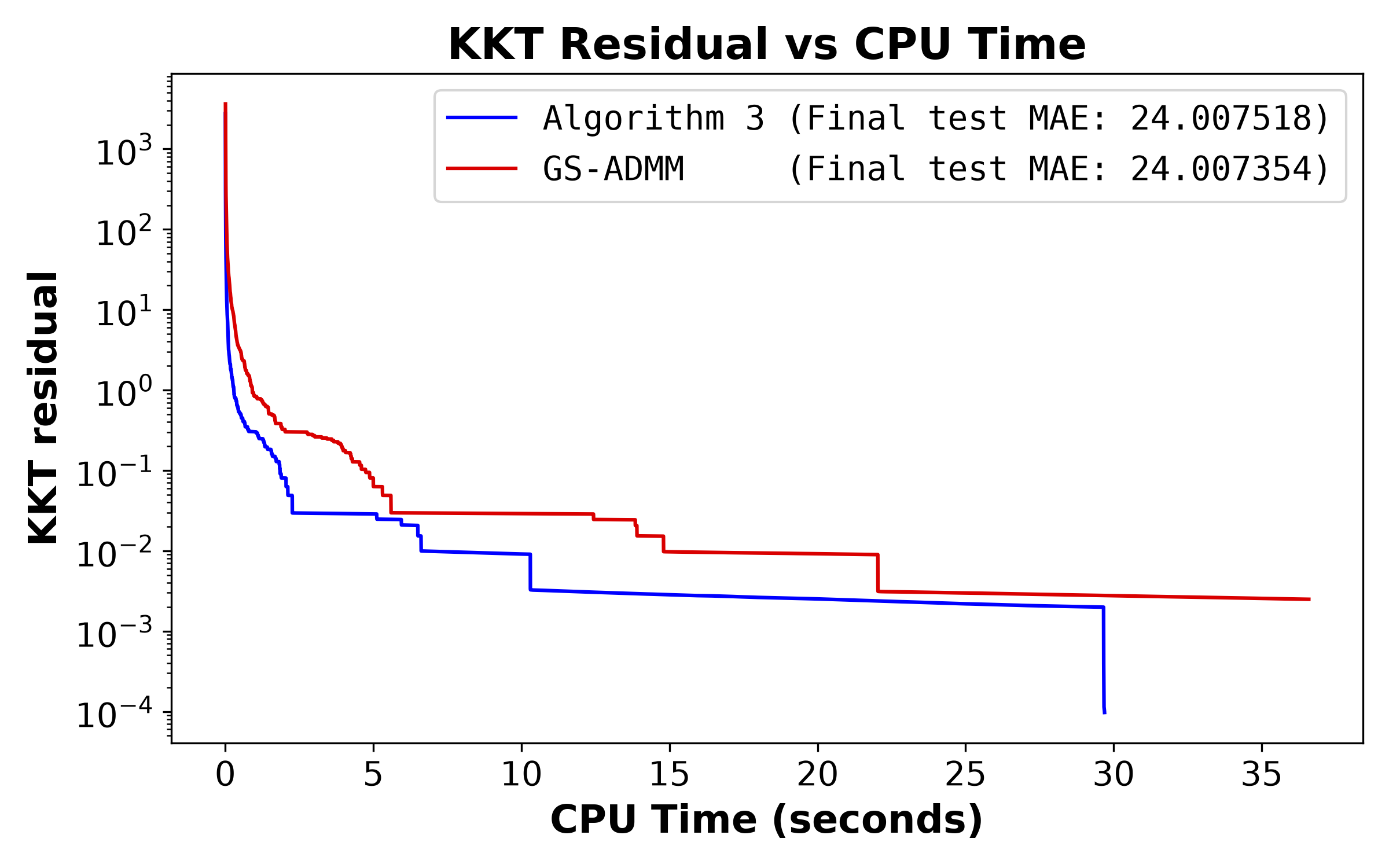}
		&
		\includegraphics[scale=0.3]
		{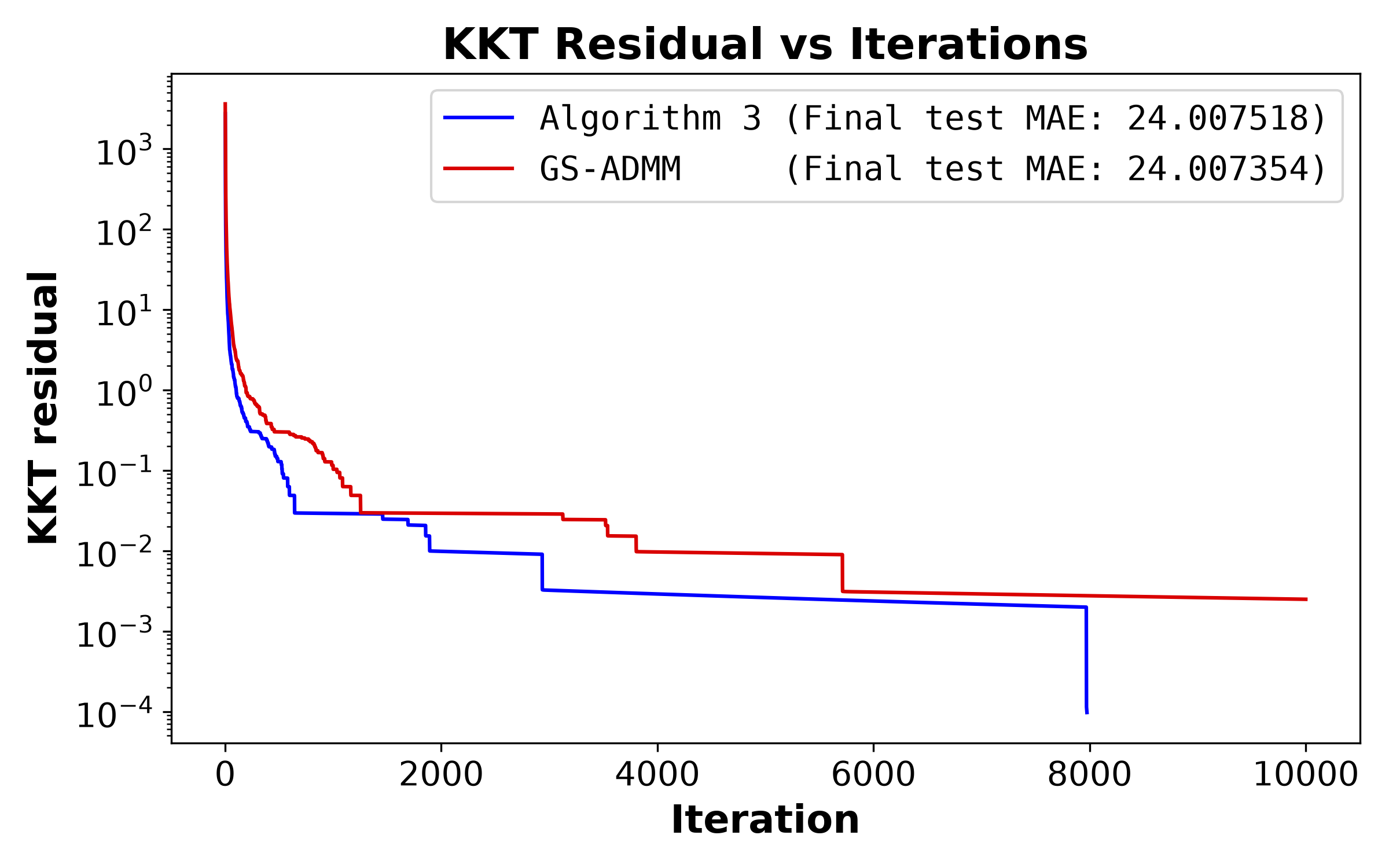}
		\\
		(c) Superconductivity: Residual vs Time & (d) Superconductivity: Residual vs Iteration \\
		& \\ 
		\includegraphics[scale=0.3]
		{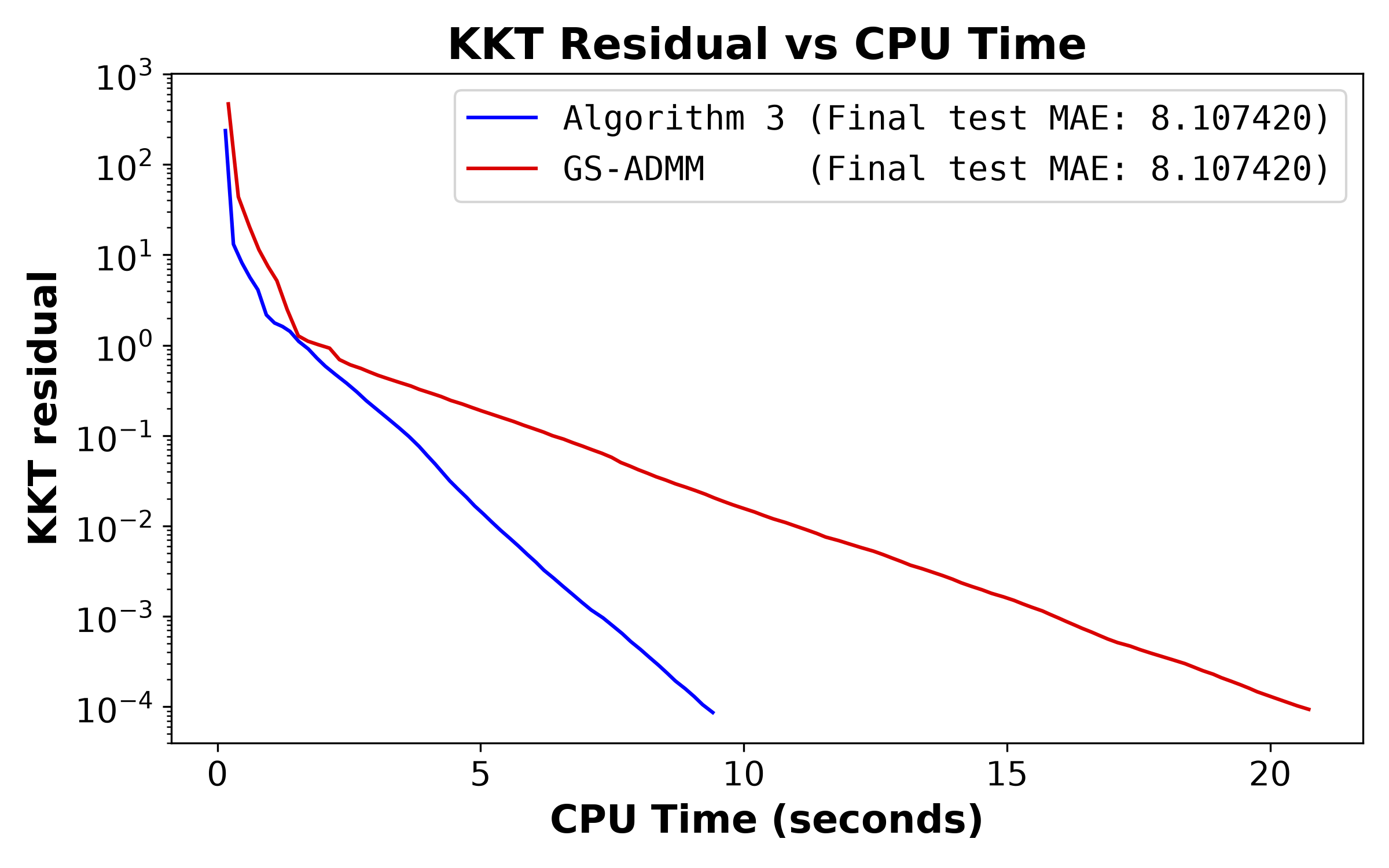}
		&
		\includegraphics[scale=0.3]
		{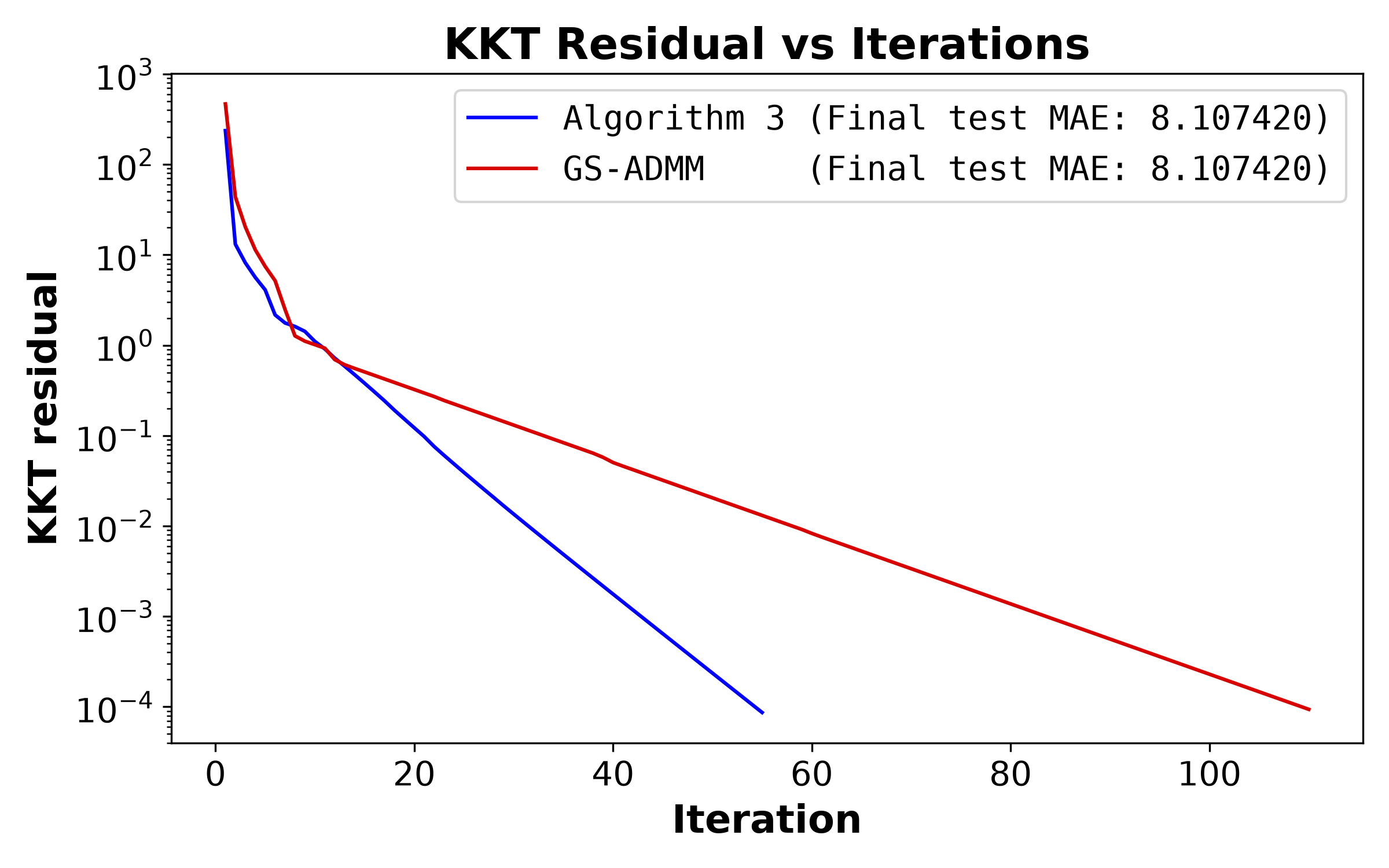} \\
		(e) YearPredictionMSD: Residual vs Time & (f) YearPredictionMSD: Residual vs Iteration \\
	\end{tabular}
	\caption{\alert{\small Residual plots of
			\cref{alg:admm_specialcase} and GS-ADMM for the robust-regression
			problems. The left column plots the KKT residual against elapsed time, while
			the right column plots the KKT residual against the iteration count.}}
	\label{fig:robust_regression_comparison}
\end{figure}

\paragraph{\alert{Parameter settings.}}
\alert{We use $N=3$, corresponding to $m=4$ blocks, and set
	$\rho=10^{-4}$. The training samples are partitioned into three blocks. We choose}
\alert{$
	\omega
	=
	0.05\lambda_{\max},$ with $
	\lambda_{\max}
	\coloneqq
	\frac{1}{n}
	\norm{X^\top\sign(-c)}_\infty.$
}

\alert{For the lifted formulation, each strongly convex block has modulus}
$
\alert{
	\sigma_i
	=
	\frac{\rho/N}{1+\norm{X_i}^2},
}
$
\alert{for $i=1,\dots,N$, and the last block has modulus $	\sigma_{N+1}
	=
	-\frac{\omega}{N\tau}$. Denoting
	$\alpha\coloneqq\min_{1\leq i\leq N}\sigma_i$, we set}
$
\alert{
	\tau
	\coloneqq
	\frac{1.3N\omega}{\alpha}. }$

\alert{This choice permits the equal selection
	$\theta=(N,\dots,N)\in\Theta$. We set $\delta=\eta\gamma$ with
	$\eta=2$ and choose}
\[
\alert{
	\beta
	\coloneqq
	N\sigma_{N+1}
	=
	-\frac{\omega}{\tau},
	\qquad
	\gamma
	\coloneqq
	\frac{\alpha-\beta}{\eta-1},
	\qquad
	\delta
	\coloneqq
	\eta\gamma.
}
\]
\alert{For GS-ADMM, we use the penalty
	$\gamma'=\gamma/N$. Both methods are initialized at zero and are terminated
	using \eqref{eq:stopcriterion_simplified} with $\epsilon=10^{-4}$, subject to
	a maximum of 10,000 iterations for the Superconductivity and YearPredictionMSD datasets and 80,000 iterations for Wine Quality.}

\paragraph{\alert{Results.}}
\alert{The residual plots for the three datasets are shown in
	\cref{fig:robust_regression_comparison}, both as functions of the iteration
	count and elapsed time. The figures illustrate the behavior of \cref{alg:admm_specialcase} on
	regression problems whose sizes range from relatively small to large. The legends also report the
	final test MAE obtained by each method. In these experiments, \cref{alg:admm_specialcase} reaches the
	prescribed residual tolerance with fewer iterations and less elapsed time
	than GS-ADMM, whereas GS-ADMM yields slightly lower test mean absolute error
	(MAE).}

\section{Conclusion}
\label{sec:conclusion}

We established the convergence of the doubly relaxed DR algorithm applied to a two-operator reformulation of the multioperator inclusion problem, in the case where all but one of the operators are strongly comonotone and the remaining operator is weakly comonotone. Previously, our work \cite{AlcantaraTakeda2025} investigated the DR algorithm for multioperator inclusions where the operators involved are weakly and strongly monotone. Hence, the present work contributes to the broader understanding of the DR algorithm in nonmonotone settings involving more than two operators, extending the theory to comonotone mappings. Related general multioperator
splitting frameworks have recently been developed in the monotone
setting; see, e.g., \cite{DaoTamTruong2025}. Since the completion of this work, we have
developed a more general splitting framework for composite semimonotone
inclusions in \cite{AlcantaraDaoTakeda2026}. In particular, that
framework provides a direct convergence analysis and recovers, as a
special case, the regime considered here in which exactly one operator
is weakly comonotone and
$\sum_{i=1}^m 1/\sigma_i<0$.

%
%

\setlength{\bibsep}{1pt} 	
\bibliography{bibfile}

\end{document}